%% file: AWB-FRH-RandomSurface.tex
\documentclass[10pt,reqno]{amsart}
\usepackage{amsmath}
\usepackage{amsfonts}
\usepackage{mathrsfs}
\usepackage{amssymb}
\usepackage{xypic}
\usepackage{amsthm}
\usepackage{url}
\usepackage{setspace}     \spacing{1}
\usepackage{marginnote}
\usepackage{enumitem}
 \usepackage[usenames,dvipsnames]{pstricks}
 \usepackage{epsfig}
 \usepackage{pst-grad} 
 \usepackage{pst-plot} 
 \usepackage[normalem]{ulem}

\newcommand\redout{\bgroup\markoverwith
{\textcolor{red}{\rule[.5ex]{2pt}{0.9pt}}}\ULon}
\usepackage{cancel}

\usepackage{xargs}
\usepackage[colorlinks=true]{hyperref} \hypersetup{urlcolor=blue, citecolor=blue,linkcolor=blue}
\usepackage{times}
\usepackage{mdframed}
\numberwithin{equation}{section}

\theoremstyle{plain}
\newtheorem{theorem}{Theorem}[section]
\newtheorem{lemma}[theorem]{Lemma}
\newtheorem{proposition}[theorem]{Proposition}
\newtheorem{prop}[theorem]{Proposition}
\newtheorem{corollary}[theorem]{Corollary}
\newtheorem{claim}[theorem]{Claim}

\theoremstyle{definition}
\newtheorem{definition}[theorem]{Definition}
\newtheorem{remark}[theorem]{Remark}

\theoremstyle{plain}

\newtheorem{innercustomprop}{Lemma}
\newenvironment{customprop}[1]
  {\renewcommand\theinnercustomprop{#1}\innercustomprop}
  {\endinnercustomprop}

\newlist{choices}{enumerate}{1} 
\setlist[choices]{label*=(\Alph*.),ref= {(\Alph*)},  resume}
\newlist{enumlemma}{enumerate}{3} 
\setlist[enumlemma]{label*={(\alph*)}, ref= {(\alph*)} }
\newlist{thislist}{enumerate}{1} 
\setlist[thislist]{itemindent= .25in, label*=(\roman*), resume}

\newcommand{\lyap}[1]{\left\vert\kern-0.25ex\left\vert\kern-0.25ex\left\vert  {#1} \right\vert\kern-0.25ex\right\vert\kern-0.25ex\right\vert}

\newcommand{\lyapno}[1]{\vert\kern-0.25ex\vert\kern-0.25ex\vert  {#1} \vert\kern-0.25ex\vert\kern-0.25ex\vert}

\newcommand{\rexp}[1]{{e}^{#1}}
\def\scrA{\mathscr A}
\def\calA{\mathcal A}

\DeclareMathOperator*{\Leb}{\mathrm{Leb}}
\def\leb{\Leb}
\def\lip{\mathrm{Lip}}

\newcommand{\diff}{\mathrm{Diff}}
\newcommand{\Diff}{\diff}

\newcommand{\Gl}{\mathrm{GL}}

\newcommand{\Sl}{\mathrm{SL}}
\newcommand{\SL}{\mathrm{SL}}

\newcommand{\id}{\mathrm{Id}}
\newcommand{\eps}{{\oldepsilon_0}}
\newcommand{\epone}{{\oldepsilon_1}}

\title[Measure rigidity for random dynamics on surfaces]{Measure rigidity for random dynamics on surfaces and related skew products}
\author[A.~Brown]{Aaron Brown}
\address{University of Chicago}
\email{awb@uchicago.edu}

\author[F.~Rodriguez~Hertz]{Federico Rodriguez Hertz}
   \address{Pennsylvania State University}
\email{hertz@math.psu.edu}

\keywords{Measure rigidity, nonuniform hyperbolicity, stiffness of stationary measures, random dynamics,  SRB measures} 

 \subjclass[2010]{Primary:  37C40, 37H99; Secondary: 37E30, 37D25, 28D15}

\long\def\symbolfootnote[#1]#2{\begingroup\def\thefootnote{\fnsymbol{footnote}}
\footnote[#1]{#2}\endgroup}

\def\I{\mathcal I}
\def\Exp{\mathbb E}
\def\Ex{\Exp}
\def\ae{a.e.\ }
\def\as{a.s.\ }

\newcounter{step}

\def\lip{\mathrm{Lip}}
\def\Lip{\lip}

\def\graph{\mathrm{graph}}

\begin{document}
\newcommand\Sigmaloc{\Sigma^{\pm}_{\text {loc}}}
\newcommand\Sigmalocu{\Sigma^{+}_{\text {loc}}}
\newcommand\Sigmalocs{\Sigma^{-}_{\text {loc}}}
\newcommand{\Fols}{\mathcal{W}^s}
\newcommand{\Folu}{\mathcal{W}^u}
\renewcommand{\L}{\mathcal{L}}
\newcommand{\Fol}{\mathcal{F}}
\newcommand{\Sal}{\mathcal{S}}

\newcommand{\aeq}{\circeq}

\newcommand{\cmt}[1]{{\color{red}{{#1}}}}

 \newcommand{\oldepsilon}{\mathchar"10F}
\renewcommand{\epsilon}{\varepsilon}
\renewcommand{\emptyset}{\varnothing}
\newcommand{\restrict}[2]{{#1}{\restriction_{{ #2}}}}
\newcommand{\restrictThm}[2]{{#1}\! \!  \restriction_{\! #2}}

\newcommand{\sm}{\smallsetminus}
\newcommand{\R}{\mathbb {R}}
\newcommand{\T}{\mathbb {T}}

\newcommand{\Q}{\mathcal {Q}}
\newcommand{\Z}{\mathbb {Z}}
\newcommand{\N}{\mathbb {N}}
\newcommand{\A}{\mathfrak {A}}
\newcommand{\B}{\mathcal  {B}}
\renewcommand{\P}{\mathcal{P}}

\newcommand{\inv}{^{-1}}
\newcommand{\Cone}{\mathcal C}

\newcommand{\note}[1]{\marginnote{{\color{red}\footnotesize \begin{spacing}{1}#1\end{spacing}}}}

\def\hol{\mathcal H}

\def\us{{u/s}}
\def\td {\tilde}

\newlength{\wideitemsep}
\setlength{\wideitemsep}{.5\itemsep}
\addtolength{\wideitemsep}{7pt}
\let\olditem\item
\renewcommand{\item}{\setlength{\itemsep}{\wideitemsep}\olditem}

\def\A{\mathcal A}
\def\F{\mathcal F}
\def\G{\mathcal G}	
	  
\def\E{\mathbb E}
	  
\def\MP{\mathcal{X}}
\renewcommand{\underbar}{\underline}
\renewcommand{\bar}{\overline}
\def\good{\mathscr G} 

\newcommand{\stab}[2]{W^s\!\left (#2, {#1}\right)}
\newcommand{\unst}[2]{W^u\!\left(#2,{#1}\right)}
\newcommand\locstab[3][r]{W^s_{ #1}\!\left( {#3},#2 \right)}
\newcommand\locunst[3][r]{W^u_{#1}\!\left({#3}, #2 \right)}

\def\calR{\mathcal R}
\def\calS{\mathcal S}

\newcommand{\stabM}[2]{W^s_{#2}\!\left({#1}\right)}
\newcommand{\unstM}[2]{W^u_{#2}\!\left({#1}\right)}
\newcommand\locstabM[3][r]{W^s_{#3,#1}\!\left(#2\right)}
\newcommand\locunstM[3][r]{W^u_{#3, #1}\!\left(#2\right)}

\newcommand{\stabp}[1]{W^s\!\left({#1}\right)}
\newcommand{\unstp}[1]{W^u\!\left({#1}\right)}
\newcommand\locstabp[2][r]{W^s_{#1}\!\left( #2\right)}
\newcommand\locunstp[2][r]{W^u_{#1}\!\left( #2\right)}

\newcommand{\Eu}[2]{E^u\left({#2, #1}\right)}
\newcommand{\Es}[2]{E^s\left({#2, #1}\right)}
\newcommand{\EuM}[2]{E^u_{#2}\left({ #1}\right)}
\newcommand{\EsM}[2]{E^u_{#2}\left({ #1}\right)}
\newcommand{\Eup}[1]{E^u\left({ #1}\right)}
\newcommand{\Esp}[1]{E^s\left({ #1}\right)}
\renewcommand{\E}{\mathcal E}

\def\Pp {\mathbb P}

\def\W{\mathcal W}
\def\scrF{\mathscr{F}}

\newcommand{\cocycle}[1][\xi]{%
  \def\ArgI{{#1}}%
  \BlahRelay
}
			\newcommand\BlahRelay[1][n]{%
			 %
			  f_{\ArgI} ^{#1}
			}

\def\nunaught{{\hat \nu}}
\def\munaught{{\hat \mu}}
\def\nuskew{{ \nu}}
\def\muskew{{ \mu}}
\def\nususp{{\boldsymbol  \nu}}
\def\mususp{{\omega}}
\def\mualt{{\hat \omega}}

\def\pt{\varsigma}

\maketitle
\begin{abstract}

Given a surface $M$ and a Borel probability measure $\nu$ on the group of $C^2$-diffeomorphisms of $M$  we study $\nu$-stationary probability measures on $M$. 
We prove for hyperbolic stationary measures the following trichotomy: either the stable distributions are non-random, the measure is SRB, or the measure is supported on a finite set and is hence almost-surely invariant.  
In the proof of the above results, we study skew products with surface fibers over a measure-preserving transformation  equipped with a decreasing sub-$\sigma$-algebra $\hat {\mathcal F}$ and derive a related result.   
A number of applications of our main theorem are presented.  



\end{abstract}

\section{Introduction}

Given an action of a one-parameter group on a manifold with some degree of hyperbolicity, there are typically many ergodic,
invariant measures with positive entropy. 
For instance, given an Anosov or Axiom A diffeomorphism of a compact manifold, the equilibrium states for H\"older-continuous potentials provide measures with the above properties \cite{MR0380889,MR0442989}.  
When passing to hyperbolic actions of larger groups, the following phenomenon has been demonstrated in many settings: the only invariant ergodic measures with positive entropy are absolutely continuous (with respect to the ambient Riemannian volume).  
For instance, consider the action of the semi-group $\N^2$ on the additive circle generated by
$$x\mapsto 2x \bmod 1\quad \quad x \mapsto 3x \bmod 1.$$
Rudolph showed for this action that the only invariant, ergodic probability measures are Lebesgue or have zero-entropy for every one-parameter subgroup \cite{MR1062766}.  
In \cite{MR1406432}, Katok and Spatzier generalized the above phenomenon to  actions of commuting toral automorphisms.  

Outside of the setting of affine actions,  Kalinin,  Katok, and Rodriguez Hertz, have recently demonstrated a version of abelian measure rigidity for nonuniformly hyperbolic, maximal-rank actions. 
In  \cite{MR2811602}, the authors consider $\Z^n$ acting by $C^{1+\alpha}$ diffeomorphisms on a $(n+1)$-dimensional manifold and prove that any  $\Z^n$-invariant measure $\mu$ is absolutely continuous assuming that 
at least one element of $\Z^n$ has positive entropy with respect to $\mu$ and that the Lyapunov exponent functionals are in \emph{general position}.  

For affine actions of non-abelian groups, a number of results have recently been obtained by  Benoist and Quint  in a series of papers  \cite{MR2831114,MR3037785,BQIII}.  For instance, consider a finitely supported measure $\nu$ on the group $\SL(n,\Z)$.  Let $\Gamma_\nu\subset \Sl(n,\Z)$ be the {(semi-)}group generated by the support of $\nu$. 
We note that $\Gamma_\nu$ acts naturally on the torus $\T^n$.  In \cite{MR2831114}, it is proved that if every finite-index subgroup of (the group generated by) $\Gamma_\nu$ acts irreducibly on $\R^n$ then every $\nu$-stationary probability measure on $\T^n$ is either finitely supported or is Haar; in particular every $\nu$-stationary probability measure is $\SL(n,\Z)$-invariant.  Similar results was obtained in \cite{MR2726604} through completely different methods.   
In \cite{MR2831114} the authors obtain similar stiffness results for groups of translations on quotients of simple Lie groups.  
More recently, in a breakthrough paper \cite{1302.3320}  Eskin and Mirzakhani consider the natural action of the upper triangular  subgroup $P\subset \Sl(2, \R)$ on  a stratum of abelian differentials on a surface.  They show that any such measures is  in fact $\Sl(2, \R)$-invariant and affine in the natural coordinates on the stratum.  Furthermore, for certain measures $\mu$ on $\Sl(2,\R)$, it is shown that all ergodic $\mu$-stationary measures  are  $\Sl(2, \R)$-invariant and affine.  


In this article, we prove a number of measure rigidity result for dynamics on surfaces.  We consider stationary measures for groups acting by diffeomorphisms on surfaces as well as skew products (or non-i.i.d.\ random dynamics) with surface dynamics in the fibers.  
All measures will be hyperbolic either by assumption or by entropy considerations.  
In this setup we prove for hyperbolic stationary measures the following trichotomy: either the stable distributions are non-random, the measure is SRB, or the measure is supported on a finite set and is hence almost-surely invariant.

In the case that $\nu$-\ae diffeomorphism preserves a common smooth measure $m$, we show for any non-atomic stationary  measure $\mu$ that  either there exists a $\nu$-almost-surely invariant $\mu$-measurable line field (corresponding to the stable distributions for \ae random composition) or the measure $\mu$ is $\nu$-almost-surely  invariant and coincides with an ergodic component of $m$. 

In the proof of the above results, we study skew products with surface fibers over a measure-preserving transformations equipped with an decreasing sub-$\sigma$-algebra $\hat \Fol$.  Given an invariant measure $\mu$ for the skew product whose fiber-wise conditional measures are non-atomic, we  assume the $\hat \Fol$-measurability of the `past dynamics' and the fiber-wise conditional measures and prove the following dichotomy: either the fiber-wise stable distributions are measurable with respect to a related decreasing sub-$\sigma$-algebra, or the measure $\mu$ is fiber-wise SRB. 


We focus here only on actions on surfaces  and  measures with non-zero exponents  though we expect the results to hold in more generality.   We rely heavily on the tools from the theory of nonuniformly hyperbolic diffeomorphisms used in \cite{MR2811602} and many ideas   developed in \cite{1302.3320} including a modified version (see  \cite[Section 16]{1302.3320}) of the ``exponential drift'' arguments from \cite{MR2831114}.   

\subsection*{Acknowledgement}
Both authors wish to express their gratitude to Alex Eskin who patiently and repeatedly explained many of the ideas  in \cite{1302.3320} to the authors.  His insights and encouragement allowed the authors to strengthen an earlier version of the paper (which  assumed positive entropy) and obtain the complete result obtained here.  
 A.~B.\ was   supported by an NSF postdoctoral research fellowship DMS-1104013.  F.~R.~H.\ was     supported    by NSF grants DMS-1201326 and DMS-1500947.  
 

\section{Preliminary definitions and constructions}
Let $M$ be a closed (compact, boundaryless) $C^\infty$ Riemannian manifold.  We write  $\diff^r(M)$ for the group of $C^r$-diffeomorphisms from $M$ to itself equipped with its natural $C^r$-topology.  
Fix $r= 2$ and consider a subgroup $\Gamma\subset \diff^2(M)$. 
We say a Borel probability measure $\mu$ on $M$ is \emph{$\Gamma$-invariant} if \begin{equation}\label{eq:inv} \mu(f\inv (A)) = \mu(A)\end{equation} for all Borel $A\subset M$ and all $f\in \Gamma$.

We note that for any continuous action by an amenable group on a compact metric space there always exists at least one invariant measure.  However, for  actions by non-amenable groups invariant measures need not exist.  
For this reason, we introduce   a weaker notion of invariance.  
Let $\nu$ be a Borel probability measure \emph{on the group} $\Gamma$.  
We say a Borel probability measure $\mu $ on $M$ is \emph{$\nu$-stationary} if 
	$$\int \mu (f\inv (A)) \ d \nu(f)= \mu (A)$$
for any Borel $A\subset M$.  	
By the compactness of $M$, it follows that  for any probability $\nu$ on $\Gamma$ there exists a $\nu$-stationary probability $\mu$ (e.g.\ \cite[Lemma I.2.2]{MR884892}.)

We note that if $\mu $ is $\Gamma$-invariant then $\mu $ is trivially $\nu$-stationary for any measure $\nu$ on  $\Gamma$.   
Given a $\nu$-stationary measure $\mu$ such that equality \eqref{eq:inv} holds for $\nu$-\ae $f\in \Gamma$, we say that $\mu$ is \emph{$\nu$-\as $\Gamma$-invariant.  }

Given a probability  $\nu$  on  $\diff^2(M)$ one defines the \emph{random walk} on the group of diffeomorphisms.  A  path in the random walk induces a sequence of diffeomorphisms from $M$ to itself.  
As in the case of a single transformation, we study the asymptotic ergodic properties of typical sequences of diffeomorphisms of $M$.  
We write $\Sigma_+ =   \left(\diff^2(M)\right)^\N$ for the space of sequences of diffeomorphisms
	$\omega= (f_0, f_1, f_2, \dots ) \in \Sigma_+.$ 
Given a Borel probability measure $\nu $ on $\diff^2(M)$, we equip $\Sigma_+$ with the product measure $\nu ^\N$. 
We remark that $\diff^2(M)$ is a Polish space, hence $\Sigma_+$ is Polish and the probability $\nu ^\N$ is Radon.
Let $\sigma\colon \Sigma_+\to \Sigma_+$ be the shift map 
$$\sigma\colon (f_0, f_1,f_2, \dots )\mapsto (f_1,f_2, \dots).$$
We have that $\nu ^\N$ is $\sigma$-invariant. 
Given a sequence $\omega= (f_0, f_1, f_2, \dots ) \in\Sigma_+$ and $n\ge 0$ we define a cocycle
	 $$\cocycle [\omega] [0]:= \id , \quad \cocycle  [\omega][\ ] = \cocycle  [\omega][1] := f_0,\quad \cocycle [\omega] [n] := f_{n-1} \circ f_{n-2} \circ \dots \circ f_1 \circ f_0.$$
We interpret $(\Sigma_+, \nu^\N)$ as a parametrization of all paths in the random walk defined by $\nu$.  
Following existing literature (\cite{MR968818}, \cite{MR1369243}), we denote by $\MP^+(M, \nu)$ the random dynamical system on $M$ defined by the random compositions $\{f^n_\omega\}_{\omega\in \Sigma_+}$.

Given a measure $\nu$ on $\diff^2(M)$ and a $\nu$-stationary measure $\mu$, we say a subset $A\subset M$ is \emph{$\MP^+(M, \nu)$-invariant} if for $\nu$-\ae $f$ and $\mu$-\ae $x\in M$
\begin{enumerate}
\item $x\in A \implies f(x)\in A$ and \item $x\in M\sm A \implies f(x) \in M\sm A.$\end{enumerate}
We say a $\nu$-stationary probability measure $\mu $ is \emph{ergodic} if, for every $\MP^+(M, \nu)$-invariant set $A$,
we have either $\mu (A) =0$ or $\mu (M\sm A) = 0$.  
We note that for a fixed $\nu$-stationary measure $\mu$  we have an ergodic decomposition of $\mu$ into ergodic, $\nu$-stationary measures  \cite[Proposition I.2.1]{MR884892}.

For a fixed $\nu$ and a fixed $\nu$-stationary probability $\mu $, one can define the $\mu$-metric entropy of the random process $\MP^+(M, \nu )$, written $h_\mu (\MP^+(M, \nu )).$  We refer to \cite{MR884892} for a definition.   

	\def\ICeq{
			 &\int \log ^+(|f|_{C^2}) + \log ^+(|f\inv|_{C^2}) \ d \nu  <\infty. \tag{$\ast$}\label{eq:IC2a} 
		} 
In the case that the support of $\nu$ is not bounded in $\diff^2(M)$, we  assume the integrability condition
\begin{align}\ICeq 
\end{align}
where $\log^+(a) = \max\{\log (x), 0\}$ and $|\cdot|_{C^2}$ denotes the $C^2$-norm.  
The integrability condition \eqref{eq:IC2a} implies the weaker  condition
\begin{equation}\int \log ^+(|f|_{C^1}) + \log ^+(|f\inv|_{C^1}) \ d \nu  <\infty \label{eq:IC1a}\end{equation} 
which guarantees  Oseledec's Multiplicative Ergodic theorem holds. 
The  $\log$-integrability of the $C^2$-norms is used later to apply tools from Pesin theory. 
\begin{proposition}[Random Oseledec's multiplicative theorem.]\label{prop:OMT}
Let $\nu$ be measure on $\diff^2(M)$ satisfying \eqref{eq:IC1a}.  Let $\mu $ be an ergodic, $\nu $-stationary probability.  

Then there are  numbers $-\infty<\lambda_1<\lambda_2<\dots<\lambda_\ell<\infty$, called \emph{Lyapunov exponents} such that for $\nu ^\N$-\ae sequence $\omega\in \Sigma_+$ and $\mu $-a.e. $x\in M$ there  is a filtration
\begin{equation}\label{eq:filtration}
\{0\} =V^0_\omega(x)\subsetneq V^1_\omega(x)\subset \dots \subsetneq V^\ell_\omega(x)= TM
\end{equation}
such that for $v\in V^k_\omega(x)\sm V^{k-1}_\omega(x)$
\[\lim _{n\to \infty} \dfrac 1 n \|D_x \cocycle  [\omega][n] (v)\| = \lambda _k.\]
\label{NEW}Moreover,  $m_i := \dim V^k_\omega(x) - \dim V^{k-1}_\omega(x)$ is constant \as and 
\begin{equation}\label{eq:jacobians}
\lim _{n\to \infty} \frac 1 n  \log | \det (D_x \cocycle )| = \sum _{i = 1}^\ell \lambda_i m_i.
\end{equation}
The subspaces $V^i_\omega(x)$ are invariant in the sense that $$D_xf_\omega V^k_\omega(x) = V^k_{\sigma(\omega)}(f_\omega(x)).$$
\end{proposition}
For a proof of the above theorem see, for example, \cite[Proposition I.3.1]{MR1369243}. 
We  write $$E^s_\omega(x): = \bigcup _{\lambda_j<0}V^j_\omega(x)$$ for the \emph{stable Lyapunov subspace} for the word $\omega$ at the point $x$.  

A stationary measure $\mu$ is \emph{hyperbolic} if no exponent $\lambda_i$ is   zero.  

We note that the random process $\MP^+(M, \nu )$ is not invertible.  Thus, while stable Lyapunov subspaces are defined for $\nu^\N$-\ae $\omega$ and $\mu$-\ae $x$, there are no well-defined unstable Lyapunov subspaces for $\MP^+(M, \nu )$.  
However, to state  results we will need a notion of SRB-measures (also called $u$-measures) for random sequences of diffeomorphisms. We will state the precise definition (Definition \ref{def:SRB}) in Section \ref{sec:condmeas} 
after introducing fiber-wise unstable manifolds for a related skew product construction.  
Roughly speaking, a $\nu$-stationary measure $\mu$ is SRB if it has absolutely continuous conditional measures along unstable manifolds.  Since we have not yet defined unstable manifolds (or subspaces), we postpone the formal definition 
and give here an equivalent property.  The following is an adaptation of \cite{MR819556}.  

\begin{prop}[{\cite[Theorem VI.1.1]{MR1369243}}]\label{prop:SRBrandom}
Let $M$ be a compact manifold and let $\nu$ be a probability on $\diff^2(M)$ satisfying \eqref{eq:IC2a}.
Then an ergodic, $\nu $-stationary probability $\mu $ is an SRB-measure if and only if 
\[h_\mu (\MP^+(M, \nu)) = \sum_{\lambda_i>0} m_i \lambda_i.\]
\end{prop}

We introduce some terminology for invariant measurable subbundles.
Given a subgroup $\Gamma\subset \diff^2(M)$, we have an induced the action of $\Gamma$ on sub-vector-bundles of the tangent bundle $TM$ via the differential.  
Consider $\nu$ supported on $\Gamma$ and a $\nu$-stationary Borel probability $\mu $ on $M$.  
\begin{enumerate}
\item 
We say a $\mu$-measurable subbundle $V\subset TM$ is  \emph{$\nu$-\as invariant} if \(Df (V(x)) = V(f(x))\)
 for  $\nu$-\ae $f\in \Gamma$ and  $\mu $-\ae $x\in M.$

\item A $(\nu^\N \times \mu) $-measurable family of subbundles $(\omega,x) \mapsto V_\omega(x)\subset T_xM$ is \emph{$\MP^+(\Gamma,\nu )$-invariant} if  for $(\nu ^\N\times \mu )$-\ae $(\omega,x)$
			\[D_xf_\omega V_\omega(x)  = V_{\sigma(\omega)}(f_\omega(x)).\]
Note that subbundles in the filtration \eqref{eq:filtration} are $\MP^+(\Gamma,\nu )$-invariant.

\item  We say a $\MP^+(M,\nu )$-invariant family of subspaces $V_\omega(x)\subset TM$ is \emph{non-random} if  there exists a  $\nu$-\as invariant $\mu $-measurable subbundle $\hat V\subset TM$ with $\hat V(x)= V_\omega(x)$  for $(\nu ^\N \times \mu)$-\ae $(\omega,x)$.
\end{enumerate}

\section{Statement of results: groups of surface diffeomorphisms} 
\label{sec:resultsRDS}
For all results in this paper, we restrict ourselves to the case that $M$ is a closed surface.  
Equip  $M$ with a background Riemannian metric. 

\label{sec:resultsStationary}
Let $\nunaught$ be a  Borel probability on the group $\diff^2(M)$ satisfying the integrability hypotheses \eqref{eq:IC2a}. 
Let $\munaught$ be an ergodic $\nunaught$-stationary measure on $M$.  
At times, we may  assume 
$ h_\munaught(\MP^+(M, \nunaught))>0. $  
By the fiber-wise Margulis--Ruelle inequality  \cite{MR1314494} 
applied to the associated skew product (see Section \ref{sec:skewRDS}), 
positivity of entropy implies
that the Oseledec's filtration \eqref{eq:filtration} is nontrivial and the exponents satisfy \begin{align}-\infty<\lambda_1<0<\lambda_2<\infty.\label{eq:pozzexp}\end{align}
In particular, the {stable Lyapunov subspace} $E^s_\omega(x)$ corresponds to the  subspace $V^1_\omega(x)$ in \eqref{eq:filtration} and is 1-dimensional.

We state our first main theorem.
\begin{theorem}\label{thm:1+}
\label{thm:main}
Let $M$ be a closed  surface and let  $\nunaught$ be a Borel probability measure on $\diff^2(M)$ 
satisfying \eqref{eq:IC2a}. Let $\munaught$ be an ergodic, hyperbolic, $\nunaught$-stationary Borel probability measure on $M$.
Then either
\begin{enumerate}
\item the stable distribution $E^s_\omega(x)$ is non-random,
\item $\munaught$ is finitely supported, and hence $\nunaught$-\as invariant,  or 
\item $\munaught$ is SRB.  
\end{enumerate}
\end{theorem}

By the above discussion and standard facts about entropy, 
if $h_\munaught(\MP^+(M,\nunaught))>0$ then $\munaught$ is hyperbolic and has no atoms.   We thus obtain as a corollary the following  dichotomy for positive-entropy stationary measures.
\begin{corollary}\label{thm:1}
Let $M$ be a closed  surface. 
Let  $\nunaught$ be a Borel probability measure on $\diff^2(M)$ 
satisfying \eqref{eq:IC2a} and let $\munaught$ be an ergodic, $\nunaught$-stationary Borel probability measure on $M$ with $h_\munaught(\MP^+(M, \nunaught))>0.$ 
Then either
\begin{enumerate}
\item the stable distribution $E^s_\omega(x)$ is non-random, or \item $\munaught$ is SRB.  
\end{enumerate}
\end{corollary}


We also immediately obtain from Theorem \ref{thm:1+} the following corollary.
\begin{corollary}
Let  $\nunaught$  be as in Theorem \ref{thm:1+} with $\munaught$ an ergodic, hyperbolic,  $\nunaught$-stationary probability measure.   Assume that $\munaught$ has one exponent of each sign and that there are no $\nunaught$-\as invariant, $\munaught$-measurable line fields on $TM$.  Then either $\munaught$ is SRB or $\munaught$ is finitely supported.
\end{corollary}

We note that in \cite{MR2831114}, the authors prove an analogous statement.  Namely, for homogeneous actions satisfying certain hypotheses, any non-atomic stationary measure $\munaught$ is shown to be absolutely continuous along \emph{some} unstable (unipotent) direction.   
Using Ratner Theory, one concludes that the stationary measure $\munaught$ is thus the Haar measure and hence \emph{invariant} for every element of the action.  
In non-homogeneous settings, such as the one considered here and the one considered in \cite{1302.3320}, there is no analogue of Ratner Theory.  Thus, in such settings 
more structure is needed in order to promote the SRB property to absolute continuity or almost-sure invariance of the stationary measure $\munaught$.   
The next theorem demonstrates that this promotion is possible assuming the existence of an almost-surely invariant volume.   
\begin{theorem}\label{thm:3} 
Let  $\Gamma\subset \diff^2(M)$ be a subgroup and assume $\Gamma$ preserves a probability measure $m$ equivalent to the Riemannian volume on $M$.
Let $\nunaught$ be a probability measure on $\diff^2(M)$ with $\nunaught(\Gamma) = 1$ and satisfying \eqref{eq:IC2a}.  Let $\munaught$ be an ergodic, hyperbolic,  $\nunaught$-stationary Borel probability measure.  Then either 
\begin{enumerate}
\item $\munaught$ has finite support, 
 \item 
the stable distribution $E^s_\omega(x)$ is non-random, or  
\item\label{case3} $\munaught$ is absolutely continuous and  is $\nunaught$-\as $\Gamma$-invariant.
\end{enumerate}

Furthermore, in conclusion \ref{case3}, we will have that $\munaught$ is---up to normalization---the restriction of $m$ to a positive volume subset. 
\end{theorem} 
In particular,  in Theorem \ref{thm:3} if the stable distribution $E^s_\omega(x)$ is not non-random, then we have the following stiffness result.
\begin{corollary} 
Let  $m$ be a  probability measure on $M$ equivalent to the Riemannian volume.  Let $\nunaught$ be a probability measure on $ \diff^2(M)$ satisfying \eqref{eq:IC2a} and such that $m$ is $\nunaught$-\as invariant.  Let $\munaught$ be an ergodic, hyperbolic,  $\nunaught$-stationary Borel probability measure.  Assume there are no $\munaught$-measurable, $\nunaught$-\as invariant line fields on $TM$.  Then $\munaught$ is invariant under $\nunaught$-\ae $f\in \diff^2(M)$.  
\end{corollary}

\section{General skew products}
In this section, we reformulate  the results stated in Section \ref{sec:resultsRDS} in terms of  results about related skew product systems.   This allows us to convert the dynamical properties of non-invertible, random dynamics, to properties of one-parameter invertible actions and to exploit tools from the theory of  nonuniformly hyperbolic diffeomorphisms.     A result for a more general skew product systems is also introduced.  
\subsection{Canonical skew product associated to a random dynamical system}\label{sec:skewRDS}  Let $M$ and $ \nunaught$ be as in Section \ref{sec:resultsStationary}. Consider the product space $\Sigma_+\times M$ and define the (non-invertible) skew product
  $\hat F \colon \Sigma_+\times M \to \Sigma_+ \times M$ by $$\hat F\colon (\omega, x) \mapsto (\sigma(\omega), f_\omega(x)).$$
We have the following reinterpretation of $\nunaught $-stationary measures.  
\begin{proposition}\cite[Lemma I.2.3, Theorem I.2.1]{MR884892} \label{prop:charStatmeas}
For  a Borel probability measure $\munaught $ on $M$ we have that
\begin{enumerate}
\item $\munaught$ is $\nunaught$-stationary if and only if $\nunaught^\N \times \munaught$ is $\hat F$-invariant;
\item a $\nunaught$-stationary measure $\munaught$ is ergodic for $\MP^+(M, \nunaught)$ if and only if $\nunaught^\N \times \munaught$ is ergodic for $\hat F$.  
\end{enumerate}
\end{proposition}


\def\fthere{Writing the cocycle as $f^n_\xi$ is standard in the literature but is  somewhat ambiguous.   We write 
$(f_{\xi})\inv$ to indicate the diffeomorphism that is the inverse of $f_\xi\colon M\to M$.  The symbol $f_\xi\inv$ indicates $(f_{\theta\inv(\xi)})\inv$.}

Let $\Sigma:= (\diff^r(M))^\Z$ be the space of bi-infinite sequences and equip $\Sigma$ with the product measure  
	$\nunaught^\Z.$
We again write  $\sigma\colon \Sigma\to \Sigma$ for the left shift  
	$ (\sigma(\xi))_i = \xi_{i+1}  .$
 Given $$\xi= (\dots, f_{-2}, f_{-1}, f_0, f_1,f_2, \dots ) \in \Sigma$$ define $f_\xi := f_0$ and define
the (invertible) skew product $F \colon \Sigma \times M\to \Sigma \times M$
by \begin{equation}\label{eq:skewdefn}F\colon (\xi, x) \mapsto (\sigma(\xi), f_\xi(x)).\end{equation}

We have the following proposition  producing the measure whose properties we will study for the remainder.  
 \begin{proposition} \label{prop:mudef} Let $\munaught$ be a $\nunaught$-stationary Borel probably measure.
 There is a unique $F$-invariant Borel probability measure $\muskew$ on $\Sigma \times M$
whose image under the canonical  projection $\Sigma\times M\to \Sigma_+\times M$  is  $\nunaught^\N\times \munaught$.

Furthermore, $\muskew$ projects to $\nunaught^\Z$ and $\munaught$, respectively, under the canonical projections $\Sigma\times M\to \Sigma$ and $\Sigma\times M\to M$ and 
is equal to the weak-$*$ limit  \begin{equation}\label{eq:mulim}\muskew = \lim_{n\to \infty} (F^n)_*(\nunaught^\Z\times \munaught).\end{equation}
\end{proposition}
See for example {\cite[Proposition I.1.2]{MR1369243}} for a proof of the proposition in this setting.  
Let $\{\mu_\xi\}_\xi\in \Sigma$ be a family of conditional measures of $\mu$ relative to the partition into fibers of $\Sigma \times M\to \Sigma$.  By a slight abuse of notation, consider $\mu_\xi$ as a measure on $M$ for each $\xi$.  It follows that for $\nunaught^\Z$-\ae $\xi\in \Sigma$ and 
$\eta\in \Sigma$ with $  \eta_i = \xi_i \text{ for all } i< 0$ that $\mu_\eta= \mu_\xi$.  

Write  $\pi \colon \Sigma\times M \to \Sigma$ for the canonical projection.  We write  $h_{\muskew} (F \mid\pi)$ for the conditional metric entropy of $(F, \muskew)$ conditioned on the sub-$\sigma$-algebra generated by $\pi\inv$.  
\begin{proposition}[{\cite[Theorem II.1.4]{MR884892}, \cite[Theorem I.2.3]{MR1369243}}]
We have the equality of entropies
\label{prop:entropiessame}
$h_\munaught(\MP^+(M, \nunaught))= h_{\muskew} (F \mid\pi).$
\end{proposition}

\subsection{General skew products}\label{sec:absSkew}
We give a generalization of the setup introduced in Section \ref{sec:skewRDS}.  
Let $(\Omega, \B_\Omega, \nuskew)$ be a Polish probability space; that is, $\Omega$ has the topology of  a complete separable metric space, $\nuskew$ is a Borel probability measure, and $\B_\Omega$ is {the $\nuskew$-completion} of the Borel $\sigma$-algebra.   
Let $\theta\colon (\Omega, \B_\Omega, \nuskew)\to (\Omega, \B_\Omega, \nuskew)$ be an invertible, ergodic, measure-preserving transformation. 
Let $M$ be a closed $C^\infty$ manifold.  Fix a background $C^\infty$ Riemannian metric on $M$ and write $\|\cdot\|$ for the norm on the tangent bundle $TM$ and $d(\cdot, \cdot)$ for the induced distance on $M$.  We note that compactness of $M$ guarantees all metrics are equivalent, whence all dynamical objects structures defined below are independent of the choice of metric.

\def\fthere{Writing the cocycle as $f^n_\xi$ is standard in the literature but is  somewhat ambiguous.   We write 
$(f_{\xi})\inv$ to indicate the diffeomorphism that is the inverse of $f_\xi\colon M\to M$.  The symbol $f_\xi\inv$ indicates $(f_{\theta\inv(\xi)})\inv$.}

We  consider  a $\nuskew$-measurable mapping $\Omega\ni\xi\mapsto f_\xi\in \diff^2(M)$.  
Define\footnote{\fthere} a cocycle $\scrF\colon\Omega \times \Z\to \diff^r(M)$ over $\theta$, written $\scrF \colon (\xi, n) \mapsto \cocycle$,  by 
\begin{enumerate}
	\item $\cocycle    [\xi] [0]:= \id$, $\cocycle    [\xi] [1] := f_\xi$,
	\item $\cocycle    [\xi] [n]:= f_{\theta^{n-1}(\xi)} \circ \dots \circ f_{\theta(\xi)}\circ f_\xi$ for $ n>0$, and
	\item $\cocycle    [\xi] [n] := (f_{\theta^{n}(\xi)})\inv \circ \dots \circ (f_{\theta\inv(\xi)})\inv = (f_{\theta^{n}(\xi)}^{|n|}) \inv $ for $ n<0.$

\end{enumerate}
As above, we will always assume the following integrability condition 
\begin{align}
&\int \log ^+(|f_\xi|_{C^2}) + \log ^+(|f\inv_\xi|_{C^2}) \ d \nuskew(\xi) <\infty. \tag{IC}\label{eq:IC2} 
\end{align}

Write $X:= \Omega\times M$ with canonical projection $\pi\colon X\to \Omega$.   For $\xi \in \Omega$, we  write $$M_\xi:= \{\xi\}\times M = \pi\inv (\xi)$$ for the fiber of $X$ over $\xi$.  
On $X$, we define the skew product $F\colon X \to X$  
$F\colon(\xi,x)\mapsto(\theta(\xi), f_\xi(x)).$
 
Note that $X= \Omega\times M$ has a natural Borel structure.  The main object of study for the  remainder will be $F$-invariant Borel probability measures on $X$ with marginal $\nuskew$.  
 \begin{definition}
A probability  measure $\muskew$ on $X$ is called 
\emph{$\scrF$-invariant} 
if it is $F$-invariant and  satisfies $$\pi_*\muskew = \nuskew.$$  
Such a measure $\muskew$ is said to be \emph{ergodic} if it is $F$-ergodic.  
\end{definition}

Let $\{\muskew_\xi\}_{\xi\in \Omega}$ denote the family of conditional probability measures with respect to the partition induced by the projection $\pi\colon X\to \Omega$.  Using the canonical identification of fibers $M_\xi= \{\xi\}\times M$ in $X$ with $M$, by an abuse of notation we consider the map $\xi\mapsto \muskew_\xi$ as a measurable map from $\Omega$ to the space of Borel probabilities on $M$.

\subsubsection{Fiber-wise Lyapunov exponents}
 We define $TX$ to be the fiber-wise tangent bundle 
 $$TX:= \Omega\times TM$$ and $DF\colon TX\to TX$ to be the fiber-wise differential
 $$DF \colon (\xi, (x,v)) \mapsto (\theta(\xi), (f_\xi (x), D_xf_\xi v)).$$

Let $\muskew$ be an ergodic, $\scrF$-invariant probability.  We have that $DF$ defines a linear cocycle over the (invertible) measure-preserving system $F\colon (X, \muskew)\to (X, \muskew)$.  
By the integrability condition \eqref{eq:IC2}, we can apply Oseledec's Theorem to $DF$ to obtain a $\muskew$-measurable splitting
\begin{equation}\label{eq:OscSplitting}T_{(\xi,x)}X:= \{\xi\}\times T_xM  = \bigoplus_j E^j(\xi,x)\end{equation} and numbers $\lambda_{\muskew}^j$ so that  
for $\muskew$-a.e. $(\xi,x)$, and every $v\in E^j(\xi,x)\sm\{0\}$
$$\lim_{n\to \pm \infty} \dfrac 1 n \log  \|DF^n( v)\|= \lim_{n\to \pm \infty} \dfrac 1 n \log  \|D_x \cocycle    [\xi] [n] v\| = \lambda^j_{\muskew}.$$

It follows from standard arguments that if the fiber-wise exponents of $DF$ are all positive (or negative) then the fiber-wise conditional measure $\mu_\xi$ are purely atomic.

\subsection{Reformulation of Theorem \ref{thm:main}}\label{sec:skewreint}
Let   $M$ and $ \nunaught$ be as in Section \ref{sec:resultsStationary} and let $ \munaught$ be an ergodic, hyperbolic, $ \nunaught$-stationary measure.  
  Let $F\colon \Sigma \times M\to \Sigma\times M$ denote the canonical skew product and let $\muskew$ be the measure given by Proposition \ref{prop:mudef}.  
We have a $\muskew$-measurable splitting of $\Sigma\times TM$ into measurable bundles $$
\{\xi\}\times T_xM= E^s{(\xi,x)} \oplus E^u{(\xi,x)} .$$
Note that, a priori, one of the bundles $E^s{(\xi,x)}$ or $E^u{(\xi,x)}$ might be trivial; however by Remark \ref{rem:samesign} below, Theorem \ref{thm:main} follows trivially in these cases.

For $\sigma\in \{s,u\}$ and $(\xi,x) \in \Sigma\times M$  we write $E^\sigma_\xi(x)\subset TM$ for the subspace with $E^\sigma{(\xi,x)} = \{\xi\} \times E^\sigma_\xi(x)$.  Projectivizing the tangent bundle $TM$, we obtain a measurable function  $$(\xi,x)\mapsto E^\sigma_\xi(x).$$

\newcommand\Gol{\mathcal G}

For $\xi = (\dots, \xi_{-2}, \xi_{-1}, \xi_0, \xi_1, \xi_2,\dots ) \in \Sigma$ write $\Sigmalocs(\xi)$ and $\Sigmalocu(\xi)$  for the \emph{local stable} and \emph{unstable sets}
$$\Sigmalocs(\xi): = \{ \eta\in \Sigma \mid \eta_i = \xi_i \text{ for all } i\ge 0\}$$
$$\Sigmalocu(\xi): = \{ \eta\in \Sigma \mid \eta_i = \xi_i \text{ for all } i< 0\}.$$

Write $\hat \Fol$ for  (the completion of) the Borel sub-$\sigma$-algebra of $\Sigma$ containing sets that are \as saturated by local unstable sets:
$C\in \hat \Fol$ if and only if $C= \hat C \mod \nunaught^\Z$ where $\hat C$ is Borel in $\Sigma$ with $$\hat C  = \bigcup_{\xi\in \hat C }\Sigmalocu(\xi).$$
Similarly, we define $\hat \Gol$ to be the sub-$\sigma$-algebra of $\Sigma$ whose atoms are local stable sets.  
Writing $\B_M$ for the Borel $\sigma$-algebra on $M$ we define sub-$\sigma$-algebras on $\Fol$ and $\Gol$ on $X$ to be, respectively,  the $\muskew$-completions of the $\sigma$-algebras $\hat \Fol \otimes \B_M$ and $\hat \Gol \otimes \B_M$.  

We note that, by construction, the assignments $\Omega\to \diff^2(M)$ given by $\xi \mapsto f_\xi$  and $\xi\mapsto \cocycle  [\xi][-1]$ are, respectively,  $\hat \Gol$- and $\hat \Fol$-measurable.  Furthermore, observing that the stable line fields $E^s_\xi(x)$ depend only on the value of $\cocycle$ for $n\ge 0$, we have the following  straightforward but crucial observation. 
\begin{proposition}\label{prop:measurability} 
The map  
$(\xi,x)\mapsto E^s_\xi(x)$ is $\Gol$-measurable and  the map 
$(\xi,x)\mapsto E^u_\xi(x)$ is $\Fol$-measurable.  
\end{proposition}
We have the following claim, which follows from the explicit construction of $\muskew$ in \eqref{eq:mulim}.  
\begin{proposition} \label{prop:hopf}
The intersection $\Fol\cap \Gol$ is equivalent modulo $  \muskew$ to the $\sigma$-algebra $\{\emptyset, \Sigma\}\otimes \B_M.$
\end{proposition}
\begin{proof}
Let $A\in \Fol \cap  \Gol$.  Since $A\in \Gol$, we have that $A \circeq \hat A$ where $\hat A$ is a Borel subset of $\Sigma\times M$ such that for any $(\xi,y)\in \hat A $ and $\eta \in \Sigmalocs(\xi)$, $$(\eta,y) \in \hat A.$$   

We write $\{\muskew^\Fol_{(\xi,x)}\}$ and  $\{\muskew^\Sigma_{(\xi,x)}\}$, respectively, for  families of conditional probabilities given by  the partition of $\Sigma\times M$ into atoms of $\Fol$  and the partition $\{ \Sigma \times \{x\} \mid x\in M\}$ of $\Sigma\times M$. 
It follows from the construction of $\muskew$ given by  \eqref{eq:mulim} that $\muskew^\Fol_{(\xi,x)}$ may be taken to be the form 
\begin{align}\label{eq:thisone} d\muskew^\Fol_{(\xi,x)}(\eta,y) = d\nunaught ^\N(\eta_0, \eta_1,  \dots) \delta_x(y) \delta _{(\xi_{-1})} (\eta_{-1})
 \delta _{(\xi_{-2})} (\eta_{-2})\dots \end{align}
for every $(\xi,x) \in X$.   

Since $A\in \Fol$ we have $\hat A \in \Fol$.  Thus, for $\muskew$-\ae $(\xi,x)\in \hat A$, $$\muskew^\Fol_{(\xi,x)}(\hat A)= 1.$$
Furthermore, it follows from \eqref{eq:thisone} and the form of $\hat A$ that if $$\muskew^\Fol_{(\xi,x)}(\hat A)= 1$$ then $$\muskew^\Fol_{(\xi',x)}(\hat A)= 1$$ for any $\xi' \in \Sigma$.  It follows that $$\muskew^\Sigma_{(\xi,x)}\hat A = 1$$ for \ae $(\xi,x) \in \hat A$. 
In particular, $\hat A\circeq   \Sigma \times \td A$ for some set $\td A\in \B_M$. 
\end{proof}

We remark that if $\xi$ projects to $\omega$ under the natural projection $\Sigma\to \Sigma_+$, then the subspace $E^s_\xi(x)$ and the subspace  $E^s_\omega(x)$ given by Proposition \ref{prop:OMT} coincide almost surely.  
It then follows from Proposition \ref{prop:hopf} that  the bundle $E_\omega^s(x)$ in  Theorem \ref{thm:main} is non-random if and only if the bundle $E^s_\xi(x)$ is $\Fol$-measurable.  
Thus, Theorem \ref{thm:main} follows from the following 2 results.

\begin{theorem} \label{thm:main2}\label{thm:skewproduct}
Let $ \nunaught$ and $ \munaught$ be as in Theorem \ref{thm:main}.  
Let $F\colon \Sigma\times M \to \Sigma \times M$ be the canonical skew product and let $\muskew$ be as in Proposition \ref{prop:mudef}.  Assume  the fiber-wise conditional measures $\muskew_\xi$ are non-atomic.  Then  either $(\xi, x) \mapsto E^s_\xi(x)$ is $\F$-measurable or $\muskew$ is fiber-wise SRB.
\end{theorem}
Recall that a measure is \emph{non-atomic} if there is no point with positive mass.  By the ergodicity of $\mu$ under the dynamics of $F$ it follows that either $\muskew_\xi$ is non-atomic \as or there is a $N\in \N$ such that  $\muskew_\xi$ is supported on exactly $N$ points a.s.  We consider  the case that $\muskew_\xi$ is finitely supported separately.  
\begin{theorem} \label{thm:main2'}\label{thm:atomicfibertoinvariant}\label{thm:mainAtominv}
Let $ \nunaught$ and $ \munaught$ be as in Theorem \ref{thm:main}.  Assume the fiber-wise conditional measures $\muskew_\xi$ are  finitely supported $ \nunaught^\Z$-a.s.  Then either $(\xi, x) \mapsto E^s_\xi(x)$ is $\F$-measurable or the measure $\munaught$ is finitely supported and $\nunaught$-\as invariant.  
\end{theorem}

\begin{remark}\label{rem:samesign}
In the proof of Theorem \ref{thm:main} below, we may assume that $\munaught$ has one exponent of each sign.  Indeed if $\munaught$ has only negative exponents then the measurability of $(\xi, x) \mapsto E^s_\xi(x)$ is trivial.  Furthermore, if $\munaught$ has only positive exponents then a standard argument  shows that $\munaught$ is finitely supported and $\nunaught$-\as invariant.  Indeed, if all exponents are positive, then the measures $\muskew_\xi$ are finitely supported for \ae $\xi$.  That $\munaught$ is $\nunaught$-\as invariant follows, for instance, from the invariance principle in \cite{MR2651382}, the $\hat \Fol$-measurability of the measure $\muskew_\xi$, and an argument similar to Proposition \ref{prop:hopf} above.  
\end{remark}



\subsection{Statement of results: general skew products}
We introduce a generalization of Theorem \ref{thm:main2}, the proof of which consumes Sections \ref{sec:mainprop}--\ref{sec:lemmaproof}. 
Let $\theta\colon (\Omega, \B_\Omega, \nuskew) \to (\Omega, \B_\Omega, \nuskew) $ be as in Section \ref{sec:absSkew}.
Let $M$ be a closed $C^\infty$ surface and let $\scrF$ be a cocycle generated by a $\nuskew$-measurable map $\xi\mapsto f_\xi$ satisfying the integrability  hypothesis 
\eqref{eq:IC2}.  
Fix $\muskew$ an ergodic, $\scrF$-invariant, hyperbolic, Borel probability measure on $X= \Omega\times M$.  
For the general setting we will further assume the measures $\mu_\xi$ are non-atomic $\nu$-a.s.  It follows that 
from the hyperbolicity and non-atomicity of the fiber-wise  measures $\mu_\xi$ that the fiber-wise derivative $DF$ has two exponents $\lambda^s$ and $\lambda^u$, one of each sign.  

We say a sub-$\sigma$-algebra $\hat \F \subset \B_\Omega$ is \emph{decreasing} (for $\theta$) if $$\theta(\hat\Fol )= \{ \theta(A)\mid A\in \hat \Fol\}\subset \hat \Fol.$$ (Note that  $\hat \Fol$ is decreasing under the forwards dynamics if the partition into atoms is an increasing partition in the sense of \cite{MR819556}.  Alternatively,  $\hat \Fol$ is  decreasing if  the map $\theta\inv\colon \Omega\to \Omega$ is $\hat\Fol$-measurable.)
As a primary example, the sub-$\sigma$-algebra of $\Sigma$ generated by local unstable sets is decreasing (for $\sigma\colon \Sigma \to \Sigma$).  

Let $\hat \F$ be an decreasing sub-$\sigma$-algebra and write $\F$ for the $\muskew$-completion of $\hat \F \otimes \B_M$ where $\B_M$ is the Borel algebra on $M$.  
As in the previous section, to compare stable distributions in different fibers over $\Omega$ write $E_\xi^s(x)\subset T_xM$ for the subspace with $E^s(\xi,x) = \{\xi\}\times E_\xi^s(x)$.  We then consider $(\xi,x) \mapsto 
E_\xi^s(x)$ as a measurable map from $X$ to the projectivization of $TM$.

With the above setup, we now state the following result which generalizes Theorem \ref{thm:main2}.

\label{sec:MainTheoremABSskew}
\begin{theorem} \label{thm:main3}\label{thm:skewproductABS}
Assume  $\muskew$ is hyperbolic and the conditional measure $\{\muskew_\xi\}$ are non-atomic a.s.  
Further assume
\begin{enumerate}
\item $\xi\mapsto \cocycle    [\xi] [-1]$ is $\hat \F$-measurable, and
\item $\xi\mapsto \muskew_\xi$ is $\hat \F$-measurable.
\end{enumerate}
Then either $(\xi,x)\mapsto E^s_{\xi}(x)$ is $\F$-measurable or $\muskew$ is fiber-wise SRB.  (See Definition \ref{def:fiberwiseSRB}.) 
\end{theorem}
Note that the hypothesis that  $\xi\mapsto \cocycle    [\xi] [-1]$ is $\hat \F$-measurable combined with the fact that $\hat \Fol$ is decreasing  implies $\xi\mapsto \cocycle    [\theta^{-j}(\xi)] [-1]$ is $\hat \Fol$-measurable for all $j\ge 0$.  It follows that $\xi\mapsto \cocycle    [\xi] [n]$ is $\hat \Fol$-measurable for all $n\le 0$.  It then   follows that  $\F$ is an decreasing sub-$\sigma$-algebra of $\B_X$.

We recall that in the case that $F$ is the canonical skew product for a random dynamical system and $\hat \Fol$ is the sub-$\sigma$-algebra generated by local unstable sets, writing $\xi = (\dots, f_{-1}, f_0, f_1, \dots)$ the $\hat\F$-measurability of $\xi\mapsto \cocycle    [\xi] [-1] = (f_{-1})^{-1}$ follows from construction.  The $\hat\F$-measurability of $\xi\mapsto \muskew_\xi$ follows from the construction of the measure $\muskew$ given by \eqref{eq:mulim} in Proposition \ref{prop:mudef}.  
Theorem \ref{thm:skewproduct} 
then follows immediately from Theorem \ref{thm:skewproductABS}.    


\section{Some Applications}
We present a number of applications of our main theorems.  
\subsection{Groups of measure-preserving diffeomorphisms}
Fix $M$ a closed surface.  Let $\mu$ be a Borel probability measure on $M$. 
Let $\diff^2_\mu (M)$ denote the group of $C^2$, $\mu$-preserving diffeomorphisms of $M$. 
Given $f\in \diff^2_\mu(M)$ write   $\lambda^i(f,\mu,x)$ for the $i$th Lyapunov exponents of $f$ with respect to the measure $\mu$ at the point $x$.  If $f$ is ergodic (for $\mu$) we write  $\lambda^i(f,\mu)$ for the $\mu$-almost surely constant value of $\lambda^i(f,\mu,x)$.  
A diffeomorphism $f\in \diff^2_\mu(M)$ is 
\emph{hyperbolic (relative to  $\mu$)} if $\lambda^i(f,\mu,x)\neq 0$ for almost every $x$ and every $i$.  

Note that if $f\in \diff^2_\mu(M)$ is  hyperbolic and $\mu$ contains no atoms, then $(f,\mu)$ has one exponent of each sign $\lambda^s(f,\mu,x)<0< \lambda^u(f,\mu,x).$ 
For such $f$, we write  $T_xM=E^s_f(x)\oplus E^u_f(x)$ for the $\mu $-measurable Oseledec's splitting induced by  $(f, \mu)$.

%

\begin{theorem}\label{thm:4}
Let $\mu$ be a  Borel probability measure on $M$ with no atoms.    Suppose $\diff_\mu^2(M)$ contains an ergodic, hyperbolic element $f$.  
Write $\Gamma= \diff^2_\mu(M)$.
\begin{enumerate}[label=({\alph*})]
\item \label{thm:4:1} If the union $E^u_f\cup E^s_f$ is not $\Gamma$-invariant and neither $E^s_f$ nor $E^u_f$ is $\Gamma$-invariant, then $\mu$ is absolutely continuous.  
\item  \label{thm:4:2} If the union $E^u_f\cup E^s_f$ is not $\Gamma$-invariant and  $E^u_f$ is $\Gamma$-invariant, then $\mu$ is an  SRB measure for $f$.
\item  \label{thm:4:3} If the union $E^u_f\cup E^s_f$ is not $\Gamma$-invariant and  $E^s_f$ is $\Gamma$-invariant, then $\mu$ is an  SRB measure for $f\inv$.
\end{enumerate}
\end{theorem}

%
%


In the case that $\lambda^s(f,\mu)\neq - \lambda^u(f, \mu)$ we can give more precise results using the following lemma.  
\begin{lemma}\label{prop:stiffdistributions}
Let $\mu$ be non-atomic and let  $f\in \diff^2_\mu(M)$ be ergodic and hyperbolic.  
Suppose $\lambda^s(f,\mu)\neq - \lambda^u(f, \mu)$.  Then any $g\in \diff^2_\mu(M)$ that preserves the union $E^u_f\cup E^s_f$ preserves the individual distributions $E^s_f$ and $E^u_f$.
\end{lemma}

\begin{proof}
Suppose $g\in \diff^2_\mu(M)$  preserves the union $E^u_f\cup E^s_f$ almost surely but \begin{equation}\label{eqjjjjklkj}D_x g(E^s_f(x))= E^u_f(g(x))\end{equation}  for a positive measure set of $x$.  
Let $\Pp TM$ denote the projectivized tangent bundle.  Let $\nu= t\delta_g + (1-t) \delta_f$ and  $\Sigma = \{f,g\}^\Z$.   On $\Sigma \times \Pp TM$ consider the measure $$d\eta(\xi,x, E) = d \nu^\Z(\xi)\  d\mu(x) \ d (.5 \delta_{E^u_f(x)} + . 5\delta_{E^s_f(x)})(E).$$  Write 
 $DF$ for  the derivative skew product 
 $DF \colon \Sigma \times TM\to \Sigma \times TM$
 and $ \Pp DF$ for the projectivized derivative skew product 
  $\Pp DF \colon \Sigma \times \Pp TM\to \Sigma \times \Pp TM$.  
 Then $\eta$ is $\Pp DF$-invariant and is ergodic by   \eqref{eqjjjjklkj}.  Let $\Phi\colon \Sigma \times \Pp TM\to \R$ be $$\Phi(\xi,x,E) = \log \| \restrict{D_xf_\xi }{E}\|.$$
Then for $\mu$-\ae $x$ and  $v\in E^s_f(x)\cup E^u_f(x)$  with $v\neq0$ we have 
\begin{align*}\lim_{n\to \infty} \frac 1 n &\log \| D_x \cocycle(v) \| 
= \lim_{n\to \infty} \frac 1 n \sum_{j=1}^{n-1}\Phi (\Pp D F ^j(\xi,x,[v]))
= \int \Phi \ d \eta 
\\& = (1-t) \left( .5 \lambda^u(f, \mu) + .5 \lambda^s(f, \mu) \right) + 
. 5 t \int \sum _{E \in\{ E^u_f(x), E^s_f(x)\}} \log  \| \restrict{D_xg }{E}   \| \ d \mu(x). 
\end{align*}
As  the $C^1$ norm of $g$ is bounded,  for $t>0$ sufficiently small,  either all fiber-wise exponents of $DF$ are negative or all fiber-wise  exponents of $DF$  are positive.  This contradicts that $\mu$ is non-atomic.  
\end{proof}

From the above lemma we have 
\begin{theorem}
Let $\mu$ be Borel probability measure on $M$ with no atoms.    Suppose $\diff_\mu^2(M)$ contains an ergodic, hyperbolic element $f$ with 
 $\lambda^s(f,\mu)\neq - \lambda^u(f, \mu)$. Then with $\Gamma= \diff_\mu^2(M)$ either
\begin{enumerate}[label=({\alph*})]
\item both $E^u_f$ and $E^s_f$ are $\Gamma$-invariant
\item exactly one  $E^u_f$ and $E^s_f$ is $\Gamma$-invariant in which case $\mu$ is SRB for $f$ or $f\inv$.
\end{enumerate}
\end{theorem}
Note that the hypotheses that $\lambda^s(f,\mu)\neq - \lambda^u(f, \mu)$ implies that $\mu$ is not absolutely continuous.

\subsection{Smooth stabilizers of measures invariant by Anosov maps} 

As a consequence of the results in the previous section, we obtain a strengthening of the   result from \cite{AWB:SS}.  

Let $f\colon \T^2\to \T^2$ be Anosov.  Then there is a hyperbolic  $A\in \Gl(2,\Z)$  such that any lift $\td f\colon \R^2\to \R^2$ of $f$ is of the form $f(x) = Ax + \eta(x)$ where $\eta\colon \R^2\to \R^2$ is $\Z^2$ periodic. Given $B\in \Gl(2,\Z)$ let $L_B\colon \T^2\to \T^2$ be the induced map.  
Then there is a (non unique) homeomorphism $h\colon \T^2\to \T^2$ with $h\circ f = L_A\circ h$.  

Let $\mu$ be a fully supported, ergodic, $f$-invariant measure.  Let   $K\subset \R^2$ be the set $K = \{ v :  h_*\mu \text{ is $T_v$-invariant}\}\subset \R^2$ where $T_v\colon \T^2\to \T^2$ is the translation by $v$.  Then $K$ descends to a closed $L_A$-invariant subgroup of $\T^2$ so is either discrete or is all of $\R^2$. The latter case can happen only if the measure $\mu$ is the measure of maximal entropy for $f$.   It follows that the group $(A-I)\inv K$ is either discrete or is all of $\R^2$.  Let $\hat K$ be the smallest subgroup of $\R^2$ that is invariant under the centralizer $C_{\Gl(2,\Z)}(A)$ of $A$ in $\Gl(2,\Z)$ and contains $(A-I)\inv K$.  Note that $\hat K$ descends to a subgroup of $\T^2$.   Then $\hat K$ is either $\R^2$ or is discrete.  Let $T_{\hat K}$ denote the corresponding group of translations on $\T^2$.  Then $T_{\hat K}$  is finite if $\mu$ is not the measure of maximal entropy.  

Recall that the centralizer  of $A$ is of the form $C_{\Gl(2,\Z)}(A) = \langle \pm M\rangle$ for some $M\in \Gl(2,\Z)$.

\begin{theorem}\label{thm:5}
Let $f\colon \T^2\to \T^2$ be a $C^{2}$ Anosov diffeomorphism and let $\mu$ be a fully supported, ergodic, $f$-invariant measure.  
If $\mu$ is not absolutely continuous  
then for every $g\in \diff^2_\mu(M)$ there is a $B\in C_{\Gl(2,\Z)}(A)$ and $v\in (A-I)\inv(K)$ with $$h\circ g\circ h\inv (x)= L_B(x) + v;$$
in particular, $\diff_\mu^2(\T^2)$ is isomorphic to a subgroup of $$ C_{\Gl(2,\Z)}(A) \ltimes T_{\hat K}$$

Moreover, if $\mu$ is not the measure of maximal entropy (for $f$) then $T_{\hat K}$ is finite, whence $\diff_\mu^2(\T^2) $ is virtually-$\Z$.  
\end{theorem}

Recall that a group is \emph{virtually-$\Z$} if it contains a finite-index subgroup isomorphic to $\Z$. 
Theorem \ref{thm:5} follows  exactly from the argument in  \cite{AWB:SS}  with  only  minor  modifications coming from Theorem \ref{thm:4}. 
\begin{proof}
Recall that if $f$ is Anosov then the measurable distributions $E^s_f$ and $E^u_f$ appearing in Oseledec's splitting coincide with  continuous transverse distributions. 

Consider first $g\in \diff^2_\mu(M)$ such that 
$Dg$ does not interchange $E^s_f$ and $E^u_g$ on a set of full measure (and hence at every point).  
Then, if $\mu$ is not absolutely continuous, by Theorem \ref{thm:4} at least one of the (continuous) distributions $E^s_f$ or $E^u_f$ is preserved (on a set of full measure and hence everywhere) by $g$.  
Then, as the integral foliations to $E^s_f$ and $E^u_f$ are unique, it follows that either the stable or unstable foliation of $f$ is preserved by every such  $g$.   

It is then shown in  \cite{AWB:SS} that $g$ necessarily preserves both the stable and unstable foliations for $f$ and hence preserves the corresponding  tangent line-fields $E^s_f$ and $E^u_f$. 
If there exists $g\in \diff^2_\mu(M)$ such that $g$ interchanges $E^s_f$ and $E^u_f$ then we may restrict to  an  index-2 subgroup preserving $E^s_f$ and $E^u_f$ and the corresponding foliations.  

The remainder of the proof of Theorem \ref{thm:5} and a more detailed description of the structure of $\diff^2_\mu(\T^2)$ proceeds exactly as in \cite{AWB:SS} and will not be repeated here.  
\end{proof}

\subsection{Perturbations of algebraic systems}\label{sec:alge}

Let $A,B\in \Gl(2,\Z)$ be  hyperbolic matrices.  Write  $E^{s}_A$ and $E^{u}_A$, respectively,  for the stable and unstable eigenspaces of $A$.  We say that $\{A,B\}$  satisfy a  \emph{joint cone condition} if there  are disjoint  open cones $C^s$  and $C^u$, containing  $\{E^s_A, E^s_B\}$ and $\{E^u_A, E^u_B\}$, respectively, with
$A\inv C^s\subset C^s$, %
$B\inv C^s\subset C^s$, 
$A C^u\subset C^u$, and 
$B C^u\subset C^u$
and a number $\kappa>1$ such that if $v\in C^u$ then $\|Bv \| > \kappa \| v\|$ and $\|Av\|> \kappa \|v\|$ and if $w\in C^s$ then $\|B\inv w \| > \kappa \| w\|$ and $ \|A \inv w\|> \kappa \|w\|.$

Given $A\in \Gl(2,\Z)$ let $L_A\colon \T^2\to \T^2$ be the induced diffeomorphism.  

\begin{prop}\label{thm:perturb}
Suppose that $A$ and $B$  do not commute and satisfy a joint cone condition.  Then for sufficiently small $C^2$ perturbations $f$ of $L_A$ and $g$ of $L_B$, for  $\nu= p\delta_f+ (1-p)\delta_g$ with $p\in (0,1)$ the only  ergodic, $\nu$-stationary measures are SRB or  finitely supported.  

Moreover for every such  $f$ and a generic $g$, the only $\nu$-stationary measure is SRB.
\end{prop}
Note that in the setting of the  above proposition,   stationary measures with the SRB property are unique.  The proof of the proposition will be given  in Section \ref{sec:left3}.

\begin{theorem}\label{thm:perturbSRB}
Let $\Gamma\subset \Sl(2,\Z)$ be an infinite subgroup that is not virtually-$\Z$. 
Let $S=\{A_1,\dots, A_n\}$ be a finite set generating $\Gamma$.  
Consider $0< p_k< 1$ with $\sum_{k=1}^np_k=1$ and let $\nu_0= \sum p_k \delta_{L_{A_k}}$.  Then there is an open set $U\subset \diff^2(\T^2)$ with $\nu_0(U) =1$  such that for every  probability $\nu$ on $U$ sufficiently close to $\nu_0$,   
 any ergodic, $\nu$-stationary measure is either atomic, or is  hyperbolic with one exponent of each sign and is  SRB.  
 \end{theorem}
 The proof of the theorem   will be given  in Section \ref{sec:left3}.
  
Let $m$ denote the Lebesgue area on $\T^2$.  If   we restrict the above   to the setting of  area-preserving perturbations, we obtain the following nonlinear counterpart to \cite{MR2831114}.  Note in particular that we obtain stiffness of all stationary measures.  
\begin{theorem}\label{thm:perturbSRBtoVolume}
Let $\Gamma\subset \Sl(2,\Z)$ be an infinite subgroup that is not virtually-$\Z$. 
Let $S=\{A_1,\dots, A_n\}$ be a finite set generating $\Gamma$.  
Consider $0< p_k< 1$ with $\sum_{k=1}^np_k=1$ and let $\nu_0= \sum p_k \delta_{L_{A_k}}$.  Then 
 there is an open set $U\subset \diff^2_m(\T^2)$ with $\nu_0(U) =1$  such that for every  probability $\nu$ on $U$ sufficiently close to $\nu_0$,  
 any ergodic, $\nu$-stationary measure is hyperbolic with one exponent of each sign and either coincides with  $m$ or is atomic.  
 
In particular, every $\nu$-stationary measure is preserved by every $g\in \diff_m^2(\T^2)$ in the support of $\nu$.  
 \end{theorem}
 The theorem follows from Theorem \ref{thm:perturbSRB} and (the proof of) Theorem \ref{thm:3}.  In the proof of Theorem \ref{thm:perturbSRB}, it is shown that for all $\nu$ sufficiently close to $\nu_0$, every ergodic $\nu$-stationary measure $\mu$ has a positive exponent.  That $\mu$ also has a negative exponent follows from \eqref{eq:jacobians}.  
 Moreover, for such $\nu$,  a positive $\nu$-measure set of $f\in \diff^2_m(\T^2)$ are Anosov, whence $m$ is ergodic for such $f$ and hence ergodic for $\nu$.  
 
 Finally, we consider stationary measures for perturbations of rotations.
Let $R_1,\dots, R_\ell$ be $\ell$ rotations  in $\R^3$  generating a dense subgroup of $\mathrm{SO}(3, \R)$.  We identify each $R_i$ with a diffeomorphism of $S^2\subset \R^3$.   Let $m$ denote the unique $\mathrm{SO}(3, \R)$ invariant measure on $S^2$.  
 \begin{theorem}
For $k\in \N$ sufficiently large,  for each $1\le i\le \ell$ there is a neighborhood $R_i\in U_i\subset \diff^k_m(S^2) $ such that given  any $g_i\in U_i$ and $\nu= \frac 1  \ell \sum _{i=1}^ \ell \delta_{g_i}$, any ergodic $\nu$-stationary measure on $S^2$ is either  finitely supported or coincides with  $m$.  
\end{theorem}  
\begin{proof}
In \cite{MR2309172} it is shown that either the diffeomorphisms  $g_i$ are simultaneously smoothly conjugated to $R_i$ or every $\nu$-stationary measure is hyperbolic.  
In the first case, the only stationary measures for the corresponding  $R_i$ is $m$ and thus using the conjugacy and that each $g_i$ preserves $m$, the only $\nu$-stationary measure is $m$.  

In the latter case, it is also shown in  \cite[Corollary 4]{MR2309172} that the stable line-field is not non-random.  The result in this case follows from Theorem \ref{thm:3} and, as is also shown in \cite{MR2309172}, that $m$ is ergodic for the perturbed system.  \end{proof}
 
\subsection{Other applications}
From Theorem \ref{thm:3}, we immediately obtain the main results of \cite{MR2831114,MR2726604} for measures $\nu$ on $\Sl(2,\Z)$ acting on $\T^2$ that satisfy a $\log$-integrability condition $\int \log\| A\| \ d \nu(A)<\infty $.  In  \cite{MR2831114} the measure $\nu$ is assumed finitely supported.  In \cite{MR2726604} a stronger integrability hypothesis is needed.  Using the methods of this paper, the results of  \cite{MR2831114} are expected to hold under $\log$-integrability hypothesis.  

Consider a flat surface $S$ with Veech group $\Gamma\subset \SL(2,\R)$.  As was pointed out to the authors by J. Athreya, Theorem \ref{thm:3} implies that if the Veech group is infinite and non-elementary then for any finitely supported measure $\nu$ generating $\Gamma$, all ergodic $\nu$-stationary measures $\mu$ on $S$ are either finitely-supported or are the invariant area.  There are technicalities  in applying Theorem \ref{thm:3} directly as the action is non-differentiable at the cone points.  This mild difficulty won't be addressed here.  

\section{Background and notation} \label{sec:BG}
In this section, we continue to work in the setting introduced in Sections \ref{sec:absSkew} and \ref{sec:MainTheoremABSskew}.  
We outline extensions of a number of standard facts from the theory of nonuniformly hyperbolic diffeomorphisms to the setting of the fiber-wise dynamics for skew products.  
As previously remarked, Theorem \ref{thm:main} holds trivially if the fiber-wise exponents are all of the same sign.  Moreover, the hypotheses of Theorem \ref{thm:skewproductABS} rule out that all exponent are of the same sign.  
We thus assume for the remainder  
that we have   one Lyapunov exponent of each sign
$\lambda^s<0<\lambda^u$.  
For the remainder, fix $0< \eps < \min\{1,\lambda^u/200, -\lambda^s/200\} .$ 

\subsection{Fiber-wise nonuniformly hyperbolic dynamics} \label{sec:PT}
We present a number of  extensions of the  theory of nonuniformly hyperbolic diffeomorphisms 
to the fiber-wise dynamics of skew products.  

\subsubsection{Subexponential estimates} 
We have the following standard results that follow from the integrability hypothesis 
\eqref{eq:IC2} and 
tempering kernel arguments  (c.f.\ \cite[Lemma 3.5.7]{MR2348606}.)

\begin{lemma}\label{prop:tempered}
There is a subset $\Omega_0\subset \Omega$ with $\nu(\Omega_0)= 1$ and  a measurable function 
 $D\colon \Omega_0\to [1,\infty)$ such that for $\nu$-\ae $\xi\in \Omega_0$ and $n\in \Z$.
\begin{enumerate}
\item $|f_{\xi}|_{C^1} \le  D(\xi)$
\item $|f\inv_{\xi}|_{C^1} \le  D(\xi)$
\item $\lip(Df_{\xi}) \le   D(\xi)$ and  $\lip(Df\inv_{\xi}) \le   D(\xi)$ 
\item  $D(\theta^n(\xi)) \le e^{|n|\eps} D(\xi)$ for all $n\in \Z$.  
\end{enumerate}
\end{lemma}
\noindent Here $\lip(Df_\xi)$ denotes the Lipschitz constant of the map $x\mapsto D_xf_\xi$ for fixed $\xi.$
\begin{lemma}\label{prop:growth}
There is a measurable  function $L\colon X\to [1,\infty)$   
such that for $\muskew$-\ae $(\xi,x)\in X$ and $n \in \Z$
\begin{enumerate}
\item For $v\in E^s_\xi( x)$, 
	\[	 L(\xi,x)\inv \exp(n\lambda^s- |n|\tfrac 1 2 \eps)\|v\|\le  \|Df_{\xi}^n v\|
		\le L(\xi,x) \exp(n\lambda^s+ |n|\tfrac 1 2 \eps)\|v\| . \]
\item For $v\in E^u_\xi( x)$, 
\[	L(\xi,x)\inv  \exp(n\lambda^u- |n|\tfrac 1 2 \eps)\|v\|\le  \|Df_{\xi}^n v\|
		\le L(\xi,x) \exp(n\lambda^u+ |n|\tfrac 1 2 \eps)\|v\|.   \]
\item $\angle\left(E^s_{\theta^n(\xi)}(D\cocycle(x)), E^u_{\theta^n(\xi)}(D\cocycle(x))\right)  >\dfrac{1}{L(\xi, x)} \exp (-|n|\eps ) $.
\end{enumerate}
Furthermore for $n\in \Z$
$$L(F^n(\xi,x)) \le L(\xi,x) e^{\eps |n|}.$$   
\end{lemma}
\noindent Here $\angle$ denotes the Riemannian angle between two subspaces.

\subsubsection{Lyapunov charts}
\label{sec:stabmanifold}
We introduce families of two-sided Lyapunov charts.  
The construction depends on the construction of a  Lyapunov norm which we present in    Section \ref{sec:lyapnorm}.  
We note that in Section \ref{sec:stabmanifold}, in the case that $\Omega =( \Diff^2(M))^\Z$ we will need one-sided charts that  depend only on the future itinerary of $\xi\in (\Diff^2(M))^\Z$.  
 Given $v\in \R^2$ decompose $v = v_1+ v_2$ according to the standard basis and write $|v|_i = |v_i|$ and $|v| = \max\{|v|_i\}$.
Write $\R^2(r)$ for the ball of radius $r$ centered at $0$.  
 
 \def\wtd{\widetilde}
\label{sec:charts2sided}
From standard constructions (see \cite[Appendix]{MR819556}, \cite [VI.3]{MR1369243})
for every  $0<\epone<\eps$, there is a measurable function $\ell\colon \Omega \times M\to [1,\infty)$  and a full  measure set $\Lambda \subset \Omega \times M$ such that  
\begin{enumerate}
\item for $(\xi,x)\in \Lambda$ there is a neighborhood $U_{(\xi,x)}\subset M $ of $x$ and a $C^\infty$ diffeomorphism $\phi(\xi,x)\colon U_{(\xi,x)} \to \R^2(\ell(\xi,x)\inv) $
with
\begin{enumerate}
	\item $\phi(\xi,x) (x) = 0$;
	\item $D\phi (\xi,x) E^s_\xi (x) = \R \times \{0\}$;
	\item $D\phi(\xi,x) E^u_\xi (x)  = \{0\} \times \R$;
\end{enumerate}
\item writing $$\td f(\xi, x)= \phi (F(\xi, x)) \circ f_{\xi} \circ \phi (\xi, x) \inv, \quad 
 \td f\inv(\xi, x)= \phi (F^{-1}(\xi, x)) \circ f_{\xi}\inv \circ \phi (\xi, x) \inv$$ 
where defined we have 
\begin{enumerate}
	\item \label{1111}$\td f (\xi, x) (0) = 0$;
	\item $D_0  \td f (\xi, x) =  \left(\begin{array}{cc}\alpha & 0 \\0 & \beta \end{array}\right)$
	where $e^{\lambda^s-\epone} \le \alpha \le e^{\lambda^s+\epone}$
	 and  $e^{\lambda^u-\epone} \le \beta \le e^{\lambda^u+\epone}$;
\end{enumerate}
writing $\lip(\cdot)$ for  the Lipschitz constant of a map on its domain
\begin{enumerate}[resume]
	\item $\lip(\td f (\xi, x) - D_0  \td f  (\xi, x))<\epone$; 
\item \label{4444} $\lip(D  \td f (\xi, x))<\ell(\xi, x)$;   

\end{enumerate} 
\item similar properties to \ref{1111}--\ref{4444} hold for $\td f\inv(\xi,x)$.
\item there is a uniform $k_0$ with  $k_0\inv \le \lip(\phi(\xi,x))\le \ell(\xi, x) $.
\item $\ell(F^n(\xi,x)) \le \ell(\xi,x) e^{|n|\epone}$ for all $n\in \Z$.  
\end{enumerate}
Let $$\lambda_0 = \max\{\lambda^u, -\lambda^s\} + 2\epone.$$  
Then, the domains of $\td f(\omega, x)$ and $\td f\inv(\omega, x)$ contain the ball in $\R^2$ of norm $\ell(\xi,x) \inv e^{-\lambda_0-\epone}$.  
Note also that the domain of $\phi(\xi,x)$ contains a ball of radius $\ell(\xi, x)^{-2}$ centered at $x$.  


Write $\R^s=\R \times \{0\}$ and $\R^u =  \{0\} \times \R$.  
Recall that $g\colon D\subset \R\to \R$ is $k$-Lipschitz if $|g(x)- g(y)| \le k |x-y|$ for $x,y\in D$.  We have the following observation.  
\begin{lemma}
\label{lemm:incling}
Let $D\subset \R^s(e^{-\lambda_0-\epone}\ell(\xi,x)\inv )$.  
Let $g\colon D \to \R^u(e^{-\lambda_0-\epone}\ell(\xi,x)\inv )$ be a $1$-Lipschitz function.  Then 
$$\td f \inv(\xi, x) (\graph (g))$$ is  the graph of a $1$-Lipschitz function $$\hat g\colon \hat D \to \R^u( e^{-\lambda_0-\epone} \ell(F\inv(\xi,x))\inv )$$  
for some $\hat D \subset \R^s\left( \ell(F\inv(\xi,x))\inv\right).$
\end{lemma}
%
\subsubsection{Stable manifold theorem}

Relative to the  charts $\phi (\xi,x)$ above, one may apply either the Perron--Irwin method or the Hadamard graph transform method to construct stable manifolds.  The existence of stable manifolds for diffeomorphisms of manifolds with non-zero exponents is due to Pesin  \cite{MR0458490}. 
In the case of random dynamical systems, given the family of charts above, the statements and proofs hold with minor modifications (see for example \cite{MR1369243}).
See Section \ref{sec:stabmanifold} for some details in the construction of stable manifolds relative to one-sided charts.

\begin{theorem}[Local stable manifold theorem]\label{thm:locPesinStab}
\item For  $(\xi,x)\in \Lambda$ there is a 
 $C^{1,1}$ function $$h^s{(\xi,x)}\colon \R^s(\ell(\xi,x)\inv) \to \R^u(\ell(\xi,x)\inv))$$ with  \begin{enumerate}
	\item $h^s(\xi,x)(0)= 0$;
\item 				$D_0h^s{(\xi,x)}= 0$;
\item $ \|Dh^s(\xi,x) \| \le 1/3$;
\item $\td f(\xi,x) (\graph( h^s(\xi,x)))\subset  \graph (h^s(F(\xi,x)))\subset \R^2(\ell(F(\xi,x)) \inv)$; in particular,  $\graph(h^s(\xi,x))$, is in the domain of $\td f(\xi,x)$.
\end{enumerate}
Setting   
$$V^s(\xi, x) := \phi(\xi,x) \inv  \left( \graph\left({ h^s{(\xi, x)}}  
\right)\right)$$ we have 
\begin{enumerate}[resume]
\item $f_{ \xi}(V^s(\xi, x)) \subset V^s(F(\xi,x))$
\item\label{itempp} for $z,y\in V^s(\xi,x)$ and $n\ge 0$ 
$$d(f^n_\xi ( z), f^n_\xi(y))\le
\ell(\xi,x) k_0 \exp((\lambda^s + 2\epone)n)d(y,z). $$  
\end{enumerate}
\end{theorem}
\def\loc{\mathrm{loc}}
We define 
$ V^s(\xi,x)\subset M$ to be the \emph{local stable manifold} at $x$ for $\xi$ relative to the above charts.   
We  similarly construct  local unstable manifolds  $V^{u}(\xi,x)$. Similar to \ref{itempp} above,  for $z,y\in V^u(\xi,x)$ and $n\ge 0$
$$d(  f^{-n}_\xi ( z), f^{-n}_\xi(y))\le
\ell(\xi,x) k_0 \exp((-\lambda^u + 2\epone)n)d(y,z).$$
We remark that the family of local stable manifolds $\{V^s(\xi,x)\}$ forms a measurable family of embedded submanifolds.  


We define the \emph{global stable} and \emph{unstable manifolds} at $x$ for  $\xi$ by
\begin{align}
W^s_\xi(x)&: = \{ y\in M \mid \limsup_{n\to \infty} \tfrac 1 n \log d (f_\xi^n(x), f_\xi^n(y) )<0\}\\
W^u_\xi(x)&: = \{ y\in M \mid \limsup_{n\to -\infty} \tfrac 1 n \log d (f_\xi^n(x), f_\xi^n(y) )<0\}. \label{eq:unstmani}
\end{align}
For $\muskew$-\ae $(\xi,x)$  we have the nested union $W^s_\xi(x) = \bigcup_{n\ge0}(\cocycle)\inv(V^s(F^n(\xi, x)))$. 
  It follows for such $(\xi,x)$ that  $W^s_\xi(x)$ is a $C^{1,1}$-injectively immersed curve tangent to $E^s_\xi(x)$. 
We write 
$$
W^s(\xi,x):= \{\xi\}\times W^s_{\xi}(x), \quad 
W^u(\xi,x):= \{\xi\}\times W^u_{\xi}(x)$$ for the associated \emph{fiber-wise stable} and \emph{unstable} manifolds in $X= \Omega\times M$.  

The above family of charts and construction of local stable and unstable manifold  depends on  $\cocycle$ for all $n\in \Z$.  However, from \eqref{eq:unstmani} it is clear that $W^u_\xi(x)$ depends only on $\cocycle$ for all $n\le 0$. This fact will be used heavily in the sequel.  In Section \ref{sec:stabmanifold} we will use one-side charts to construct local stable manifolds that depend only on $\cocycle$ for all $n\ge 0$.

\subsection{Affine parameters}\label{sec:affpar}
Since each stable  manifold $\stabM x \xi$ is a curve, it has a natural parametrization via the Riemannian arc length.  
We define an alternative parametrization, defined on almost every stable manifold, that conjugates the non-linear dynamics $\restrict {f^n_\xi}{\stabM x \xi}$ and the linear dynamics $\restrict {Df^n_\xi}{E^s _\xi(x)}$.  We sketch  the construction and refer the reader to \cite[Section 3.1] {MR2261075} for additional details.  
\begin{prop}\label{prop:Stabman}
For almost every $(\xi,x)$ and any $y\in \stabM x \xi$, there is a $C^{1,1}$ diffeomorphism $$H^s_{(\xi,y)} \colon \stabM x \xi \to T_y \stabM x \xi$$ 
such that 
\begin{enumerate}
\item  restricted to   $\stabM x \xi$ the parametrization  
 intertwines the  nonlinear dynamics $f_\xi$ with the differential $D_yf_\xi$: 
$$ D_y f_\xi \circ  H^s_{(\xi,y)} = H^s_{F(\xi,y)} \circ \restrict{f_\xi}{\locstabM x \xi};$$

\item $H^s_{(\xi,y)} (y)= 0$ and $D_yH^s_{(\xi,y)} = \id$;
\item if $z\in \stabM x \xi$ then the change of coordinates   
\[H_{(\xi,y)} ^s \circ \left(H_{(\xi,z)} ^s\right)\inv \colon T_z \stabM x \xi \to T_y \stabM x \xi\] 
is an affine map  with derivative 
$$D_v\left(H_{(\xi,y)} ^s \circ \left(H_{(\xi,z)} ^s\right)\inv\right) = \rho_{(\xi,y)}(z)$$ for any $v\in T_z \stabM x \xi$  
where $ \rho_{(\xi,y)}(z)$  is defined below.  
\end{enumerate}
\end{prop}
We take $(\xi,x)$ to be in the full measure $F$-invariant set such that for any $y,z\in \stabM x\xi$ there is some $k\ge 0$ with $f_\xi ^k(z)$ and $f_\xi ^k(y)$ contained in $V^s(F^k(\xi,x))$ and sketch the construction of $H_{(\xi,y)}^s$.   
First consider any  $y,z\in V^s(\xi,x)$  
and define
$$J(\xi, z):= \|D_zf_\xi  v\| \cdot \|v\|\inv$$ for any non-zero $v\in T_z\stabM x \xi$ 
where $\|\cdot\|$ denotes the Riemannian norm on $M$.
We define 
\begin{equation}\label{eq:rho}\rho_{(\xi,y)}(z) := \prod_{k = 0}^\infty\dfrac{J(F^k(\xi, z))}{J(F^k(\xi,y))}\end{equation}
Following   \cite[Section 3.1] {MR2261075},   
the right hand side of \eqref{eq:rho} converges uniformly in $z$ to a Lipschitz function.  The only modifications needed in our setting are the  sub-exponential growth of $\|Df_\xi\|$ and the Lipschitz constant  of $Df_\xi$ along orbits given by Lemma \ref{prop:tempered}, as well the sub-exponential growth in $n$ of  the Lipschitz variation of the tangent spaces $T_zV^s(F^n(\xi,x))$ in $z$.  The growth of the  Lipschitz constant of   $T_zV^s(F^n(\xi,x))$ follows from   the proof of the stable manifold theorem (for example in \cite{MR1369243}) or by an argument similar to \cite[Lemma 4.2.2]{MR819556}.
We may extend the definition of $\rho_{(\xi,y)}(z)$ to any $z,y\in \stabM x \xi$ using that $f_\xi ^k(z)$ and $f_\xi ^k(y)$ are contained in $V^s  (F^k(\xi,x))$ for some $k\ge 0$. 



We now define the affine parameter $H^s_{(\xi,y)}\colon \stabM x \xi\to T_y\stabM x \xi$ as follows.  
We define $H^s_{(\xi,y)}$ to be orientation-preserving and  
$$\|H^s_{(\xi,y)}(z)\|:= \int_y^z \rho_{(\xi,y)}(t) \ dt $$
where $\int_y^z \psi(t) \ dt$ is the integral of the function $\psi$, along the curve from $y$ to $z$ in $\stabM x \xi$,  with respect to the Riemannian arc-length on $\stabM x \xi$.  

It follows from computations in \cite[Lemma 3.2, Lemma 3.3] {MR2261075} that the map $H^s_{(\xi,y)}$ constructed above satisfies the properties above. 
We similarly construct unstable affine parameters $H^u_{(\xi,x)}$ with analogous properties.

\begin{remark}
  The unstable line fields $E^u_\xi(x)$, unstable manifolds, and corresponding affine parameters are constructed using only  the dynamics of  $\cocycle $ for $n\le 0$.  Recall that we assume $\xi\mapsto \cocycle [\xi][-1]$ is $\hat \Fol$-measurable and that $\theta (\hat \Fol) \subset \hat \Fol$.  It follows that  $\xi\mapsto  \cocycle [\xi][n]$ is $\hat \Fol$-measurable for all $n\le 0$.  Thus, the line fields $(\xi,x)\mapsto E^u_\xi(x)$, the unstable manifolds $(\xi,x)\mapsto \unstM x \xi$, and the corresponding affine parameters $H^u_{(\xi,x)}$ are $\Fol$-measurable.  
\end{remark}
\subsubsection{Parametrization of stable and unstable manifolds}\label{sec:frames}
We use the affine parameters $H^s$ and the background Riemannian norm on $M$ to parametrize local stable manifolds.  
For $(\xi,x)\in X$ such that  affine parameters are defined, write 
$$\displaystyle \locstabM  x  \xi := (H^s_x)\inv\left (\{v\in E^s_\xi(x) \mid \|v\| < r\}\right) $$ for the local stable manifold in $M$ and 
$$ 
\locstab  x  \xi := \{\xi\}\times  \locstabM  x  \xi$$ for the corresponding  fiber-wise local stable manifold.  
We use similar notation for local unstable manifolds. 

We fix, once and for all,  a family $v^\sigma_{(\xi,x)}\in E^\sigma_{\xi}(x)\subset T_xM$ such that 
\begin{enumerate}
\item $(\xi,x)\mapsto v^s_{(\xi,x)}$ is $\muskew$-measurable;
\item $(\xi,x)\mapsto v^u_{(\xi,x)}$ is $\Fol$-measurable;
\item $\| v^s_{(\xi,x)}\| = \|v^u_{(\xi,x)}\|= 1$.
\end{enumerate}
The family $\{v^s_{(\xi,x)}\}$ and $\{v^u_{(\xi,x)}\}$ induce, respectively, $\muskew$- and $\Fol$-measurable trivializations of the stable and unstable bundles.
Recall that the affine parameters on unstable manifolds are constant along atoms of $\Fol$.  
We then obtain, respectively, $\muskew$- and $\Fol$-measurable maps $(\xi,x)\mapsto \I^s_{(\xi,x)}$ and $(\xi,x)\mapsto \I^u_{(\xi,x)}$ from $X$ to the space of $C^1$-embeddings of 
$\R$ into $M$ given by 
\begin{equation}\label{eq:affparammfolds}
\I^s_{(\xi,x)}\colon t \mapsto (H^s_{(\xi,x)} )\inv (tv^s_{(\xi,x)}),  \quad 
   \I^u_{(\xi,x)} \colon t \mapsto (H^u_{(\xi,x)} )\inv (tv^u_{(\xi,x)}).
\end{equation}

\subsection{Families of conditional measures}
\label{sec:condmeas}
The family of fiber-wise unstable manifolds $\{W^u(\xi,x)\}_{(\xi,x)\in X}$ forms a partition of a full measure subset of $X$.  However, such a partition is generally non-measurable. To define conditional measures we consider  a measurable partition $\P$ of $X$  such for $\muskew$-\ae $(\xi,x)\in X$ there is an $r>0$ such that $\locunst x \xi \subset  \P(\xi,x)\subset W^u(\xi,x)$. 
Such a partition is said to be \emph{$u$-subordinate}.   
 Let $\{\td \muskew^{\P}_{(\xi,x)}\}_{(\xi,x)\in X}$ denote a family of conditional probability measures with respect to such a  partition $\P$. 

\begin{definition}\label{def:fiberwiseSRB}
An $\scrF$-invariant measure $\muskew$ is \emph{fiber-wise SRB} if for any $u$-subordinate measurable partition $\P$ with corresponding family of conditional measures $\{\td \muskew^{\P}_{(\xi,x)}\}_{(\xi,x)\in X}$, the measure $\td \muskew^\P_{(\xi,x)}$ is absolutely continuous with respect to a Riemannian volume on $\unst  x \xi$ for \ae $(\xi,x)$. 
\end{definition}

In the setting introduced in Section \ref{sec:resultsStationary} we have the following.  
\begin{definition}\label{def:SRB}
Let $M$ be a closed manifold, $ \nunaught$ a Borel measure on $\diff^2(M)$,  
and  let $\munaught$ be a $\nunaught$-stationary probability measure.  We say $\munaught$ is \emph{SRB} if the measure $\muskew$ given by Proposition \ref{prop:mudef} is fiber-wise SRB for the associated canonical skew product \eqref{eq:skewdefn}.
\end{definition}

\begin{remark}
In fact, it follows from the proof of Proposition \ref{prop:SRBrandom} (see also \cite{MR1646606} for a related statement for general skew products) 
that $\muskew$ is fiber-wise SRB if and only if the conditional measures $\{\td \muskew^{\P}_{(\xi,x)}\}_{(\xi,x)\in X}$ are \emph{equivalent} to Riemannian volume on $\unst \xi x$ restricted to $\P(\xi,x)$.  Furthermore, with respect to the affine parameters introduced in Section \ref{sec:affpar}, the conditional measures coincide  up to normalization with the Haar measure. See \cite[Corollary 6.1.4]{MR819556}.
  \end{remark}
 

\newcommand{\scond}[2][(\xi, x)]{{ \muskew^s_{#1}}\left({#2}\right)}
\newcommand{\ucond}[2][(\xi, x)]{ \muskew^u_{#1}\left({#2}\right)}

Following a standard procedure, by fixing a normalization, for \ae $(\xi,x)$ we  define a \emph{locally-finite}, infinite measure $ \muskew^u_{(\xi,x)}$ on the curve $\unst x \xi$ that  restricts to $\{\td \muskew^{\P}_{(\xi,x)}\}_{{(\xi,x)}\in X}$, up to normalization,  for any {$u$-subordinate} partition $\P$ of $X$.  We choose the normalization  $\mu^u_{(\xi,x)}(\locunst [1] x \xi) = 1$.  
  Such a measure will be locally-finite in the internal topology of $\unstp p$ induced, for instance, by the affine parameters.  
We remark that the fiber entropy vanishes 
 if and only if the measures $\muskew^u_{(\xi,x)}$ and $\muskew^s_{(\xi,x)}$ have support  $\{(\xi,x)\}  $ for almost every $(\xi,x)\in X$.

\subsection{Relationships between entropy, exponents, and dimension}\label{sec:SRB}
Given $ (\xi,x)\in X$ we define the following pointwise dimensions
\begin{enumerate}
\item $\displaystyle \dim^u(\muskew,  (\xi,x)):= \lim_{r\to 0} \dfrac{\log \left( \ucond[ (\xi,x)] {\locunst x \xi}\right)}{\log r}$
\item $\displaystyle \dim^s(\muskew,  (\xi,x)):= \lim_{r\to 0} \dfrac{\log \left(\scond[(\xi,x)] {\locstabp x \xi}\right)}{\log r}$
\item  $\displaystyle \dim(\muskew,  (\xi,x)):= \lim_{r\to 0} \dfrac{\log (\muskew_\xi\{y\in M: d(x,y)<r\})}{\log r}$
\end{enumerate}
We note that $\dim(\muskew, (\xi,x))$ is the pointwise dimension of conditional measure $\muskew_\xi$ at the point $x$.  In the case that $\muskew$ is obtained from a stationary measure $\munaught$ from Proposition \ref{prop:mudef}, this need not coincide with the pointwise dimension of $\munaught$ at $x$.  
We have that  $\dim^u(\muskew, (\xi,x))$ and $\dim^s(\muskew, (\xi,x))$ are well defined and are furthermore constant \ae by the ergodicity of $\muskew$.   Write $\dim^{s/u}(\mu)$ for these constants.  Note that  $\dim(\muskew, (\xi,x))$ may not be defined if there are zero exponents.  

We have the following proposition.  Recall we write $\pi\colon X\to \Omega$ for the natural projection and   $h_{\muskew} (F \mid\pi)$ for the conditional metric entropy of $(F, \muskew)$ conditioned on the sub-$\sigma$-algebra generated by $\pi\inv$.  
\begin{proposition}\label{prop:ento}
In our setting, 
\begin{enumerate}
\item $h_\muskew(F\mid\pi) = \lambda ^u \dim^u(\muskew) = -\lambda ^s \dim^s(\muskew) $;
\item $\dim (\muskew,  (\xi,x)) = \dim^u(\muskew) + \dim^s(\muskew) $ for $\muskew$-\ae $ (\xi,x)$.
\end{enumerate}
\end{proposition}
(1) follows from a  generalization to the case of skew products  of the   Ledrappier-Young entropy formula  \cite{MR819557}.
This generalization  appears in \cite{MR2218998} in the case of  i.i.d.\ random dynamics;  modifications for the  case  of general skew products are outlined in \cite{MR2032494}.  
(2) follows from the results of \cite{MR2032494} generalizing to the random setting the dimension formula for hyperbolic measures proven   in \cite{MR1709302}.  

In our setting, we then have the following equivalent characterizations of the fiber-wise SRB property.
\begin{lemma}\label{prop:SRBdefn}
The following are equivalent.
\begin{enumerate}
\item $\mu$ is fiber-wise SRB;
\item  $h_\muskew(F\mid\pi) = \lambda^u$;
\item the measures $\muskew^u_{(\xi,x)}$ are equivalent to    Riemannian arc-length on $\unstM x \xi$ almost everywhere;
\item $\dim^u (\muskew) = 1$.
\end{enumerate}
\end{lemma}


\subsection{The family $\bar \muskew_{(\xi,x)}$}\label{sec:barmu}
Using the affine parameters  $H^u_{(\xi,x)}\colon \unstM x \xi \to E^u _\xi (x)$ and the trivialization \eqref{eq:affparammfolds}, we define a family of locally-finite  Borel  measures on $\R$ by 
\begin{equation}
\bar \muskew_{(\xi,x)}:= (\I^u_{(\xi,x)})\inv_*\muskew_{(\xi,x)} ^u. \label{eq:mubardefn}
\end{equation}
We equip the space of locally-finite  Borel  measures on $\R$  with its standard Borel structure (dual to compactly supported continuous functions).
We thus obtain a measurable function from $X$ to the locally-finite  Borel  measures on $\R$. 
Since the family of measures ${(\xi,x)}\mapsto \muskew_{{(\xi,x)}} ^u$ and parametrizations $\I^u$ are $\Fol$-measurable, it follows that $${(\xi,x)}\mapsto \bar \muskew_{(\xi,x)}$$ is $\Fol$-measurable.  

The family $\{\bar \muskew_{(\xi,x)}\}_{{(\xi,x)}\in X}$ will be our primary focus in the sequel.  In particular, the SRB property of $\muskew$ will follow by showing that for $\muskew$-\ae $p$, the measure  $\bar \muskew_{(\xi,x)}$ is the Lebesgue (Haar) measure on $\R$ (normalized on $[-1,1]$).

\section{Main Proposition and Proof of Theorem \ref{thm:skewproductABS}}
\label{sec:mainprop}

\subsection{Main Proposition}

The major technical result in the proof of Theorem \ref{thm:skewproductABS}
is the following key proposition, whose proof occupies Sections \ref{sec:lemmas}--\ref{sec:lemmaproof}.  
Given two locally finite measures $\eta_1$ and $\eta_2$ on $\R$ we write $\eta_1 \simeq \eta_2$ if there is some $c>0$ with $\eta_1 = c\eta_2$.

\begin{prop}
\label{prop:main}\label{lem:main}
Assume in Theorem \ref{thm:skewproductABS} 
that $(\xi,x)\mapsto E^s_\xi(x)$ is not $\F$-measurable.  
Then there exists $M>0$ such that for every sufficiently  
 small $\epsilon>0$ there exists a 
measurable set $G_\epsilon\subset X$  
with 
$$\muskew(G_\epsilon)> 0$$ such that for any ${(\xi,x)}\in G_\epsilon$ there is an affine map 
	$$\psi\colon \R\to \R$$ 
	 with 
	\begin{enumerate}
	\item $\dfrac{1}{M}\le | D\psi|\le {M}$;
	\item $\dfrac{\epsilon}{M} \le |\psi(0)|\le M\epsilon$;
	\item $\psi_*\bar \muskew_{(\xi,x)}\simeq\bar \muskew_{(\xi,x)}$. 
	\end{enumerate}

Furthermore, writing $$G:= \{ {(\xi,x)}\in X \mid {(\xi,x)}\in G_{1/N} \text{\ for infinitely many $N$}\}$$
we have $\muskew(G)>0$.  
\end{prop}

\begin{remark}\label{rem:posmeasure}
Given the space of locally finite Borel  measures on $\R$, the set of measures satisfying (1)--(3) of Proposition \ref{prop:main} for fixed $\epsilon$ and $M$ is closed.  By restricting to measurable sets on which ${(\xi,x)}\mapsto \bar\mu_{(\xi,x)}$ is continuous,  for any fixed $M$ defining $G_\epsilon$  to be the set of ${(\xi,x)}$ such that $\bar\mu_{(\xi,x)}$ satisfies  (1)--(3) above it follows that $G_\epsilon$ is measurable.  Thus,  the proof of  Proposition \ref{prop:main} reduces to showing that $G_\epsilon$ and $G$ have positive measure for some $M$.  
\end{remark}


\subsection{Proof of Theorem \ref{thm:skewproductABS}}
Theorem \ref{thm:skewproductABS} follows from Proposition  \ref{lem:main} by  standard arguments.  
We sketch these  below and referring to \cite{MR2261075} for more details.  
\def\Aff{\mathcal{A}}
 
 \begin{lemma}\label{lem:translation1}
Under the hypotheses of Proposition \ref{lem:main}, for \ae $(\xi,x)\in X$, $\bar \muskew_{(\xi,x)}$ is invariant under the group of translations.  In particular, for \ae $(\xi,x)\in X$, $\bar \muskew_{(\xi,x)}$ is the Lebesgue measure on $\R$ normalized on $[-1,1]$.  
\end{lemma}

\begin{proof} 
\def\AffR{\mathrm{Aff}(\R)}
Let $\AffR$ denote the group of invertible affine transformations of $\R$.  
For $(\xi,x)\in X$, let $\Aff(\xi,x)\subset \AffR$ be the group of affine transformations $\psi\colon \R\to \R$ with $$\psi_*\bar \muskew_{(\xi,x)} \simeq \bar\muskew_{(\xi,x)}.$$  
We have  that $\Aff(\xi,x)$ is a closed subgroup of the Lie group $\AffR$.  (See the proof of \cite[Lemma 3.10]{MR2261075}.)
By Proposition \ref{lem:main}, for $(\xi,x)\in G$, the group $\Aff(\xi,x)$ contains elements of the form $t\mapsto \lambda_j t+ v_j$ with $v_j\neq 0$, $|v_j|\to 0$ as $j\to \infty$, and $\lambda_j\in \R$ such that  $|\lambda_j|$ is  uniformly bounded away from $0$ and $\infty$.  
Then, for $(\xi,x)\in G$,  $\Aff(\xi,x)$ contains at least one map of the form 
	$$t\mapsto \lambda t$$ 
for some accumulation point $\lambda$ of $\{\lambda_j\}\subset \R$.
We  may thus find  a subsequence of $$\{t\mapsto \lambda \inv\lambda_j t+ v_j\}$$ converging to the identity in $\Aff(\xi,x)$.  It follows that $\Aff(\xi,x)$ is not discrete.  In particular, for every $(\xi,x)\in G$ the group $\Aff(\xi,x)$ contains a one-parameter subgroup of $\AffR$.

\def\Cc{\mathcal C}
For $(\xi,x)\in X$ denote by $\Cc_{(\xi,x)} \colon \R \to \R$ the linear map $$\Cc_{(\xi,x)} = (\I^u_{F(\xi,x)})\inv\circ F \circ \I^u _{(\xi,x)}$$ where $\I^u_{(\xi,x)}$ denotes the parametrization \eqref{eq:affparammfolds}. 
As $(\Cc_{(\xi,x)})_*\bar\muskew_{(\xi,x)} \simeq \bar\muskew_{F(\xi,x)}$ we have that $$\Aff(F(\xi,x)) = \Cc_{(\xi,x)}\Aff(\xi,x) \Cc_{(\xi,x)}\inv.$$ 
Let $\Aff_0(\xi,x)\subset \Aff(\xi,x)$ denote the identity component of $\Aff(\xi,x)$. 
Then $\A_0(F(\xi,x)) $ is isomorphic to $\A_0(\xi,x)$ for \ae $(\xi,x)\in X$.  
Since $\muskew(G)>0$, it follows by ergodicity that $\A_0(\xi,x)$ contains a one-parameter subgroup for \ae $(\xi,x)\in X$.  

The one-parameter subgroups of $\AffR$ are either pure translations or are conjugate to scaling.  We show that $\Aff(\xi,x)$ contains the group of translations for \ae $(\xi,x)\in X$.  
Suppose for purposes of contradiction that $\Aff_0(\xi,x)$ were conjugate to scaling for a positive measure set of $(\xi,x)\in X$. By ergodicity, it  follows that $\Aff_0(\xi,x)$ is conjugate to scaling for \ae $(\xi,x)\in X$.  
For such $(\xi,x)$, there are $t_0\in \R, \gamma\in \R_+$ with 
$$\Aff_0(\xi,x) = \{t\mapsto t_0 + \gamma^s(t-t_0)\mid s\in \R\}.$$
In particular, for such $(\xi,x)$ the action of $\Aff_0(\xi,x)$ on $\R$  contains a unique fixed point $t_0(\xi,x)$. 

For $(\xi,x)\in G$ the fixed point $t_0(\xi,x)$ is non-zero since, as observed above, there are $\psi\in \Aff(\xi,x)$ arbitrarily close to the identity with $\psi(0) \neq 0$.  
Furthermore, writing $\psi\colon t \mapsto t_0(\xi,x) + \gamma^s(t-t_0(\xi,x))$ we have 
$$\Cc _{(\xi,x)} \circ \psi\circ \Cc_{(\xi,x)}\inv\colon t\mapsto \pm\|\restrict{ DF}{E^u(\xi,x)}\|t_0(\xi,x) + \gamma^s\left(t-\pm\|\restrict{ DF}{E^u(\xi,x)}\|t_0(\xi,x)\right)$$ 
where the sign depends on whether or not $\Cc _{(\xi,x)}\colon \R \to \R$ preserves orientation.  
It follows for $(\xi,x)\in G$ that $|t_0(F^n(\xi,x)) |= \|\restrict{ DF^n}{E^u(\xi,x)}\|  \ |t_0(\xi,x)|$ becomes arbitrarily large, contradicting Poincar\'e recurrence.

Therefore, for almost every $(\xi,x)\in X$, the group $\Aff(\xi,x) $ contains the group of translations.  
We finish the proof by showing that for such $(\xi,x)$, the measure $\bar \muskew_{(\xi,x)}$ is \emph{invariant} under the group of translations. 
 For $s\in \R$ define $T_s\colon \R \to \R$ by $T_s\colon t \mapsto t+s$
 and 
define $c_{(\xi,x)}\colon \R \to \R$ by 
$c_{(\xi,x)}(s) = \bar \muskew _{(\xi,x)}([-s-1,-s+1])$.  Then 
$$\frac{d(T_s)_* \bar\muskew_{(\xi,x)}} { d\bar \muskew_{(\xi,x)}} = c_{(\xi,x)}(s) .$$  
As the group $\Aff(\xi,x) $ contains all translations, the measure $\bar\muskew_{(\xi,x)}$ has no atoms   and we have that $c_{(\xi,x)}\colon \R \to \R$ is continuous.

Note that 
$$\Cc_{(\xi,x)} \circ T_s \circ \Cc_{(\xi,x)}\inv = T_{ \pm \|\restrict{ DF}{E^u(\xi,x)}\| s}$$ and for $n\in \Z$
\begin{equation}c_{(\xi,x)}(s) = c_{F^n(\xi,x)}\left( \pm  \|\restrict{ DF^n}{E^u(\xi,x)}\| s \right)\label{eq:bad}\end{equation} where the signs depend on whether or not $DF$ or $DF^n$ preserves the orientation on unstable subspaces.  
Define the set $$B_{r,\epsilon}:= \{(\xi,x)\in X: |c_{(\xi,x)}(s)-1|<\epsilon \text{  for all } |s|<r\}.$$  For each $\epsilon>0$ pick $r$ so that $\muskew(B_{r,\epsilon})>0$.  Applying \eqref{eq:bad}, by ergodicity almost every point visits  $B_{r,\epsilon}$ infinitely often as $n\to -\infty$ contradicting \eqref{eq:bad} unless 
$|c_{(\xi,x)}(s)-1|<\epsilon$ for all $s$ and \ae $(\xi,x)\in X$.  
Taking $\epsilon\to 0$  shows that $c_{(\xi,x)}(s)= 1$ for all $s\in \R $  and \ae $(\xi,x)\in X$ completing the proof of the lemma.
\end{proof}

Theorem \ref{thm:skewproductABS} 
now follows as an immediate corollary of Proposition \ref{lem:main} and  Lemma \ref{lem:translation1}.

\section{Suspension flow}
As in the proof of \cite{MR2831114} and \cite[Sections 15--16]{1302.3320}, to prove Proposition \ref{lem:main} we introduce a suspension flow.  
On the product space $\R \times X$ consider the identification $$(t,\xi,x)\sim \left(t-1,F(\xi,x)\right)= (t-1,\theta(\xi), f_\xi(x))$$ and define  the quotient space $Y = (\R \times X)/\sim.$
We denote by $[t,\xi,x]$ the element of the quotient space $Y$.  
The space $Y$ is equipped with a natural flow 
	$$\Phi^t\colon Y \to Y, \quad \quad \Phi^t([s,\xi,x]) = [s+t,\xi,x].$$  
We have that $\Phi^t$ 
preserves a Borel probability measure $$d \mususp ([t,\xi,x]) = d \muskew(\xi,x) \ d t.$$  

%

It is  convenient to consider the measurable parametrization $[0,1)\times X\to Y$ given by $(\pt, \xi,x)\mapsto [\pt,\xi,x]$.  In these coordinates the  flow $\Phi^t\colon Y \to Y$ is given by 
$$\Phi^t(\pt,\xi,x) = (\{\pt+t\}, F^{\lfloor \pt+t\rfloor}(\xi,x))$$ 
where $ \lfloor \ell \rfloor$ denotes the integer part of $\ell$ and $\{\ell\} = \ell-  \lfloor \ell \rfloor$. 
When we write $(\pt,\xi,x)\in Y$ 
 it is implied that $0\le \pt <1$ and that $(\pt,\xi,x)$ is identified with $[\pt,\xi,x]$.  
Given $(\pt, \xi) \in   [0,1) \times \Omega $ we write $M_{(\pt, \xi)}= \{\pt\}\times \{\xi\} \times M$.
We will also write $\Theta^t\colon [0,1)\times \Omega\to [0,1)\times \Omega$ 
for the induced suspension flow.  

Note that the parametrization $Y = [0,1)\times X$ makes $Y$ into a Polish space with respect to which the measure $\mususp$ is Radon.  To discuss convergence and continuity, we equip $[0,1)$ and $\Omega$ with  complete, separable metrics and endow $Y= [0,1)\times \Omega \times M$ with the product metric.  

We use the parametrization $[0,1)\times X\to Y$ to extend the definition of local and global unstable manifolds. Given $p = (\pt,\xi,x)\in Y$ write 
\begin{itemize}
	\item $\locunstM[r]  x {(\pt,\xi)} = \locunstM[r] x \xi\subset M$; 		 $\unstM  x {{(\pt,\xi)}} = \unstM x \xi\subset M$;
		\item $\locunstp[r]   {p} = \{\pt\} \times \locunst[r] x \xi= \{\pt\} \times \xi\times \locunstM[r] x \xi\subset Y$;
		\item $\unstp {p} = \{\pt\} \times \unst x \xi= \{\pt\} \times \xi\times \unstM x \xi\subset Y$.
\end{itemize}
We similarly extend the definition of local and global stable manifolds, affine parameters, frames for the stable and unstable spaces introduced in Section \ref{sec:frames}, and the induced parametrizations $\I^u$ and $\I^s$.    
Given $p= (\pt,\xi,x)\in Y$, we write $\mususp^u_p= \delta_\pt \times \muskew^u_{(\xi,x)}$ for the locally finite measures on $\unstp {p}$ normalized on $\locunstp[1]   {p} $ and $\bar\mususp_p := (\I^u_p)\inv_*(\mususp^u_p) = \bar \mu_{(\xi,x)}$ for the corresponding measure on $\R$.

	\def\orb{\mathcal O}
	\def\Orb{\orb}
	\def\Morb{\mathcal M}
	\def\Uorb{\mathcal U}
	

Although the flow $\Phi^t\colon Y \to Y$ is, at best, measurable, the restriction $$\Phi^t \colon M_{(\pt, \xi)} \to M_{\Theta^t\left(\pt, \xi\right)}$$
is a $C^2$-diffeomorphism.  
Define a fiber-wise tangent bundle 
$$ TY:= [0,1) \times TX = [0,1) \times \Omega \times TM$$
and the fiber-wise differential $D\Phi^t\colon TY\to TY$
$$D\Phi^t\colon \left(\pt,\xi, (x,v)\right) \mapsto \left(\{\pt+ t\}, \theta^{\lfloor \pt+t\rfloor}(\xi), \left(f_\xi^{\lfloor \pt+t\rfloor}(x), D_x f_\xi^{\lfloor \pt+t\rfloor}(v)\right) \right).$$
We trivially extend  norms on $TM$ to $TY$  
by identifying $\{(\pt,\xi)\} \times T_xM$  with $T_xM$.

\def\nucondF{\nu^{\hat\Fol}}
\section{Preparations for the proof of Proposition \ref{lem:main}}\label{sec:lemmas}\label{sec:lemMainPrep}
We begin with a number of constructions and technical lemmas that will be used in the proof of  Proposition \ref{lem:main}. 


\subsection{Modification of $\hat \Fol$}\label{sec:modifyF}
Recall we assume the function  $\xi\mapsto \cocycle    [\xi] [-1]$ is $\hat \F$-measurable  which implies the entire past dynamics $\xi\mapsto \cocycle    [\xi] [n]$ is $\hat \Fol$-measurable for all $n\le 0$.  It is convenient for technical reasons below  to allow the  first future iterate $\cocycle[\xi][] $ to be measurable on $\hat \Fol$ as well.  As 
$f_\xi = \left(\cocycle [\theta (\xi)][-1]\right)\inv$, this can be accomplished by replacing $\hat \Fol$ with $\theta\inv(\hat \Fol)\subset \hat \Fol$.  Then $\theta\inv(\hat \Fol)$ is a decreasing sub-$\sigma$-algebra for which 
$\xi\mapsto \cocycle    [\xi] [n]$ is measurable for all $n\le 1$.  Moreover, as $f_\xi$ is constant on atoms of $\theta\inv(\hat \Fol)$, we have 
that $$(\xi,x)\mapsto E^s_{\xi}(x)$$ is $\F$-measurable if and only if $(\xi,x)\mapsto E^s_{\xi}(x)$ is $F\inv(\F)$-measurable.  

Thus for the remainder, we replace $\hat \Fol$ and $\Fol$ with $\theta\inv(\hat \Fol)$ and $F\inv (\Fol)$ respectively.  With this  new notation,  we then have that  $\xi\mapsto \cocycle    [\xi] [n]$ is $\hat \Fol$-measurable for all $n\le 1$. 

\subsection{Lyapunov norms}\label{sec:lyapnorm}
From Lemma \ref{prop:growth}, for each $p\in Y$ we  observe the hyperbolicity of the cocycle $D\Phi^t$ after a finite amount of time.  
We define here two norms, called \emph{Lyapunov norms}, with respect to which the hyperbolicity of $D\Phi^t$ is seen immediately.  
 We remark that while the induced Riemannian norm $\|\cdot\|$ on $TY$ is constant in the first parameter of the parametrization $[0,1)\times X \to Y$, the Lyapunov norms defined below will vary 
 in the parameter $\pt$.


We first define the Lyapunov norms for the  skew product $F$ on $X$.  
For  $(\xi,x)\in X$, $\sigma \in\{s,u\}$, and $v\in E^\sigma_\xi(x)$ 
define  the \emph{two-sided Lyapunov norm}
\begin{align}
&\lyap{v}_{\eps, \pm, (\xi,x)}^\sigma:= \left(\sum_{n\in \Z}   \|Df^n_\xi v\|^2 \rexp{ -2 \lambda^\sigma n -  2 \eps |n|}\right)^{1/2}
 \end{align}
and the past \emph{one-sided   Lyapunov norm}
\begin{align}
&  \lyap{v}_{\eps, -, (\xi,x)}^\sigma := \left(\sum_{n\le0}   \|Df^n_\xi v\|^2 \rexp{ -2 \lambda^\sigma n -  2 \eps |n|}\right)^{1/2}.  
\end{align}
It follows from Lemma \ref{prop:growth} that the sums above converge for almost every $(\xi,x)\in X$.  Observe that for $v\in E^\sigma_\xi(x),$  $$\|v\|  \le \lyap{v}_{\eps, -, (\xi,x)}^\sigma,\quad 
\|v\|\le   \lyap{v}_{\eps, \pm, (\xi,x)}^\sigma.$$


\begin{remark}\label{rem:onesided} 
Recall that we have 
$\xi \to \cocycle[\xi] [-n] $ is $\hat \Fol$-measurable for all $n\ge 0$ whence the assignment $(\xi,x) \mapsto E^u_\xi(x)$ is $\Fol$-measurable. 
Recall the $\Fol$-measurable family of vectors $v^u_{(\xi,x)}\in E^u_\xi(x)$ built in Section \ref{sec:frames}.  It follows from construction that   $(\xi,x)\mapsto \lyapno{v^u_{(\xi,x)}}_{\eps, -, (\xi,x)}^u$ is $\Fol$-measurable. 
Moreover, as discussed in Section \ref{sec:modifyF}, since we assume $\xi \mapsto \cocycle[\xi] [] $ is $\hat \Fol$-measurable, we have that $(\xi,x)\mapsto \lyapno{D_x\cocycle[\xi][] v^u_{(\xi,x)}}_{\eps, -, F(\xi,x)}^u$ is $\Fol$-measurable. 
This will be the primary  reason for using the one-sided Lyapunov norms rather than two-sided Lyapunov norms below.  
\end{remark}

We have the following bounds on hyperbolicity.  For the one-sided norm $ \lyap{\cdot}_{\eps, -, (\xi,x)}^\sigma $, the bounds are of most use when $\sigma = u$.  (One can similarly define the future one-sided norm  $ \lyap{\cdot}_{\eps, +, (\xi,x)}^s $ which is more natural for the stable bundle.)

\begin{lemma}\label{prop:LyapNormPropsDiscrete}\label{prop:discLyap}
For $\muskew$-\ae $(\xi,x)\in X$, 
$v\in E^\sigma_\xi(x),$ 
 $n\in \Z $, and $k\ge 0$ we have 

\begin{enumerate}
\item $\rexp{n\lambda^\sigma -  |n|\eps} \lyap v_{\eps, \pm, (\xi,x)}^\sigma\le 
	 \lyap{Df^n_\xi v}_{\eps, \pm, F^n(\xi,x)}^\sigma
	 \le \rexp{n\lambda^\sigma +  |n|\eps} \lyap v_{\eps, \pm, (\xi,x)}^\sigma$
\item $  \rexp{k\lambda^\sigma  -k\eps } \lyap v_{\eps,-, (\xi,x)}^\sigma \le  \lyap{Df^k_\xi v}_{\eps, -, F^k(\xi,x)}^\sigma.$
\end{enumerate}
\end{lemma}

\subsubsection{Extensions to $Y$}\label{sec:lyaponY}
We extend the Lyapunov norms to $TY$ as follows:  For $p = (\pt,\xi,x)\in Y$ and $w\in E^\sigma_\xi (x)= E^\sigma_{(\pt,\xi)}(x)\subset T_xM$, 
define 
\begin{equation}\label{eq:extendtoY}{\lyap {w} _{\eps, -, p}^\sigma = \left(\lyap {w} _{\eps, -, (\xi,x)}^\sigma \right)^{1-\pt}\left(\lyap {D_x f_\xi w} _{\eps, -, F(\xi,x)}^\sigma \right)^{\pt}}.\end{equation}
Identifying   $E^\sigma(p)= \{(\pt, \xi)\} \times E^\sigma_\xi(x)\subset TY$ with $E^\sigma_{(\pt,\xi)}(x)\subset TM$, we extend the definition of $\lyap {\cdot} _{\eps, -, p}^\sigma$ to $E^\sigma(p)$.  
We similarly extend 
the two-sided Lyapunov norms to $TY.$ 

Given $t\in \R$ we write 
$$\lyap {\restrict{D\Phi^t}{E^\sigma(p)}}_{\eps, -}, \quad \quad \lyap {\restrict{D\Phi^t}{E^\sigma(p)}}_{\eps, \pm}$$
to indicate the operator norm of $D\Phi^t\colon E^\sigma (p) \to E^\sigma(\Phi^t(p))$ with respect to the corresponding norms.  

We have the following extension of Lemma \ref{prop:discLyap}. 

	\begin{lemma}\label{prop:LyapNormProps}
				For $\mususp$-\ae $p = (\pt,\xi,x)\in Y,$ 
				$v\in E^\sigma(p)$,  $t\in \R$, and $s\ge 0$
				we have 

				\begin{enumerate}
				\item \label{prop:Lyap1} $\rexp{t\lambda^\sigma -  |t|\eps} \lyap v_{\eps, \pm,p}^\sigma\le
					 \lyap{D\Phi^t v}_{\eps,\pm, \Phi^t(p)}^\sigma	 \le \rexp{t\lambda^\sigma +  |t|\eps} \lyap v_{\eps, \pm, p}^\sigma,$ 
				
				\item \label{prop:Lyap3} $\rexp{s\lambda^\sigma -s\eps} \lyap v_{\eps,-, p}^\sigma\le  \lyap{D\Phi^s v}_{\eps, -, \Phi^s(p)}^\sigma.$
				
				\end{enumerate}
\end{lemma}
We  have the following estimate which allows us to compare the Lyapunov norm with the induced Riemannian norm. 
Recall the functions $D\colon \Omega \to \R$ and $L\colon X\to \R$ in Lemmas \ref{prop:tempered} and \ref{prop:growth}. Let  $c_1 = e^{\eps}(1-e^{-\eps})^{1/2}.$  
 \begin{lemma}\label{prop:subexpcomp}
For any $w\in E^u(\pt,\xi,x)$, 
$$ \|w\|\le  \lyap w ^u _{\eps, - , p}\le  L(\xi,x) D(\xi) c_1\|w\|.$$

In particular, defining $\hat L\colon Y\to [1,\infty)$ by $$\hat L(\pt,\xi,x)= L(\xi,x) D(\xi)c_1$$
we have 
	$$\hat L(\Phi^t(p))\le \rexp{2\eps(|t|+1)}\hat L(p)$$
and 
\begin{equation}\label{eq:subcomp}
\hat L(p)\inv \|\restrict {D\Phi^t}{E^u(p)}\|\le  \lyap{\restrict {D\Phi^t}{E^u(p)}} ^u _{\eps, - }\le \rexp{2\eps(|t|+1)}\hat L(p) \|\restrict {D\Phi^t}{E^u(p)}\|.
\end{equation}
\end{lemma}
\noindent (A similar estimate holds for the two-sided norms.)

\begin{proof} Recall that for $w\in E^u_\xi(x)$, we have $\|w\| \le \lyap w ^u _{\eps, - , (\xi,x)}$ and  $$\lyap w ^u _{\eps, - , (\xi,x)}<\lyap{D_x\cocycle[\xi][] w}^u_{\eps, - , F(\xi,x)}       .$$  The lower bound then follows.  

For the upper bound  we have for $(\xi,x)\in X$ and $w\in E^u_\xi(x)$ that 
		\begin{align*}\lyap{w}^u_{\eps, - ,(\xi,x)} 
		&\le \left(
			\sum_{n\le 0}(L(\xi,x))^2 \|w\|^2 \rexp{2n \lambda^u + |n|\eps} \rexp{-2n\lambda^u - 2 \eps|n|} 	
		\right)^{1/2}\\&
		=L(\xi,x) \left(1-\rexp{-\eps} \right)^{-1/2}\|w\|.\end{align*}
Similarly, from Lemma \ref{prop:growth} we have 
$$\lyap{D_xf_\xi w}^u_{\eps, - ,(\xi,x)} 
\le L(\xi,x) \rexp{\eps} \left(1-\rexp{-\eps} \right)^{-1/2}\|D_xf_\xi w\|.$$
Then for $p= (\pt,\xi,x)\in Y$ and $w\in E^u(p)$, with $b =  \left(1-\rexp{-\eps} \right)^{-1/2}$ we have 
\begin{align*}
\lyap{w}^u_{\eps, - ,p} 
	&:= {\left(\lyap{w}^u_{\eps, - ,(\xi,x)}\right)^{1-\pt} \left(\lyap{D_xf_\xi w}^u_{\eps, - ,F(\xi,x)}\right)^{\pt}}\\
	&\le \left(L(\xi,x)b\|w\| \right)^{1-\pt}   \left(L(\xi,x)\rexp{\eps}b \|D_xf_\xi w\|\right)^{\pt}\\
	&={\left(L(\xi,x)b \|w\| \right)^{1-\pt}   \left(L(\xi,x)\rexp{\eps}b D(\xi) \| w\|\right)^{\pt}}\\
	& \le L(\xi,x) b \rexp{\eps} D(\xi) \| w\|.		\qedhere
\end{align*}
\end{proof}

Declaring that $ E^u(p)$ and $E^s(p)$ are orthogonal, we may extend the definitions of both the two-side and one-side Lyapunov norms to 
 all of $T_pY$.  
When clear from context, we will drop the majority of  sub- and superscripts from the Lyapunov norms.

\subsection{The  time changed flow}\label{sec:timechange}

\def\ess{\mathscr S}
\def\tee{\mathscr T}
\def\enn{\mathscr N}
 
It is convenient 
 to work with a flow $\Psi^s$ that is a time change of $\Phi^t$ and for which the norm of the restriction of $D\Psi^s$ to the unstable spaces  grows at a constant rate (with respect to the one-sided norm $\lyap {\cdot} ^u_{\eps, -}$.)  
 
For $p\in Y$ and $t\in \R$, define  
\begin{equation}\label{eq:timechange}
\ess_p(t) = \log\left(\lyap {\restrict{D\Phi^{t}}{E^u(p)} }_{\eps, -}^u \right) .
\end{equation}
It follows from 
construction and Lemma \ref{prop:LyapNormProps} 
 that, for $\mususp$-\ae $p\in Y$, the function  
$\ess_p\colon \R \to \R$ is an orientation-preserving homeomorphism.  
Moreover, as $E^u(\xi,x)$ is 1-dimensional, the map $Y\times \R \to \R$ given by $(p,t) \to \ess_p(t)$
satisfies the cocycle equation 
$\ess_p(t_1+ t_2) = \ess_{\Phi^{t_2}(p)}( t_1) + \ess_p(t_2) $. 
It follows that $\Psi^s\colon Y\to Y$ given by $$\Psi^s(p) = \Phi^{\ess_p\inv(s)}(p)$$ defines a measurable flow on $Y$ that is a time change of $\Phi^t$.

Given $p= (\pt,\xi,x)\in Y$ define \begin{equation}h(p)= h(\xi,x) =  \log \left(\lyap {\restrict{D_xf_\xi }{E^u_\xi(x)} }_{\eps, -, {(\xi,x)}}^u\right).\label{eq:emptyrooftop}\end{equation}
We note that for $- \pt\le t<1-\pt$
$$\log \lyap {\restrict {D\Phi^t}{E^u(p)}}_{\eps, -, p} = t h(p).$$  In particular, if $0\le \pt + s/ h(p)<1$ then $\Psi^s(p) =  (\pt + s/ h(p), \xi,x)$; that is, $h(p)\inv$ is the local change of speed of the original flow $\Phi^t$.    
It follows that $$\ess_p(t) = \int _0 ^t h(\Phi^s (p)) \ ds .$$
By \eqref{eq:subcomp}, Lemma \ref{prop:growth}, and the fact that  $h(p)\ge\lambda^u-\eps$ for almost all $p$, we have for any $t\ge 0$ that 
 \begin{equation}\label{eq:controlled} (\lambda^u-\eps) t\le \ess_p(t)\le a(p) + b_0(t+1)\end{equation}
 where 
 $$a(\pt,\xi,x) = 
 \log (L(\xi,x)^2 D(\xi)c_1 ), \quad \quad b_0 = \lambda^u+3\eps.$$

We 
claim \begin{claim} \label{claim:finmeas}
$\displaystyle\int  h(\xi,x) \ d \muskew(\xi,x) <\infty.$
\end{claim}
\begin{proof}

Consider $0\neq w\in {E^u_\xi(x)}$. 
  We have 
\begin{align*}
 	\left(\lyap{D_xf_\xi w}_{\eps,-, F(\xi,x)}^u\right)^2
		&= 
		   \|Df_\xi w\|^2 
		 + \sum_{n \le -1}   \|Df^{n+1}_\xi w\|^2 
	 	\rexp{-  2\lambda^un  - 2 \eps|n|}
		\\
		&=
 \|D_xf_\xi w\|^2		 
	 +\rexp{2\lambda^u   -2 \eps  }
	\left(  \sum_{\ell\le 0}   \|Df^{\ell}_\xi w\|^2 
	 	\rexp{-  2\lambda^u \ell - 2 \eps |\ell|}\right)\\
		&=
 \|D_xf_\xi w\|^2		 
	 +\rexp{2\lambda^u   -2 \eps  }\left(\lyap{ w}_{\eps,-, (\xi,x)}^u\right)^2
\end{align*}
and since  $\|w\|^2 \le \left(\lyap{ w}_{\eps,-, (\xi,x)}^u\right)^2$, we have
\begin{align*}	
\frac{ 	\left(\lyap{D_xf_\xi w}_{\eps,-, F(\xi,x)}^u\right)^2}{\left(\lyap{ w}_{\eps,-, (\xi,x)}^u\right)^2}
		\le 
 \|\restrict{Df_\xi}{E^u_\xi(x)} \|^2		 
	 +\rexp{2\lambda^u   -2 \eps  }
	 		\le 
 \|{Df_\xi} \|^2		 
	 +\rexp{2\lambda^u   -2 \eps  }.
\end{align*}


Recall 
 $\int \log^+\left(|f_\xi|_{C^1} \right) \ d \nu(\xi)<\infty$ by hypothesis \eqref{eq:IC2}. 
The claim follows as
\begin{gather*}\int \log\left( |f_{\xi}|_{C_1}^2   +\rexp{2\lambda^u   -2 \eps  } \right) \ d \nu(\xi)< \infty. \qedhere\end{gather*}
\end{proof}

From Claim \ref{claim:finmeas} it follows that 
$\Psi^s\colon Y\to Y$ preserves a probability measure $ \mualt$ given by $$ d  \mualt (\pt,\xi,x) =
\tfrac 1 {\int h(\xi,x) \ d\muskew(\xi,x)} h(\xi,x)  \ d\muskew(\xi,x) \ d \pt.$$ Observe that $\mualt$ and $\mususp$ are equivalent measures.  
Furthermore, since the $\sigma$-algebras of $\Phi^t$- and $\Psi^s$-invariant sets coincide it follows that $\Psi^s$ is $\mualt$-ergodic.  

\subsection{Decreasing subalgebras, conditional measures, and the martingale convergence argument}\label{sec:MCT}


We write  $\Sal\subset \B_Y$ and $\hat \Sal\subset \B_{[0,1)\times \Omega}, $ 
respectively, for the completions of $\B_{[0,1)}\otimes  \Fol$ and  $\B_{[0,1)}\otimes \hat \Fol$. 
Note that we have $$\Phi^t(\Sal ) \subset \Sal, \quad \quad \Theta^t (\hat \Sal) \subset \hat \Sal$$ for all $t\ge 0$ whence $\Sal$ and $\hat \Sal$ are  decreasing $\sigma$-algebras for the respective flows. 
In particular, the map
\begin{equation}Y\times [0,\infty)\to \left(Y,\Sal, \mususp \right),\quad \quad (p,t) \mapsto \Phi^{-t}(p)\label{eq:xyz9}\end{equation}
is $\Sal \otimes \B_{[0,\infty)}$-measurable. 
Thus the backwards flow $\Phi^{-t}, t\ge 0$ induces a measurable semi-flow on the factor space  $\left(Y,\Sal, \mususp \right)$.  

As discussed in Remark \ref{rem:onesided} 
the past dynamics $\xi \mapsto \cocycle$, $n\le 0$ is $\hat \Fol$-measurable and thus the  unstable spaces $E^u_\xi(x)$ and family of one-sided norms $\lyap {\cdot}_{\eps, -, (\xi,x)}^u$ are $\Fol$-measurable.   Furthermore, as the extension of the norms $\lyap {\cdot}_{\eps, -, (\xi,x)}^u$ to $Y$ in \eqref{eq:extendtoY} involves only the past dynamics and a single future iterate $\cocycle [\xi][]$, 
it follows that the 
  family of one-sided norms $\lyap {\cdot}_{\eps, -,p}^u$ on $Y$ are $\Sal$-measurable.  
 It follows that, restricted to the past, the  cocycle defining the time change 
$$Y\times [0,\infty)\to[0,\infty),\quad \quad (p,t) \mapsto \ess_p(-t)$$ 
is $\Sal\otimes \B_{[0, \infty)}$-measurable. 
Thus, the backwards time-changed flow $\Psi^{-s}, s\ge 0$, given by $\Psi^{-s}(p) = \Phi^{\ess_p\inv(-s)}(p)$, induces a measurable semi-flow  on $\left(Y,\Sal, \mususp \right)$.  
In particular, $\Psi^s(\Sal) \subset \Sal$ for  $s\ge 0$.   

Given $m\in \R$ define the sub-$\sigma$-algebra on $Y$ by  $$\Sal^m:= \Psi^{m}(\Sal)= \{ \Psi^m(C)\colon C\in \Sal\}.$$
From the above discussion, we have the following.  
\begin{claim}
For $m\le \ell$ we have 
$\Sal^\ell\subset \Sal^m.$ 
In particular, $\{\Sal^m\}_{m\ge 0}$ defines a decreasing filtration on $(Y, \mualt)$.  
\end{claim}
As usual, we write $\Sal^\infty = \bigcap_{m=0}^\infty \Sal^m$.

\def\mucondS{\mususp^{\Sal}}   
\def\mualtcondS{\mualt^{\Sal}}
\def\thisfoot{\footnote{  Recall that given a sub-$\sigma$-algebra $\calA$ of a Lebesgue probability space $(\Omega, \mathcal{B}, \mu)$, there is a unique (up to \as equivalence)  measurable partition $\alpha$, called the partition into atoms.  If $\{\mu^\alpha_\omega\}$ denotes a family of conditional measures induced by the partition $\alpha$ then $\Exp(f \mid \calA)(\omega) = \int f \ d \mu^\alpha_\omega$.   }
}
\subsubsection{Families of conditional measures}\label{sec:CondMeas}
We fix once and for all a measurable partition\thisfoot of $\Omega$ into atoms of $\hat \F$ and  an induced family of conditional probabilities  $\{\nucondF_\xi\}_{\xi \in \Omega}$. 
Since in Theorem \ref{thm:skewproductABS} 
 we assume the map  $\xi\mapsto \muskew_\xi$ is $\hat \F$-measurable, 
defining a family of measures  $\{\mu^\Fol_{(\xi,x)}\}_{(\xi,x)\in X}$  by 
$$d \mu^\Fol_{(\xi,x)}(\eta,y) :=d{\nucondF_\xi}(\eta) \delta_x(y),$$ it follows that $\{\mu^\Fol_{(\xi,x)}\}$ defines a family of conditional measures induced by $\Fol$.  For $p= (\pt, \xi,x)\in Y$ we write $\mucondS_p = \delta_\pt \times \mu^\Fol_{(\xi,x)}$.  Then $\{\mucondS_p\}_{p \in Y}$ is a family of 
conditional measures of $\mususp$ induced by $\Sal$.  

By a slight abuse of notation, for  $p= (\pt,\xi,x)$ we may consider $ \mucondS_p$ as measures on $\Omega$ by declaring 
$$d \mucondS_p(\eta) = d \mucondS_p(\pt, \eta, x).$$
This identifies $\mucondS_p(\eta)$ with $\nucondF_\xi$. In particular, if 
$q= (\pt,\xi, y)$, then under this identification we have 
$\mucondS_p  =  \mucondS_q.$

Recall that $\mususp$ and $\mualt$ are equivalent measures; moreover $\dfrac{d\mualt}{d\mususp}(\pt,\xi,x)= \dfrac 1 {\int h \ d \mususp} h(\pt, \xi,x) $ where $h$ is the speed change in \eqref{eq:emptyrooftop}.  
Thus,   defining 
$$d \mualt ^\Sal_p(q) := \dfrac{h(q)}{\int h \ d\mususp^\Sal_p} \ d \mususp^\Sal_p(q)$$ it follows that
$\{ \mualt ^\Sal_p\}_{p\in Y}$ defines a family of conditional measures for $\mualt$ induced by   $\Sal$.  As
$h\colon Y \to \R$ is $\Sal$-measurable, we may take   $$\mualt ^\Sal_p = \mususp ^\Sal_p.$$

\subsubsection{Martingale convergence argument} \label{subsec:MCT}
Consider any bounded measurable $g\colon Y \to \R$.  As $\mualt $ is $\Psi^s$-invariant, 
we have  for $m\in \R$
\begin{align}\label{eq:crazymgale1}
\int g (\Psi^{m}(q))\ d\left(\mususp^\Sal_{\Psi^{-m}(p)}\right)(q)&= 
\int g (\Psi^{m}(q))\ d\left(\mualt^\Sal_{\Psi^{-m}(p)}\right)(q) 
\\&= \int g (q') \ d \left(\left(\Psi^m\right)_* \mualt^\Sal_{\Psi^{-m}(p)}\right) (q')\label{eq:crazymgale3}
\\&\circeq  \Exp _{\mualt} (g \mid \Sal^m)(p)\label{eq:crazymgale2}.
\end{align}
where the first two equalities hold everywhere by definition and the last equality holds as almost-everywhere defined functions.  

The right-hand side of \eqref{eq:crazymgale2} defines a reverse Martingale with respect the  decreasing filtration $\Sal^m$ on $(Y, \mualt)$.  
By the convergence theorem for reverse martingales, along any discrete subsequence of $m_j\in [0,\infty)$
 we have, almost surely, that  $$\Exp (g\mid \Sal^{m_j})(p)\to \Exp (g \mid \Sal^\infty)(p) .$$ On the other hand, given any $m$ and any $p\in Y$,  writing $\Psi^{-m}(p) = (\pt,\xi,x)$, for all $\epsilon< (\lambda^u-\eps)\inv (1-\pt)$
$$(\Psi^\epsilon )_* \mualt^\Sal_{\Psi^{-m}(p)} = \mualt^\Sal_{\Psi^{-m+\epsilon}(p)}.$$ 
It follows that the sample paths defined by \eqref{eq:crazymgale3} are constant on half-open intervals whose lengths are at least $(\lambda^u-\eps)\inv$.  Taking a discrete subgroup with gaps less than $(\lambda^u-\eps)\inv$ it follows that for  almost every $p\in Y$, the left-hand side of  \eqref{eq:crazymgale1} converges to 
$\Exp (g \mid \Sal^\infty)(p) $ as $m\to \infty$.

\subsection{Stopping times and bi-Lipschitz estimates}\label{sec:stopping}
Given $p=(\pt,\xi,x)\in Y$, $\delta>0$, $\epsilon>0$ and $m\in \R$ define 
	$$\tau_{p, \delta, \epsilon}(m): = \sup \left\{ \ell\in \R : \lyap {\restrict {D\Phi^m}{E^s(p)}}_{\eps, \pm, p}^s \cdot  \lyap {\restrict {D\Phi^\ell}{E^u(\Phi^m(p))}}_{\eps, \pm,\Phi^m(p)}^u\delta\le \epsilon\right\}$$
and $$L_{p, \delta, \epsilon}(m) = m + \tau_{p, \delta, \epsilon}(m).$$

Note that 	$\tau_{p, \delta, \epsilon}\colon \R \to \R$ and $L_{p, \delta, \epsilon}\colon \R \to \R$ are increasing homeomorphisms.  In fact we have the following.  
\begin{lemma}\label{lem:bilipest}
$L_{p, \delta, \epsilon}$ and $\tau_{p, \delta, \epsilon}$ are bi-Lipschitz with  constants uniform in $p, \delta, \epsilon$.  In particular, for $\ell \ge0$ 
\begin{align}
\label{eq:bilip1}{\dfrac {-\lambda ^s - 3 \eps}{\lambda^u + \eps} \ell}&\le \tau_{p, \delta, \epsilon}(m+\ell) - \tau_{p, \delta, \epsilon}(m) \le {\dfrac {-\lambda ^s +3 \eps}{\lambda^u -  \eps} \ell } \\
 \dfrac{  \lambda^u-\lambda^s-2\eps}{ \lambda^u+\eps } \ell&\le L_{p, \delta, \epsilon}(m+\ell) - L_{p, \delta, \epsilon}(m) \le  \dfrac{  \lambda^u-\lambda^s+2\eps}{ \lambda^u- \eps } \ell\label{eq:bilip2}
\end{align}
\end{lemma}

\begin{proof} 

Write $\tau_p = \tau_{p, \delta, \epsilon}$. 
By definition we have 
\begin{equation}\begin{aligned}\label{eq:QI}
&\lyap{ \restrict{D\Phi^{\tau_p(m+\ell)}}{\Eup {\Phi^{m+\ell}(p)}}}_{\eps,\pm} \cdot \  \lyap{\restrict{D\Phi^{\tau_p(m)}}{\Eup {\Phi^{m+\ell}(p)}}}^{-1}_{\eps,\pm}   
\\& \quad  \quad   \cdot  \   
 \lyap{\restrict{D\Phi^{\tau_p(m)}}{\Eup {\Phi^{m+\ell}(p)}}}_{\eps,\pm}   
   \cdot  \   
   \lyap{\restrict{D\Phi^{m+\ell}}{\Esp p}}_{\eps,\pm} \ \cdot \delta 
=\epsilon.
  \end{aligned}\end{equation}

As $\tau_p(m+\ell) \ge \tau_p(m)$, we  bound the product of the first two  terms of the left-hand side of \eqref{eq:QI} by
 \begin{align*}
&\exp((\lambda^u-\eps)(\tau_p(m+\ell) - \tau_p(m))) \\
&\le \lyap{ \restrict{D\Phi^{\tau_p(m+\ell)}}{\Eup {\Phi^{m+\ell}(p)}}}_{\eps,\pm} \cdot \  \lyap{\restrict{D\Phi^{\tau_p(m)}}{\Eup {\Phi^{m+\ell}(p)}}}^{-1}_{\eps,\pm}
\\ &\le \exp((\lambda^u+\eps)(\tau_p(m+\ell) - \tau_p(m)) )
\end{align*}

To bound the remaining terms of \eqref{eq:QI} 
first note that 
  \begin{align*}
&\lyap{\restrict{D\Phi^{\tau_p(m)}}{\Eup {\Phi^{m+\ell}(p)}}}_{\eps,\pm}   \\  
&\quad= \lyap{\restrict{D\Phi^{ \tau_p(m)  +\ell}}{\Eup {\Phi^{m}(p)}}}_{\eps,\pm}  \cdot 
\lyap{\restrict{D\Phi^{\ell }}{\Eup {\Phi^{m}(p)}}}_{\eps,\pm} \inv\\  
&\quad =\lyap{\restrict{D\Phi^{\tau_p(m)}}{\Eup {\Phi^{m}(p)}}}_{\eps,\pm}  \cdot
\lyap{\restrict{D\Phi^{ \ell }}{\Eup {\Phi^{m+ \tau_p(m)}(p)}}}_{\eps,\pm}  \cdot 
\lyap{\restrict{D\Phi^{\ell }}{\Eup {\Phi^{m }(p)}}}_{\eps,\pm} \inv.
  \end{align*}
We have 
  $$
   e^{-2\eps \ell}  \le 
\lyap{\restrict{D\Phi^{ \ell }}{\Eup {\Phi^{m+ \tau_p(m)}(p)}}}_{\eps,\pm}  \cdot 
\lyap{\restrict{D\Phi^{\ell }}{\Eup {\Phi^{m }(p)}}}_{\eps,\pm} \inv  \le e^{2\eps \ell}.$$
whence it follows that 
  $$
 e^{-2\eps \ell} \le \dfrac{  \lyap{\restrict{D\Phi^{\tau_p(m)}}{\Eup {\Phi^{m+\ell}(p)}}} _{\eps,\pm} }
 { \lyap{\restrict{D\Phi^{\tau_p(m)}}{\Eup {\Phi^{m}(p)}}} _{\eps,\pm} }
 \le  e^{2\eps \ell}.$$
As  
\begin{align*} &\exp((\lambda^s - \eps)\ell) \epsilon
   \\& \le \lyap{\restrict{D\Phi^{\tau_p(m)}}{\Eup {\Phi^{m}(p)}}}_{\eps,\pm}   
   \cdot  \   
   \lyap{\restrict{D\Phi^{m+\ell}}{\Esp p}}_{\eps,\pm} \ \cdot \delta 
\\ &\le \exp((\lambda^s + \eps)\ell) \epsilon\end{align*}
we have 
\begin{align*} &\exp((\lambda^s -3 \eps)\ell) \epsilon
   \\& \le \lyap{\restrict{D\Phi^{\tau_p(m)}}{\Eup {\Phi^{m+\ell}(p)}}}_{\eps,\pm}   
   \cdot  \   
   \lyap{\restrict{D\Phi^{m+\ell}}{\Esp p}}_{\eps,\pm} \ \cdot \delta. 
\\ &\le \exp((\lambda^s +3 \eps)\ell) \epsilon.\end{align*}

Reassembling \eqref{eq:QI} we  have 
\[
\exp\big((\lambda^u-\eps)(\tau_p(m+\ell) - \tau_p(m))\big)
\exp((\lambda^s - 3\eps)\ell) \epsilon\le \epsilon\]
and 
\[\exp\big((\lambda^u+\eps)(\tau_p(m+\ell) - \tau_p(m))\big) 
\exp((\lambda^s + 3\eps)\ell)\epsilon \ge  \epsilon\] 
hence 
\[
\dfrac{-\lambda ^s-3 \eps}{\lambda^u+\eps} \ell 
	\le 
\tau_p(m+ \ell) - \tau_p(m)
	\le 
\dfrac{-\lambda ^s+ 3\eps}{\lambda^u-\eps} \ell\]
proving \eqref{eq:bilip1}.


We derive \eqref{eq:bilip2} from \eqref{eq:bilip1} noting
\begin{gather*}
\dfrac{-\lambda ^s- 3\eps}{\lambda^u+\eps} \ell + \ell
	\le 
L(m+\ell)-L(m) 
	\le 
\dfrac{-\lambda ^s+ 3\eps}{\lambda^u-\eps} \ell +\ell
.	\qedhere\end{gather*}

\end{proof}

Let $\leb$ denote the Lebesgue measure on $\R$.  We have the following fact.  
\begin{claim}\label{claim:bilipRND}
Let $g\colon \R \to \R$ be a bi-Lipschitz homeomorphism  with $$a|y-x|\le |g(y) - g(x)| \le b|y-x|.$$
Then 
$g_*\leb\ll\leb\ll g_*\leb$ and 
$a\le\frac {d\leb}{dg_*\leb}\le b$.
\end{claim}

 \def\rhere{r_1}

\subsection{Dichotomy for invariant subspaces}
For this  and the following subsection, we return to the skew product $F\colon X\to X$.  In this section we establish the following dichotomy for $DF$-invariant subbundles of $TX$. 
Let $F\colon X\to X$  and $\muskew$ be as in Theorem \ref{thm:skewproductABS}.  
\def\V{\mathcal V}
Consider a $\muskew$-measurable line field $\V\subset TX$.  Write $V_\xi(x)\subset T_xM$ for the family of subspaces with $$\V(\xi,x) = \{\xi\}\times    V_\xi(x) .$$   
The measurability of $\V$ with respect to a sub-$\sigma$-algebra of $X$ is the measurability of the function $(\xi,x)\mapsto V_\xi(x)$ with the standard Borel structure on $TM$.
We say  $\V$ is \emph{$DF$-invariant} if for $\muskew$-\ae $(\xi,x)\in X$
\[DF_{(\xi,x)} \V(\xi,x) = \V(F(\xi,x))\text {\ or \ }   D_xf_\xi V_\xi(x) = V_{\theta(\xi)}( f_\xi(x)).\]
		
Recall that 
$\F$  in Theorem \ref{thm:skewproductABS} is an decreasing sub-$\sigma$-algebra; that is,   $F(\F) \subset \F$.  
We write $\F_{\infty}$ 
for the smallest $\sigma$-algebra containing $\bigcup_{n\ge 0}F^{-n}(\F)$.  We  similarly define  $\hat \F_{\infty}$.  
(We remark that in the case that $\hat \Fol$ is the $\sigma$-algebra of local unstable sets in Section \ref{sec:skewreint}, $\hat\Fol_\infty$ and $\Fol_\infty$ are, respectively,  the completions of the Borel algebras on $\Sigma$ and $\Sigma\times M$.) 
		

Recall we write $\{\nucondF_\xi\}_{\xi \in \Omega}$ for a family of  conditional probabilities  induced by $\hat \Fol$. 

\begin{lemma}\label{lem:VFdichot}
		Let $\muskew$ and $\scrF$ be as in Theorem \ref{thm:skewproductABS}.  Then 
	\begin{enumerate}[label=({\arabic*}), ref= ({\arabic*})]
		\item \label{dich1}  the line field $(\xi,x)\mapsto E_\xi^s(x)$ is $\F_\infty$-measurable;
		\item \label{dich2}  
		for any $DF$-invariant, $\F_{\infty}$-measurable line field   $\V\subset TX$ either 
${(\xi,x) \mapsto V_\xi(x)}$  is $\F$-measurable, or 
		\begin{align}\label{eq:olive} 
\text{for $\nu$-a.e.\ $\xi$, $\muskew_\xi$-a.e.\ $x$, and  $\nucondF_\xi$-a.e.\ $\eta$, 
\(V_\xi(x) \neq V_\eta(x).\)}
\end{align}

		\end{enumerate}
\end{lemma}

Recall the  family of conditional measures $\{\muskew^\Fol_{(\xi,x)}\}$ induced by $\Fol$ is defined   
by
$d \muskew^\Fol_{(\xi,x)} (\eta, y) =d \nucondF_\xi(\eta) \times  \delta_x(y)$. 
  Thus, for $\nu$-a.e.\ $\xi$, and $\muskew_\xi$-a.e.\ $x$, $V_\eta(x)$ is defined for  $\nucondF_\xi$\ae $\eta$ and the comparison  \eqref{eq:olive} is well defined \ae  

\begin{proof}
To see \ref{dich1} we recall that $\xi\mapsto \cocycle    [\xi][-n]$ is $\hat\Fol$-measurable for all $n\ge0$.  Then 
$$  \xi\mapsto \cocycle = \left(\cocycle[\theta^n(\xi)][{-n}]\right) \inv$$
is $\theta^{-n} (\hat \F)$-measurable.  It follows that $\xi\mapsto \cocycle $ is $\hat \F_\infty$-measurable for all $n\ge 0$.  
Since $E^s_\xi(x)$ depends only on $\cocycle$ for $n\ge 0$, we have  $$(\xi,x)\mapsto E^s_\xi(x) = \left\{v\in T_xM \mid  \lim _{n\to \infty} \frac 1 n |D\cocycle  (v)| <0\right\}$$ is $\F_\infty$-measurable.  

\def\Qpart{\mathcal Q}
 
To prove \ref{dich2} let $\P$ denote the measurable partition of $X$ into level sets of $(\xi,x) \mapsto V_\xi(x)$.
We assume \eqref{eq:olive} fails: 
\begin{equation}\begin{aligned}\label{eq:treeness}
	\muskew\bigg\{(\xi,x) \mid \muskew^\Fol_{(\xi,x)}& (\P(\xi,x))  > 0\bigg\}= \muskew\bigg\{(\xi,x) \mid \nucondF_\xi\big\{ \eta \mid V_\xi(x) = V_\eta(x) \big\} > 0\bigg\}>0.  
\end{aligned}\end{equation}
From \eqref{eq:treeness} we will deduce $\F$-measurability of $(\xi,x)\mapsto V_\xi(x)$.  

\def\meashere{\muskew _{(\xi,x)}^\Fol}
\def\measheren{\muskew _{(\xi,x)}^{\Fol_n}}
Let 
	\(\Fol_{n}: = F^{-n} (\F)\)
and write $\{\measheren\}$ for a corresponding family of conditional measures. Also write  \(\hat \Fol_{n}: = \theta^{-n} (\hat \F)\).


For each $(\xi,x)\in X$ define 
\[ \phi_n(\xi,x) := 
\measheren(\P(\xi, x)).
\]
  We have  
\[ \phi_n(\xi,x) =  \Ex_{\meashere}(1_{\P (\xi, x)}(\cdot)\mid \F_{n})(\xi,x) = \Ex_{\nucondF_\xi}(1_{\P (\xi, x)}(\cdot, x)\mid \hat \F_{n})(\xi). \] 
Consider any $(\xi,x)$ with $ \muskew^\Fol_{(\xi,x)} (\P(\xi,x))>0$ and such that  $\V$ is $\Fol_\infty$-measurable modulo $ \muskew^\Fol_{(\xi,x)} $. 
For $\eta\in \Omega$ define $$\psi_n(\eta) := 
 \Ex_{\nucondF_\xi}(1_{\P (\xi, x)}(\cdot, x)\mid \hat \F_{n})(\eta)
.$$
Then $\psi_n(\eta) $ is a martingale (with filtration $\hat \F_n$ on the measure space $(\Omega, \B_\Omega, \nucondF_\xi)$) 
whence (using  the $\Fol_\infty$-measurability of $\V$)
$$\psi_n(\eta)\to \Ex_{\nucondF_\xi}(1_{\P (\xi, x)}(\cdot, x)\mid \hat \F_{\infty})(\eta) = 1_{\P (\xi, x)}(\eta,x)$$ $ \nucondF_\xi$-\as as $n\to \infty$.  
In particular, for $\meashere$-\ae $(\eta,x) \in \P(\xi,x)$
$$\phi_n(\eta,x) \to 1$$ as $ n\to \infty.$
It follows from  \eqref{eq:treeness}  that 
\begin{equation}\label{eq:downlow}\muskew\left\{ (\xi,x) \in X \mid \phi_n(\xi,x) \mapsto 1 \text{\ as\ } n\to \infty\right\}>0.\end{equation}

The $\F$-measurability of $(\xi,x)\mapsto V_\xi(x)$ is equivalent to the assertion that 
\[\muskew\{(\xi,x)\mid \phi _0(\xi,x) = 1\} = 1.\]  
Since $F^n_*(\muskew _{(\xi,x)}^{\Fol_n}) = \muskew _{F^n(\xi,x)}^{\Fol}$,  $\V$ is $DF$-invariant, and $(\xi,x) \mapsto D_x\cocycle$ is $\F_n$-measurable, we  have that $\phi_0(F^n(\xi, x)) = \phi_n(\xi,x).$
The ergodicity and $F$-invariance of $\muskew$ and 
\eqref{eq:downlow} then imply that $\phi_0\equiv 1$ on a set of full measure completing the proof.
\end{proof}

\subsection{Sets of good angles, geometry of intersections, and bounds on distortion}\label{sec:all the carp}
We remark that in this section, all estimates are with respect to the background Riemannian  metric on $M$.  
Let $X_1\subset X$ denote the full $\muskew$-measure subset such that $E^{u/s}_\xi(x)$ is defined and $W^{u/s}_\xi(x)$ is an injectively immersed curve tangent to $E^{u/s}_\xi(x)$. 
Furthermore, assume the affine parameters and corresponding parametrizations  $\I^{u/s}$ in \eqref{eq:affparammfolds} are defined on $W^{u/s}_\xi(x)$ for every $(\xi,x)\in X_1$. 
Given $\gamma_ 1>0$, let {$\Lambda({\gamma_1})\subset X_1$} denote the set of points where  
$$\angle 	\left(E^s_\xi (x) , E^u_\xi (x)\right)> \gamma_1.$$
Given $0<\gamma_2 <\gamma_1/2$ and $(\xi,x)\in  \Lambda(\gamma_1)$ 	define 
$A_{\gamma_2}(\xi,x)$ to be the set of $\eta\in \Omega$ with 
\begin{enumerate}
\item $(\eta,x) \in X_1$
\item $ \angle 	\left(E^s_\xi (x) , E^s_\eta (x)\right)>    \gamma_2$, and 
\item $	\angle\left(E^u_\xi (x) , E^s_\eta (x)\right)>	 \gamma_2 .$
\end{enumerate}

As $\muskew(X_1) = 1$, as remarked in the previous section, for almost every $(\xi,x)$ we have $(\eta,x) \in X_1$ for $ \nucondF_{\xi}$-\ae $\eta$.  
For $0<a<1$ we define the set
$\calA_{\gamma_1,\gamma_2, a}\subset \Lambda(\gamma_1)$ by 
$$\calA_{\gamma_1,\gamma_2, a}:= \left\{ (\xi,x) \in \Lambda(\gamma_1) \mid  \nucondF_{\xi} (A_{\gamma_2}(\xi,x)) >a\right\}.$$
	
	From Lemma \ref{lem:VFdichot} we obtain the following.  
\begin{lemma}\label{lem:goodangle}
Assume that $(\xi,x)\mapsto E^s_\xi(x)$ is not $\Fol$-measurable.  Then for any $\alpha>0$ and $0<a<1$ there exists $\gamma_1>0$ and $\gamma_2>0$ with 
$$\muskew(\calA_{\gamma_1,\gamma_2, a}) >1-\alpha.$$
\end{lemma}


Fix a uniform $\rho_0>0$ to be smaller than the injectivity radius of  $M$  and given $x\in M$  let 
	$$\exp_x\colon B(\rho_0)\subset T_xM\to M$$
denote the exponential map.  
We recall that for every $(\xi,x)\in  X_1$ we have selected 
$v^u(\xi,x) \in E^u_\xi(x)$ and $v^s(\xi,x) \in E^s_\xi(x)$ such that $(\xi,x)\mapsto v^u(\xi,x) $ is $\Fol$-measurable and $(\xi,x)\mapsto v^s(\xi,x) $ is $\muskew$-measurable.  

By Lusin's theorem, there is a compact subset $\Lambda_2\subset  X_1$, of measure arbitrarily close to 1, on which the family of parametrized stable and unstable manifolds 
$${(\xi,x)} \mapsto \I^\sigma_{(\xi,x)} $$
 vary continuously in the $C^1$ topology on the space of embeddings  $C^1([-r,r], M)$ for  $\sigma= \{s,u\}$ and all $0<r<1$.  

Given $x\in M$, a subspace $V\subset T_xM$, and $0<\gamma<\pi$ we denote by $\Cone_\gamma(V)$ the open cone of angle $\gamma$ around the subspace $V$.  
We have the following. 
\begin{lemma}
Given any $\gamma>0$, there exist $\hat r_1, \hat r_0>0$ such that for all $(\xi, x)\in \Lambda_2$ and all $(\xi,y)\in \Lambda_2$  with $d(x,y)<\hat r_0$ 
\begin{enumerate}
\item $\exp_x\inv\left (\locunstM [\hat r_1] x { \xi} \right) \subset \Cone_{\gamma} \left(E^u_{\xi}(x)\right)$;
\item $\exp_x\inv\left(\locstabM [\hat r_1] x { \xi} \right)\subset \Cone_{\gamma} \left(E^s_{ \xi}(x)\right)$;
\item {$\exp_x\inv\left(\locstabM [\hat r_1] y {\xi} \right)\subset \Cone_{\gamma} \left(E^s_{ \xi}(x)\right) + \exp_x\inv(y).$}
\end{enumerate}
\end{lemma}

Fix $\epone= \eps/10$.  Fix a family of Lyapunov charts $\phi(\xi,x)$ with corresponding function $\ell\colon X\to [1,\infty)$ and retain all related notation from  Section  \ref{sec:charts2sided}.  Let $\Lambda_3\subset \Lambda_2$ be a set on which $\ell$ is bounded above by $\ell_0$ and such that  there exist   $0<\td r_0$ and $0<\td r_1$  such  that for $(\xi,x)\in \Lambda_3$, \begin{enumerate}
\item $\locstabM [\td r_1] x \xi \subset V^s(\xi,x)$ where $V^s(\xi,x)$ is the local stable manifold built in Theorem \ref{thm:locPesinStab}; 
\item the diameters of  $\locstabM [\td r_1] x \xi $ and $\locunstM [\td r_1] x \xi $ are less that $\dfrac {\ell_0^{-3} e^{-\lambda_0-\epone}}{10 k_0}$;
\item  if $(\xi,x), (\xi,y)\in \Lambda_0$ with $d(x,y)\le \td r_0$ then 
$$\phi(\xi,x)\left(\locstabM [\td r_1] y \xi\right) $$ is the graph of a $1$-Lipschitz function $g_{x,y}\colon D\subset \R^s\to \R^u(\ell_0\inv e^{-\lambda_0 - \epone})$ for some $D\subset \R^s(\ell_0\inv e^{-\lambda_0 - \epone})$.  
\end{enumerate}
Note, in  particular, for $(\xi,x)$ and $(\xi,y)$ above that $\locstabM [\td r_1] y \xi$ is in the domain of the chart $\phi(\xi,x)$.  
We may take $\mu(\Lambda_3)$ arbitrarily close to $\mu(\Lambda_2)$.

\input{controlledintersections}

Appealing repeatedly to Lusin's theorem, and standard estimates in the construction of stable and unstable manifolds, we may choose parameters satisfying the following.  See Figure \ref{fig:1}.  (Note that in our application of Figure \ref{fig:1}  we have  $W^u_\xi(x) =W^u_\eta(x)$.)

\begin{lemma}\label{prop:controlledintersections}
For every $0<\gamma_1$,  $0<\gamma_2<\gamma_1/2$, and $\Lambda_3\subset \Lambda_2\subset \Lambda(\gamma_1)$ as above there exist a subset $\Lambda '\subset \Lambda_3$  with $\muskew(\Lambda')$ arbitrarily close to $\muskew(\Lambda_3)$,  positive constants $r_0<\td r_0,  r_1<\td r_1$, and constants $C_1, C_2, C_3,  D_1>1$, with the following properties.

 For $(\xi,x) \in \Lambda' $ we have 
\begin{enumlemma}
\item $\frac 1  {C_2} d(x,w) \le \|H^u_{(\xi,x)}(w)\|\le {C_2} d(x,w)$ for all $w\in \locunstM [r_1] x \xi$. 
\label{item8:3} 
\item $\frac 1  {C_2} d(x,w') \le \|H^s_{(\xi,x)}(w')\|\le {C_2} d(x,w')$ for all $w'\in \locstabM [r_1] x \xi$. 
\label{item8:4} 
\end{enumlemma}
{ For $(\xi,x), (\xi,y) \in \Lambda' $ 
with $d(x,y)<r_0$}
\begin{enumlemma}[resume]
\item \label{itme1} $\locstabM [ r_1] x \xi \cap \locunstM [ r_1]y \xi$ is a singleton $\{z\}$ and the intersection is uniformly transverse;  
\end{enumlemma}
{ furthermore, if $\eta\in A_{\gamma _2}(\xi, x)$ and $(\eta,y), (\eta,x) \in \Lambda'$}
\begin{enumlemma}[resume]
\item \label{itme2}  $\locunstM [ r_1] x \xi \cap \locstabM [ r_1]y \eta$ is a singleton $\{v\}$ and the intersection is uniformly transverse, and   
\item \label{itme3}  if $D_1\cdot \|H^u_{(\xi,y)}(z)\|\le \|H^s_{(\xi,x)}(z) \| $ then 
$$ \dfrac{1}{C_3} \|H^s_{(\xi,x)}(z) \| \le \|H^u_{(\xi,x)}(v) \|\le C_3 \|H^s_{(\xi,x)}(z) \|.$$
\end{enumlemma}
\def\rhere{{r_1}}
Additionally,    we have a uniform bound $C_1$ so that for $(\xi,x)\in \Lambda'$ and  $ w\in \locunstM[ \rhere]  x \xi$ 
\begin{enumlemma}[resume]
\item \label{item8:1} $\displaystyle\dfrac 1 {C_1}\le  \|\restrict {D_x\cocycle[\xi][-n]  }{T_x\locunstM[ \rhere] x \xi }\| \cdot \|\restrict{D_w\cocycle[\xi][-n] }{T_w\locunstM[ \rhere] x \xi } \|\inv \le C_1 $ for all $n\ge 0$
\end{enumlemma}
and for $(\xi,y)\in \Lambda'$ with $d(x,y)<r_0$ and $z$ as in \ref{itme1}
\begin{enumlemma}[resume]
\item \label{item8:2} $\dfrac1 {C_1}\le  \displaystyle \|\restrict {D _x \cocycle }{T_x\locunstM[ \rhere] x \xi} \| \cdot \|\restrict{D_z \cocycle }{T_z\locunstM[ \rhere] y \xi} \|\inv \le C_1$ for all $n\ge 0$.  
\end{enumlemma}

%
\label{lem:contsets} \label{lem:standardcrap}
\end{lemma}
\begin{proof}
Conclusions  \ref{item8:3}--\ref{itme2} follow simply from the $C^1$ topology and Luzin's theorem.  

For \ref{itme3}, we work in the exponential chart at $x$.   By the law of sines, given fixed $\gamma_1$ and  $\gamma_2$, we pick a sufficiently small $\gamma>0$  so that if $\hat y\in C_\gamma (E^s_\xi(x))$ then there exists $\hat C>1$ with 
$$\dfrac {d(0,\hat y)}{\hat C}\le \max \{ d(0, v): v\in C_\gamma(E^u_\xi(x)) \cap C_\gamma(E^s_\eta(x)) + \hat y \} < \hat C d(0,\hat y) $$
We then obtain  \ref{itme3}  from the uniform Lipschitz bounds on the exponential map and the affine parameters.  See Figure \ref{fig:1}.

The estimates \ref{item8:1} and \ref {item8:2} follow from the fact that the pairs $\cocycle[\xi][-n] (x)$ and $\cocycle [\xi][-n] (w)$,    $\cocycle (x)$ and $\cocycle(z)$, and  $D _x \cocycle (T_x\locunstM[ \rhere] x \xi)$ and $D_z \cocycle (T_z\locunstM[ \rhere] y \xi)$ are exponentially asymptotic while $|f_\xi|_{C^1}$, $\lip(Df^n_{\xi})$, and the Lipschitz constants for the variation of the tangent spaces to $\cocycle [\xi] [-n] (\locunstM x \xi )$ grow sub-exponentially for  $\xi\in \Omega_0$ and $(\xi,x)$ satisfying Proposition \ref{prop:Stabman}.  
\end{proof}

The following lemma will be needed in Claim \ref{claim:GP} below.  
\begin{lemma}\label{lem:inclination}
Take $(\xi,x) \in \Lambda'$ and $ (\xi,y)\in \Lambda'$ with $d(x,y)<r_0$, and set
$z= \locstabM [ r_1] x \xi \cap \locunstM [ r_1]y \xi$  and 
$w=  \locstabM [ r_1] y \xi \cap \locunstM [ r_1]x\xi$.
Let $\Gamma\subset  \locstabM [ r_1] y \xi$ be the curve with endpoints $w$ and $y$.  Let $n\ge 0$ be such that   $$\left|\phi(F^{-j}(\xi,x))(\cocycle [\xi][-j] (z))\right|\le \frac{e^{-\lambda_0-\epone} \ell_0\inv  e^{-\epone j}}{10}$$
for all $0\le j\le n$.  
Then $\cocycle [\xi][-n] (\Gamma)$ is in the domain of $\phi(F^{-n}(\xi,x))$ and  $$\phi(F^{-n}(\xi,x))(\cocycle [\xi][-n] (\Gamma))$$ is the graph of a $1$-Lipschitz function $$g\colon \hat D\subset \R^s\to \R^u(e^{-\lambda_0-\epone} \ell  (F^{-n}(\xi,x) )\inv )$$ for some $\hat  D\subset\R^s(e^{-\lambda_0-\epone} \ell  (F^{-n}(\xi,x) )\inv )$.  
\end{lemma}
\begin{proof}
We prove by induction on $n$.  For $n=0$ the conclusion follows from  hypotheses and the choice of $\Lambda_3$.
For $n\ge 1$, assume $$\phi(F^{-(n-1)}(\xi,x))(\cocycle [\xi][-(n-1)] (\Gamma))$$ is the graph of a $1$-Lipschitz function $g\colon D\subset \R^s\to \R^u(\ell  (F^{-(n-1)}(\xi,x) )\inv e^{-\lambda_0-\epone})$ for some $D\subset\R^s(e^{-\lambda_0-\epone} \ell  (F^{-(n-1)}(\xi,x) )\inv )$.  
From Lemma \ref{lemm:incling} 
it follows that $\cocycle [\xi][-n] (\Gamma)$ is in the domain of $\phi(F^{-n}(\xi,x))$ and 
$$\phi(F^{-n}(\xi,x))(\cocycle [\xi][-n] (\Gamma))$$ is the graph of a $1$-Lipschitz function $\hat g\colon \hat D\subset \R^s\to \R^u(e^{-\lambda_0-\epone} \ell  (F^{-n}(\xi,x) )\inv )$ for some $\hat D\subset\R^s(\ell (F^{-n}(\xi,x) )\inv )$.
By the hypothesis, we have that 
$ \cocycle [\xi][-n] (z)$ 
is contained in the domain of $\phi(F^{-n}(\xi,x))$.   We have 
\begin{align*}
d(\cocycle [\xi][-n] (y)), (\cocycle [\xi][-n] (z)))
&\le \ell_0 k_0 e^{n(-\lambda^u + 2 \epone ) } \frac {\ell_0^{-3} e^{-\lambda_0-\epone}}{10 k_0} \\
&\le e^{n(-\lambda^u + 2 \epone ) } \frac 1{10}\ell_0^{-2} e^{-\lambda_0-\epone}\\
\end{align*}
Thus, 
\begin{align*}
\left|  \phi(F^{-n}(\xi,x)) (\cocycle [\xi][-n] (y))- \phi(F^{-n}(\xi,x))(\cocycle [\xi][-n] (z)) \right|
&\le \ell_0 e^{n\epone}e^{n(-\lambda^u + 2 \epone ) } \frac 1{10}\ell_0^{-2} e^{-\lambda_0-\epone}\\&\le\frac 1 {10} \ell(F^{-n}(\xi,x))\inv e^{-\lambda_0-\epone}\\
\end{align*}
hence,
$$\left|  \phi(F^{-n}(\xi,x)) (\cocycle [\xi][-n] (x))- \phi(F^{-n}(\xi,x))(\cocycle [\xi][-n] (y)) \right|\le \frac 2 {10} \ell(F^{-n}(\xi,x)) \inv e^{-\lambda_0-\epone}.$$
Similarly, 
$$\left|  \phi(F^{-n}(\xi,x)) (\cocycle [\xi][-n] (x))- \phi(F^{-n}(\xi,x))(\cocycle [\xi][-n] (w)) \right| \le \frac 1 {10} \ell(F^{-n}(\xi,x)) \inv e^{-\lambda_0-\epone}$$
hence 
$\hat  D\subset\R^s(e^{-\lambda_0-\epone} \ell  (F^{-n}(\xi,x) )\inv )$.  
\end{proof}



Under the hypotheses of Lemma \ref{lem:inclination}, if   $F^{-n}(\xi,x)\in \Lambda'$ and  $F^{-n}(\xi,y)\in \Lambda'$
then $\phi(F^{-n}(\xi,x))(\Gamma)$ and $$\phi(F^{-n}(\xi,x)) \left(\locunstM [ r_1]{  \cocycle [\xi][-n](x) }{\theta^{-n} (\xi)} \right)$$ have at most one point of intersection.  Thus we have the following corollary.
\begin{corollary}\label{cor:tolemma}For $n$ satisfying the  hypotheses of Lemma \ref{lem:inclination}, if  $F^{-n}(\xi,x)\in \Lambda'$, $F^{-n}(\xi,y)\in \Lambda'$, and $d(\cocycle [\xi][-n](x) , \cocycle [\xi][-n](y) )\le r_0$ then 
$$  \locstabM [ r_1] { \cocycle [\xi][-n](y) } {\theta^{-n} (\xi)} \cap \locunstM [ r_1]{  \cocycle [\xi][-n](x) }{\theta^{-n} (\xi)} = { \cocycle [\xi][-n](w) }.$$
\end{corollary}

\begin{remark}\label{rem:scase}
In the case that $\Omega = \Sigma$, recall that the stable line fields, stable manifolds, and corresponding affine parameters are defined using only the forwards itinerary.  As $\cocycle[\eta][n]= \cocycle [\xi] $   for all $\eta\in \Sigmalocs(\xi)$ and $n\ge 0$, it follows that we may choose $v^s({\xi,x})$ and the corresponding  parametrizations $\I^s_{(\xi,x)} $ of the stable manifolds to be constant on sets of the form $\Sigmalocs(\xi)\times \{x\}$ and thus induce corresponding objects  $v^s({\omega,x})$ and $\I^s_{(\omega,x)} $  for the one-sided skew product on $\Sigma_+\times M$.  
  We may then take 
 $\hat \Lambda^s\subset \Sigma_+\times M$ of $(\nunaught^\N\times \munaught)$-measure arbitrarily close to $1$ so that on $\hat \Lambda^s$ the parametrized stable manifolds 
$${(\omega,x)} \mapsto \left (t\mapsto \I^s_{(\omega,x)}\left(tv^s({\omega,x})\right) \right)$$
vary continuously in the space of $C^1$-embeddings  $[-r,r]\to M$ for all sufficiently small  $0<r<1$.  

Let $ \pi_+\colon \Sigma\to \Sigma _+$ be the natural projection.  We may modify parts \ref{itme1}--\ref{itme3}, \ref{item8:2}  of Lemma  \ref{lem:standardcrap} as follows:
Choose $\Lambda' $ in Lemma  \ref{lem:standardcrap} so that if $(\xi,x)\in \Lambda'$ then $(\pi_+(\xi), x)\in \hat \Lambda^s$.  Then the constants can be chosen so that: 
\begin{customprop}{\ref{prop:controlledintersections}'}\label{ninetynine}
{ For $(\xi,x)\in   \Lambda' $, $(\zeta, y)\in \Lambda'$ with $\pi_+(\zeta) = \pi_+ (\xi)$ 
and $d(x,y)<r_0$}
\begin{enumerate}[label*=\emph{(\alph*')}, ref= {(\alph*')} ]
  \setcounter{enumi}{2}
\item \label{itme1'} $\locstabM [ r_1] x \xi \cap \locunstM [ r_1]y \zeta$ is a singleton $\{z\}$ and the intersection is uniformly transverse;  
\end{enumerate}
{furthermore, if $\eta\in A_{\gamma _2}(\xi, x)$, $\eta'\in \Sigma$ is such that $\pi_+(\eta') = \pi_+ (\eta)$ and $(\eta',y), (\eta,x) \in \Lambda'$ then}  
\begin{enumerate}[label*=\emph{(\alph*')}, ref= {(\alph*')}, resume]
\item \label{itme2'}  $\locunstM [ r_1] x \xi \cap \locstabM [ r_1]y {\eta'}$ is a singleton $\{v\}$ and the intersection is uniformly transverse, and   
\item \label{itme3'}  if $D_1\cdot \|H^u_{(\zeta,y)}(z)\|\le \|H^s_{(\xi,x)}(z) \|$ then  
$$ \dfrac{1}{C_3} \|H^s_{(\xi,x)}(z) \| \le \|H^u_{(\xi,x)}(v) \|\le C_3 \|H^s_{(\xi,x)}(z) \|.$$
\end{enumerate}
Additionally, for $x,y,z$ as above   we have uniform bounds 

\begin{enumerate}[label*=\emph{(\alph*')}, ref= {(\alph*')} ]
  \setcounter{enumi}{6}
\item \label{item8:2'} $\dfrac1 {C_1}\le  \displaystyle \|\restrict {D _x \cocycle }{T_x\locunstM[ \rhere] x \xi} \| \cdot \|\restrict{D_z \cocycle [\zeta][n]}{T_z\locunstM[ \rhere] y \zeta} \|\inv \le C_1$ for all $n\ge 0$.  
\end{enumerate}
\end{customprop}
Note that in \ref{item8:2'},  
we have $\cocycle = \cocycle [\zeta][n]$ for all $n\ge 0$.  


\end{remark}

\section{Proof of Proposition \ref{prop:main}}\label{sec:lemMainProof}\label{sec:proofmainlem}\label{sec:lemmaproof}
\def\stepp{Step \setcounter{thestep}{1}\arabic{thestep}}
\def\MCT{\mathfrak S}
\def\Rec{\mathcal R}

Given the skew-product $F\colon X \to X$ and $\muskew$ satisfying \eqref{eq:IC2}, we have    $\lambda^s<0<\lambda^u$ given by the hyperbolicity of $DF$ and $\eps< \min\{1,\lambda^u/200, -\lambda^s/200\}$ fixed in Section \ref{sec:PT}. 
We recall the family $\bar \muskew_{(\xi,x)}$ introduced in Section \ref{sec:barmu}. 
Recall that our goal is to prove for such measures that the measurable sets $G_\epsilon$ and $G$ in  Proposition \ref{prop:main} have positive $\mu$-measure (for some fixed $M$ and all sufficiently small $\epsilon$).  This will be shown in Section \ref{sec:finallytheproof}.  

We define $X_0\subset X_1\subset X$ to be the full $\muskew$-measure, $F$-invariant subset of $\Omega_0\times M$ where all propositions from Section \ref{sec:BG} hold and such that 
the stable and unstable manifolds, Lyapunov norms, affine parameters, and the parametrizations  $\I^\us$ are defined.  We also assume that for $(\xi,x) \in X_0$ the  measures $\muskew_\xi$, $\muskew^u_{(\xi,x)}$, $\muskew^s_{(\xi,x)}$, and  $\bar\muskew_{(\xi,x)}$ are defined
 and satisfy
$F_*\muskew_\xi = \muskew_{\theta(\xi)}$, 
$F_*\muskew^{u/s}_{(\xi,x)}\simeq \muskew^{u/s}_{F(\xi,x)}$, and $\left(\I^u_{F(\xi,x)}\right)\inv\circ F \circ \I^u_{(\xi,x)}(  \bar\muskew_{(\xi,x)}) \simeq \bar \muskew_{F(\xi,x)}$.  We further assume    $\bar \muskew_{(\xi,x)}$ contains $0$ in its support.  
 We further take $X_0$ so that for  $(\xi,x)\in X_0$ and $\nucondF _\xi$-\ae $\eta\in \Omega$ we have $\cocycle = \cocycle [\eta][n]$ for all $n\le 0$;  for such $\eta$ we have 
 $\unstM x \xi = \unstM x \eta$, and corresponding equality of affine parameters, parametrizations $\I^u$, and measures $ \bar\muskew_{(\xi,x)} =  \bar\muskew_{(\eta,x)}$.  
 Finally, we  assume that if $(\xi,x)\in X_0$ then  $\mu_\xi(X_0)=1$ and  $\muskew^\Fol_{(\xi,x)}(X_0) = \nucondF_\xi(\{\eta\colon (\eta,x)\in X_0\} )= 1$.  

Write $Y_0 = [0,1)\times X_0$.

   \subsection{
Choice of parameters and sets}\label{sec:choices}
We pick any

 \begin{choices}
 \item $0<\beta<1$ such that $\dfrac {1+\beta}{1-\beta} < \dfrac{\lambda^u- \lambda^s -2 \eps}{ -\lambda^s + \eps}$;
 \end{choices}
and fix the following constants for the remainder
 \begin{choices}
\item   ${\kappa_1= \dfrac{  \lambda^u-\lambda^s-2\eps}{ \lambda^u+\eps }}$;   ${\kappa_2= \dfrac{  \lambda^u-\lambda^s+2\eps}{ \lambda^u-\eps }}$;
 \item $\alpha_0 = 
 \dfrac{1}2 - \dfrac{ 1}{2} \dfrac{ (1+ \beta)(-\lambda^s + \eps)}{(1-\beta)(\lambda^u - \lambda^s - 2 \eps)}$; 
 \item 
  $\alpha =\left( \dfrac {\kappa_1\alpha_0}{5(\kappa_1 +\kappa_2)}\right)$. 
 \end{choices}
Note $\alpha_0>0$ by the choice of $\beta$.  
%
%

Recall that the measures  $\mususp$ and $\mualt$ are equivalent.  
We select 
\begin{choices}
\item \label{N0} $N_0$ such that $\mususp\left\{ p: \frac 1 {N_0} \le \dfrac {d\mualt}{d\mususp}(p)\le N_0\right\} >1-\alpha/2$.
\item \label{choice:cont}  By Lusin's Theorem, we may choose a compact subset $K_0\subset Y_0\subset Y$ with $\mususp$- and $\mualt$-measure sufficiently close to $1$ on which
\begin{enumerate}[label= \roman*)]
\item the frames for the stable and unstable subbundles $p\mapsto v^s_p$, $p\mapsto v^s_p$ defined in Section \ref{sec:frames};
\item the stable and unstable manifolds parametrized by \eqref{eq:affparammfolds};
\item all Lyapunov norms defined in Section \ref{sec:lyaponY}; 
\item the families of measures $\bar \mususp_p$
\end{enumerate} 
vary continuously.  

We may also assume the functions 
$a(p)$ in \eqref{eq:controlled}  and $\hat L(p)$  in Proposition \ref{prop:subexpcomp} are bounded on $K_0$, respectively, by  
by  $a_0$ and $\hat L$.  
Finally, we may assume there is a $L_1>1$ so that 
for $p= (\pt,\xi,x)\in K_0$ and $y\in \locstabM [1]x {(\pt,\xi)}$, writing $(\pt_t,\xi_t,x_t) = \Phi^t(p), $ and $
(\pt_t,\xi_t,y_t) = \Phi^t(\pt,\xi,y)$, for any $t\ge 0$  
$$d(x_t, y_t)\le L_1 e^{t(\lambda^s + \eps)}d(x,y).$$
 \item \label{M0} Let $M_0>1$ denote the maximal ratio of all Lyapunov and Riemannian norms on $  K_0$:

$$ M_0 = \sup_{p= 
\in K_0}\sup_{0\neq v\in T_xM}\left\{
\left( \dfrac{\lyap v_{\eps, \pm, p}}{\lyap v_{\eps, -, p}}\right)^{\pm1}, 
 \left(\dfrac{\lyap v_{\eps, \pm, p}}{\| v\|}\right)^{\pm1},  \left( \dfrac{\lyap v_{\eps, -, p}}{\| v\|}\right)^{\pm1}   \right\}.$$

\item  \label{1442} As discussed in Section \ref{sec:all the carp}, fix $\epone<\eps/10$, a function $\ell\colon X\to [1,\infty)$ and a family of Lyapunov charts  $\phi(\xi,x)$.  Let $\ell_0>1$ be such that $\ell(\xi,x)\le \ell_0$ for $(\xi,x)\in\Lambda_3$ where $\Lambda_3$ is as in Section  \ref{sec:all the carp}. 
\item  \label{13} By Lemma \ref{lem:goodangle}, we pick $\gamma_1, \gamma_2$ so that 
$\muskew(\calA_{\gamma_1,\gamma_2, .9})>1-\alpha.$ 
We fix $ C_1, C_2, C_3>1$,   $0< r_0, r_1<1$,  $D_1>1$  and $\Lambda'\subset \Lambda_3$ with measure sufficiently close to $1$ satisfying Lemma \ref{prop:controlledintersections}.
We write  $\scrA = [0,1)\times \calA_{\gamma_1,\gamma_2, .9}$ and for $p=(\pt,\xi,x)\in Y_0$,  $A_{\gamma_2}(p) = 
A_{\gamma_2}(\xi,x)$.  
\item \label{14}  Take $\hat r = \min\{r_0/(2C_2), r_1\}.$  


 \item \label{11}  Set $\hat T = \log (\hat L M_0^2C_1^3)/(\lambda^u - \eps)$.

Fix a 
compact set $\Lambda''\subset Y_0$ and $D_0>0$ such that for $p\in \Lambda''$ and $t \in [-\hat T, \hat T]$ $$ D_0\inv \le \|\restrict{D\Phi^t}{E^u(p)}\| \le D_0.$$  


\item\label{12}
We fix a compact $\Omega'\subset \Omega$ with $\nu(\Omega')$ sufficiently close to 1 and such that for all $j\in \Z$ with $|j|\le \hat T +1$ we have $\xi \mapsto \theta^j(\xi)$ and $\xi \mapsto \cocycle [\xi][j]$ are continuous when restricted to $\Omega'$.  Consider  $q_n\in[0,1)\times \Omega'\times M$ converging  to $q\in[0,1)\times \Omega'\times M$, and let $t_n$ be a sequence with $|t_n|\le \hat T$, $t_n \to t$ and $\Phi^{t_n}(q_n) \in [0,1-a]\times \Omega\times M$ for all $n$ and some $a>0$.  It then follows that $\Phi^{t_n}(q_n)$ converges to $\Phi^t(q)$.

\item The above choices can be made so that 
setting
$$K := K_0 \cap \Lambda'' \cap \big([0,1)\times \Lambda' \big)\cap \big([0,1)\times \Omega'\times M\big)$$
we may ensure \begin{equation}\label{eq:measurebounds} 
	\mususp(K)>1 - \frac \alpha {10} \text{ and } \mualt(K)>1 - \frac { \alpha}{40N_0}   
\end{equation}
\item \label{16} Fix 
$\hat \alpha = \dfrac { \alpha}{40 N_0}  $ .   Let $U\subset Y$ be any  open set to be specified later with $\mualt(U)< \hat \alpha.$ We have  $\mualt(K\sm U)> 1 - \frac { \alpha}{20N_0}.$
\end{choices}

Recall if $\psi\colon Y\to [0,1]$ is $\mususp$-measurable with $\int \psi \ d \mususp > 1- ab$ then 
\begin{equation}\label{eq:splitmeasure} \mususp\{ p : \psi(p) > 1-a\} > 1- b.\end{equation}
%

Recall that we have the filtration $\{\Sal^m: m\in \R\}$ on $(Y, \mualt)$ decreasing to  $\Sal^\infty= \bigcap_{m\in \R} \Sal^m.$ 
From  \eqref{eq:splitmeasure}  and \ref{N0} we claim
\begin{claim}\label{claim:silly}
With $N_0$,  $K$, and $U$ as above
\begin{enumlemma}
\item $\mususp\left\{ p : \Exp_\mususp (1_K\mid \Sal)(p)>.9 \right\} >1-\alpha$\label{claim:1};
\item $\mualt\left\{p : \Exp_\mualt (1_{K\sm U}\mid \Sal^\infty)(p)>.9 \right\} >1-\alpha/(2N_0)$\label{claim:2};
\item $\mususp\left\{p : \Exp_\mualt (1_{K\sm U}\mid \Sal^\infty)(p)>.9 \right\} >1-\alpha$\label{claim:3}.
\end{enumlemma}
\end{claim}
\begin{choices}\item Let  $S_0:=  \{p \in Y \mid \mususp^\Sal_p(K)>.9 
\}.$  
\end{choices}

As discussed in Section \ref{subsec:MCT}, taking  $g(p) = 1_{K\sm U}(p)$ 
we have $$\mususp_{ \Psi^{-m}(p)}^\Sal(  \Psi^{-m}(K\sm U)) =   \Exp_\mualt(1_{K\sm U}\mid \Sal^m)(p)
 \to \Exp_\mualt(1_{K\sm U}\mid \Sal^\infty)(p)$$ as $m\to \infty$ for  $\mususp$-\ae $p\in Y$.  
Given $M>0$ let 
$$S_M = 	\{p \in Y \mid \mususp_{ \Psi^{-m}(p)}^\Sal(  \Psi^{-m}(K\sm U)) >.9 
	 \text { for all } m\ge M\} .$$

\begin{choices}
\item \label{17} Fix $\hat M$ so that $\mususp(S_{\hat M})>1-\alpha.$

\end{choices}
\def\rec{\mathcal{R}}

Given $T>0$, define $\rec(T)\subset K$ to be the set of $p\in K$ such that for $B= K, \scrA, S_{\hat M}, $ or $S_0$  and any $T',T''\ge T$
\begin{enumerate}[label=\roman*)]
\item $\dfrac 1 {T'} \leb (\{ t\in [0,T']\colon \Phi^t(p) \in B\})>1-\alpha$; 
\item $\dfrac 1 {T''} \leb( \{ t\in [-T'',0]\colon \Phi^t(p) \in B\})>1-\alpha$;
\end{enumerate}
and  thus $\dfrac 1 {T'+T''} \leb( \{ t\in [-T'',T']\colon \Phi^t(p) \in B\})>1-\alpha$.  
\begin{choices}
\item \label{19} By the pointwise ergodic theorem, fix $T_0$ with  $\mususp(\rec(T_0))>0$.  
\item \label{20} Finally, set $\epsilon_0= \min\left \{\dfrac{\hat r}{M_0^4},\dfrac{ r_1}{2C_3 M_0^6},   \dfrac{e^{\lambda^s - \eps} e^{-\lambda_0-\epone} \ell_0^{-2}}{10 C_2 }\right\}$.  
\end{choices}

\subsection{
Choice of time intervals}\label{sec:times}
Consider a fixed $\epsilon<\epsilon_0.$  This $\epsilon$ will be as  in Proposition \ref{lem:main}. 
 Given $0<\delta<1$ we define 

 \begin{enumerate}
 \item $m_\delta = \dfrac {(1+\beta)\log \delta- \log (M_0^4)}{\lambda^u- \lambda^s -2 \eps}$
  \item $M_\delta = \dfrac {(1-\beta) \log \delta- \log \epsilon}{ -\lambda^s + \eps} $
 \end{enumerate}

 Note that for all sufficiently small $0<\delta<1$ we have $M_\delta<m_\delta<0.$
For $0<\delta<\epsilon$   consider any $\ell$ with  
$$\dfrac{\log \delta- \log \epsilon}{-\lambda^s+ \eps}\le \ell\le 0.$$  By the definition of $\tau_{p, \delta, \epsilon}$ (see Section \ref{sec:stopping}) we have 
 $\tau_{p, \delta, \epsilon}(\ell)\ge 0$ and
$$
\rexp{\tau_{p, \delta, \epsilon}(\ell)(\lambda^u-\eps)}\rexp{\ell(\lambda^s+\eps)}\delta\le \epsilon \le
\rexp{\tau_{p, \delta, \epsilon}(\ell)(\lambda^u+\eps)}\rexp{\ell(\lambda^s-\eps)}\delta. 
$$
It follows  for such $\ell$ that
 \begin{equation}\label{eq:controlassymp}\dfrac{\log(\epsilon/\delta)+ (-\lambda^s+\eps)\ell}{\lambda^u+\eps}\le \tau_{p, \delta, \epsilon}(\ell)\le \dfrac{\log(\epsilon/\delta)+ (-\lambda^s-\eps)\ell}{\lambda^u-\eps}.\end{equation}
In particular, for any $M_\delta\le \ell \le m_\delta<0$,   \eqref{eq:controlassymp} holds and $\tau_{p, \delta, \epsilon}(\ell)>0$.  

From \eqref{eq:controlassymp} we obtain the following asymptotic behavior.  
\begin{claim} \label{cor:infinitecontrol} For fixed $\epsilon>0$ we have that 
\begin{enumlemma}
\item \label{it:here1}
$\tau_{p, \delta, \epsilon} (0) = L_{p, \delta, \epsilon}(0)\to \infty$,
\item \label{it:here2}
 $\tau_{p, \delta, \epsilon} (M_\delta)\to \infty $, and   
\item \label{it:here3}  $L_{p, \delta, \epsilon}(M_\delta)\to -\infty$
\end{enumlemma}
as $\delta\to 0$; furthermore, the divergence is uniform in $p\in Y_0$.  
\end{claim}

\begin{proof}
Conclusions \ref{it:here1} and \ref{it:here2} follow from  \eqref{eq:controlassymp}.

For  \ref{it:here3} we have 
\begin{align*}
L_{p, \delta, \epsilon}(M_\delta)&\le  \dfrac{\log(\epsilon/\delta)+ (-\lambda^s-\eps)M_\delta}{\lambda^u-\eps}+M_\delta\\
&=\dfrac {\log \epsilon + (\lambda^u - \lambda^s - 2 \eps) M_\delta - \log \delta}{\lambda^u - \eps}\\
& = \dfrac {\log \epsilon}{\lambda^u - \eps} - \dfrac {(\lambda^u - \lambda^s - 2 \eps)\log \epsilon}{(\lambda^u - \eps)(-\lambda^s+\eps)} 
	+\left[
		\dfrac {(1-\beta)(\lambda^u - \lambda^s - 2\eps)}{-\lambda^s +\eps} - 1
	\right]\log \delta
\end{align*}
and the limit follows  from our choice of $\beta$ as
from the fact that \begin{gather*}\dfrac {\lambda^u - \lambda^s - 2\eps}{-\lambda^s +\eps}>\frac {1+\beta}{1-\beta} > \frac 1{1-\beta}.\qedhere\end{gather*}\end{proof}

The choice of $m_\delta$ above guarantees that, for  $(\pt,\xi,x)$ and $(\pt,\xi, y) $ in $K$ with $x$ and $y $ sufficiently close in $M$, the image of $y$  under the backwards flow $\Phi^t$ is   in general position 
with respect to the hyperbolic splitting $T_{\Phi^t(\pt,\xi,x)}M = E^u({\Phi^t(\pt,\xi,x)}) \oplus E^s( {\Phi^t(\pt,\xi,x)})$ for $t<m_\delta$.  $M_\delta$ is chosen so that, in addition to the properties in Claim \ref{cor:infinitecontrol}, the images of $x$ and $y$ do not drift too far apart under the backwards flow.   We make this precise in the following claim.  
\input{fig3}

Recall $\hat r$ is chosen in \ref{14} and that $r_1$ and $r_0$ were fixed in \ref{13} to be as in Lemma \ref{prop:controlledintersections}.  See Figure \ref{fig:22}.
\begin{claim}\label{claim:GP} \label{claim:bounds}
Let $p=(\pt,\xi,x)$ and $ q=(\pt,\xi, y) $ be in $K$ with 
$d(x,y) <  r_0$.  Let $$z = \locunstM [ r_1] y \xi \cap \locstabM [r_1] x \xi,\quad \quad  w= \locstabM [ r_1] y \xi \cap \locunstM [r_1] x \xi.$$
Assume $z\neq x$ and set $\delta = \|H^s_p (z)\|$.  For any $m$ with  $$M_\delta\le m\le m_\delta<0$$ set
\begin{enumerate}[label=\roman*)]
\item $\hat p_m=(\hat \pt_m, \hat \xi_m,\hat  x_m )=  \Phi^m(p)$, 
 $\hat q_m =(\hat \pt_m, \hat \xi_m,\hat  y_m )
 =  \Phi^m(q)$,
\item 
$(\hat \pt_m,\hat  \xi_m,\hat  z_m )=  \Phi^m(\pt,\xi,z)$,
$(\hat \pt_m,\hat  \xi_m,\hat  w_m )=  \Phi^m(\pt,\xi,w)$,

\item $\td p_m= (\td \pt_m, \td \xi_m, \td x_m )= \Phi^{L_{p, \delta, \epsilon}(m)}(p)$,
$\td q_m =  (\td \pt_m, \td \xi_m, \td y_m )=\Phi^{L_{p, \delta, \epsilon}(m)}(q)$,
\item $(\td \pt_m, \td \xi_m, \td z_m )=  \Phi^{L_{p, \delta, \epsilon}(m)}(\pt,\xi,z)$,
$(\td \pt_m, \td \xi_m, \td w_m )=  \Phi^{L_{p, \delta, \epsilon}(m)}(\pt,\xi,w)$.

\end{enumerate}
Then, if $\hat p_m, \hat q_m, \td p_m, \td q_m\in K$,  we have 
\begin{enumlemma}
\item \label{itemclaim:1} $\delta^{-\beta}\cdot \|H^u_{\hat q_m}(\hat z_m)\|\le \|H^s_{\hat p_m}(\hat z_m) \|\le \hat r \delta^{\beta}$ 
\item \label{itemclaim:1'} $\delta^{-\beta} \cdot \|H^u_{\hat p_m}(\hat w_m)\|\le \|H^s_{\hat p_m}(\hat z_m) \|\le \hat r \delta^{\beta}$

\item \label{itemclaim:3} $d(\hat x_m, \hat y_m)<r_0$
\item \label{itemclaim:2} $\|H^u_{\td p_m}(\td w_m)\| \le  \hat r\delta^\beta$
\item \label{itemclaim:4} 
$d(\td x_m, \td y_m)< C_2 \delta^\beta  \hat r+  C_2 L_1  \rexp{\tau_{p, \delta, \epsilon}(m)(\lambda^s+\eps)} r_1.$

\end{enumlemma}
\end{claim}

\begin{proof}
Note  $0<\delta<1$. 
For \ref{itemclaim:1}, first observe that $\|H^u_{q}(z)\|\le 1$. 
We then obtain the lower bound
	\begin{align*}
			\|H^s_{\hat p_m}(\hat z_m) \|
					&\ge \dfrac 1 {M_0}  \lyap{H^s_{\hat p_m}(\hat z_m) }_{\eps,\pm, \hat p_m}^s
					\ge \dfrac 1 {M_0} \rexp{(\lambda^s+\eps)m} \lyap{H^s_{p}(z) }_{\eps,\pm, p}^s
					\\ &\ge   \dfrac 1 {M_0^2} \delta \rexp{(\lambda^s+\eps)m} 
					\ge   \dfrac 1 {M_0^2} \delta \rexp{(\lambda^s+\eps)m_\delta} \\
					&=   \dfrac 1 {M_0^2} \exp \Bigg[\Bigg(
						\dfrac{(\lambda^s+\eps)\left((1+\beta)\log\delta - \log (M_0^4)\right)}{\lambda^u - \lambda^s - 2 \eps} \\ &  \quad   \quad \quad\quad\quad\quad  + (1+\beta)\log \delta - \log (M_0^4)\Bigg) - \beta \log \delta+ \log (M_0^4)\Bigg]\\
						&=   \dfrac {M_0^4 } {M_0^2} \rexp{(\lambda^u -\eps)m_\delta}\delta^{-\beta}
						\ge   \dfrac {M_0^4 } {M_0^3} 	\delta^{-\beta}\lyap{H^u_{\hat q_m}(\hat  z_m)}^u_{\eps, \pm, \hat p_m}
						 \\& \ge \delta^{-\beta}	\|H^u_{\hat q_m}(\hat z_m)\|.
						\end{align*}
The lower bound in \ref{itemclaim:1'} is identical.  
The upper bound in \ref{itemclaim:1} and \ref{itemclaim:1'}  follows since 
		\begin{align*}
			\|H^s_{\hat p_m}(\hat z_m) \|
					&\le {M_0}  \lyap{H^s_{\hat p_m}(\hat z_m) }_{\eps,\pm, \hat p_m}^s
					\le {M_0} \rexp{(-\lambda^s+\eps)|m|} \lyap{H^s_{p}(z) }_{\eps,\pm, p}^s\\
					&\le {M^2_0} \rexp{(-\lambda^s+\eps)|m|} \|H^s_{p}(z) \|
					\le {M^2_0} \rexp{(-\lambda^s+\eps)|M_\delta|} \|H^s_{p}(z) \|\\					
					&\le {M^2_0} \exp\left[ (-\lambda^s+ \eps)\dfrac {-(1-\beta)\log \delta + \log \epsilon}{-\lambda^s +\eps} +\log \delta \right]\\
					&\le {M^2_0} \epsilon\delta^\beta\le \hat r\delta^\beta.
		\end{align*}
		
For \ref{itemclaim:2} we have 
\begin{align*}
	\|H^u_{\td p_m}(\td w_m)\| & 
	\le M_0   \lyap{H^u_{\td p_m}(\td w_m) }_{\eps,\pm, \td p_m}^u\\
	&\le  M_0^2 \lyap {\restrict {D\Phi^{\tau_{p, \delta, \epsilon}(m)}}{E^u(\Phi^m(p))}}_{\eps, \pm, \hat p_m}^u \|H^u_{ \hat p_m} (\hat w_m) \|\\
	&\le  \delta^\beta {M_0^2} \lyap {\restrict {D\Phi^{\tau_{p, \delta, \epsilon}(m)}}{E^u(\Phi^m(p))}}_{\eps, \pm,\hat p_m}^u 
	 \left\| H^s_{\hat p_m} (\hat z_m)\right\|\\
	&\le \delta^\beta {M_0^4}\lyap {\restrict {D\Phi^{\tau_{p, \delta, \epsilon}(m)}}{E^u(\Phi^m(p))}}_{\eps, \pm, \hat p_m}^u 
	\lyap {\restrict {D\Phi^m}{E^s(p)}}_{\eps, \pm, p}^s \delta
	\\
	&\le  \delta^\beta {M_0^4} \epsilon \le \delta^\beta  \hat r.
	\end{align*}
From \ref{itemclaim:1} and Lemma \ref{lem:standardcrap}\ref{item8:3} and \ref{item8:4}  we have the inequality 
$ d(\hat x_m , \hat y_m) \le 2C_2\hat r$.  By the definition of $ \hat r$  we obtain \ref{itemclaim:3}.  

For \ref{itemclaim:4}, recall that $z$ is in the domain of the chart $\phi(\xi,x)$.  We have $$|\phi(\xi,x)(z)| \le C_2\ell_0\delta.$$
Then for $0\le j \le \lfloor {-M_\delta +1}\rfloor$ we have 
 \begin{align*}
 |\phi(F^{-j}(\xi,x))(\cocycle [\xi][-j](z))|
 & \le e^{(-\lambda^s + 2\epone)j}C_2\ell_0\delta\\
  & \le e^{-8\epone j} e^{(-\lambda^s + \eps)j}C_2\ell_0\delta\\
  & \le e^{-8\epone j} e^{(-\lambda^s + \eps)(-M_\delta +1)}C_2\ell_0\delta\\
  &\le e^{-8\epone j}e^{(-\lambda^s + \eps)} C_2 \ell_0 \delta^\beta\epsilon.  
\end{align*}
Since $\epsilon\le \epsilon_0 \le  \frac{e^{\lambda^s - \eps} e^{-\lambda_0-\epone} \ell_0^{-2}}{10 C_2 } $ and  $d(\hat y_m, \hat x_m)<r_0$,  it follows from Corollary \ref{cor:tolemma}  that $\hat w_m\in  \locstabM [ r_1] {\hat y_m} {\hat \xi_m}$.  
 \ref{itemclaim:4} follows as   
\begin{align*}
	d(\td x_m, \td y_m)& \le d(\td x_m, \td w_m)+d(\td w_m, \td y_m)\\
&\le C_2\|H^u_{\td p_m}(\td w_m)\| +  L_1  \rexp{\tau_{p, \delta, \epsilon}(m)(\lambda^s+\eps)} d(\hat y_m, \hat w_m) \\ 
& \le C_2 \delta^\beta  \hat r+   L_1  \rexp{\tau_{p, \delta, \epsilon}(m)(\lambda^s+\eps)}(C_2 r_1).
\qedhere
\end{align*}

\end{proof}

We observe that for
\begin{align*} g(\delta) 
=\dfrac {M_\delta - m_\delta}{M_\delta} =
1 - \dfrac{((1+\beta)\log \delta - \log(M_0^4))(-\lambda^s+  \eps)}{((1-\beta)\log\delta - \log \epsilon)(\lambda^u - \lambda^s - 2\eps)}
\end{align*}
we have  
$$\lim_{\delta\to 0} g(\delta)= 1 - \dfrac{(1+ \beta)(-\lambda^s + \eps)}{(1-\beta)(\lambda^u - \lambda^s - 2 \eps)}=2\alpha_0.$$

We define one final  parameter in addition to those from Section \ref{sec:choices}.  
\begin{choices}
\item For  $T_0>0$ and $\epsilon_0$ fixed in \ref{19} and \ref{20}, given any  $0<\epsilon<\epsilon_0$ 
 we define  $0<\delta_0(T_0, \epsilon)<1$ so that  for all $0<\delta<\delta_0(T_0,\epsilon)$ we have 
\begin{enumerate}
	\item $M_\delta<m_\delta<0$
	\item $\frac {M_\delta - m_\delta}{M_\delta} \ge \alpha_0$
	\item $ M_\delta<-T_0$
	\item $L_{p, \delta, \epsilon}(M_\delta)<-T_0$ for all $p\in Y_0$
	\item $L_{p, \delta, \epsilon}(0)>T_0$ for all $p\in Y_0$
	\item $\tau_{p, \delta, \epsilon} (M_\delta)>\max\{\hat M, \frac {\hat M}{\lambda^u-\eps}\}$ where $\hat M$ was fixed in \ref{17}
	\item $\delta^{-\beta} > D_1$.  
\end{enumerate}
\end{choices}
	
\subsection{
 Key Lemma}
We have the following lemma, whose proof follows from the above choices.  
\begin{lemma}\label{lem:key1}\label{lem:key}
Given $0<\epsilon<\epsilon_0$ 
and any open $U\subset Y$ with $\mualt(U)<\hat \alpha$, there exist sequences
\begin{enumerate}[label=\roman*)]
\item $\td p_j=(\td \pt_j, \td \xi_j,\td  x_j )$
\item $\td q_j=(\td \pt_j, \td \xi_j,\td  y_j )$
\item $p'_j=( \pt'_j, \xi'_j,  x'_j )$
\item $q'_j=( \pt'_j, \xi'_j,  y'_j )$
\item $q''_j=( \pt''_j, \xi''_j,  y''_j )$
\end{enumerate}
such that for all  $j$
	\begin{enumlemma}
		\item \label{lemitem:a} $p'_j, \td p_j, q''_j, \td q_j\in K$;
		\item \label{lemitem:2}$p'_j\notin U$; 

		\item \label{lemitem:0} $q'_j = \Phi^{t_j}(q''_j)$ for some $|t_j|\le \hat T$ where $\hat T$ is as in \ref{11};
		\item \label{lemitem:4} there is $v_j'\in \locunstM [\rhere]{x'_j}{\xi'_j}$ with 
			$\frac 1{C_3M_0^6}\epsilon  \le \| H^u_{p_j'} (v_j') \|\le  C_3M_0^6 \epsilon$
			and $d(y_j', v_j') \to 0$ as $j\to \infty$.
		\end{enumlemma}
		Moreover,
		\begin{enumlemma}[resume]
		\item \label{lemitem:1}$d(\td x_j , \td y_j) \to 0$ as $j\to \infty$ and 

		
		\item \label{lemitem:6}
for every $j$ there are $a_j$ and $b_j$ with   $|a_j|, |b_j| \in [M_0^{-4}, M_0^4]$ with 
$$\bar \mususp_{\td p_j } \simeq \left(\lambda_{a_j} \right)_*\bar \mususp_{p_j'}, \quad \text{ and }\quad \bar \mususp_{\td q_j } \simeq(\lambda_{b_j})_* \bar \mususp_{q_j''}.$$  
		\label{item:lem3}
	\end{enumlemma}

\end{lemma}
In \ref{lemitem:6} above, $\lambda_a\colon \R \to \R$ denotes the multiplication map $\lambda_a \colon t\mapsto at$.  

\input{choicesinendgame}

\subsubsection{Construction of the sequences in Lemma \ref{lem:key1}}

\label{sec:choiceslemma}

Let $0<\epsilon< \epsilon_0$ 
 be fixed and let $U\subset Y$ satisfy $\mualt(U)<\hat \alpha$.  We take this to be the $U$ in \ref{16}. We  construct the sequences in  Lemma \ref{lem:key1} through a sequence of claims and then show they have the desired properties.  

Recall that we assume $\muskew_\xi$ is non-atomic for $\nu$-\ae $\xi$.  It follows that 
  $\muskew_{\xi}$ is not locally supported on $\locunst x {\xi}$ for almost every $(\xi,x)$.  Indeed, as $\mu_\xi$ is assumed non-atomic, if otherwise it would follow that $\muskew^u_{(\xi,x)}$ was not atomic, whence $h_{\muskew}(F\mid \pi)>0$. But then  $\muskew^s_{(\xi,x)}$ would necessarily be non-atomic \as and thus $\muskew_{\xi}$ could  not be locally supported on $\locunst x {\xi}$.  It follows that   $\mususp_{(\pt,\xi)}$ is not locally supported on $\locunst x {\pt,\xi}$
  
Recall $\rec(T_0)$ fixed in \ref{19}.  We fix $p = (\pt,\xi,x)\in \rec(T_0)\subset K$ such that  $p$ is a  $\mususp_{(\pt,\xi)}$-density point of $ \rec(T_0)$ for our fixed $T_0>0$ and such that  $\mususp_{(\pt,\xi)}$ is not locally supported on $\locunst x {\pt,\xi}$.  
It follows that  there exists  a sequence of points $\{y_j\} \subset M$ such that  $q_j = (\pt,\xi, y_j) \in \rec(T_0)\subset K$, 
$d(x,y_j)\le  r_0$, and $y_j\notin \locunst[r_1] x \xi $ for all $j$ and $d(x,y_j)\to 0$ as $j\to \infty$.    
For each $j>0$ set (c.f. Figure \ref{fig:22})
\begin{itemize}
	\item $z_j= \locstabM[r_1] x \xi  \cap  \locunstM[r_1] {y_j} \xi$;
	\item $w_j= \locstabM[r_1] {y_j} \xi  \cap  \locunstM[r_1] x \xi$;
	\item $\delta_j= \|H^s_p(z_j)\|$. 
\end{itemize}
We have  $\delta_j>0$ for all $j$ and $\delta_j\to 0$.    By omission, we  may assume 
$\delta_j< \delta_0(T_0, \epsilon)$ for  all $j$.


We select a sequence of times $\{m_j\}$ satisfying the  following claim.  Recall $\scrA$ defined in \ref{1442}.  
\begin{claim}\label{prop:densities}
Writing $\delta= \delta_j$, there exists an $m\in [M_\delta, m_\delta]$ with 
\begin{enumerate}
	\item $\Phi^m(p)\in \scrA\cap K\cap S_0$ and $\Phi^m(q_j)\in K\cap S_0$
	\item $\Phi^{L_{p, \delta, \epsilon}(m)}(p) \in S_{\hat M}\cap K$ and $\Phi^{L_{p, \delta, \epsilon}(m)}(q_j) \in S_{\hat M}\cap K$
\end{enumerate}
\end{claim}

\begin{proof}[Proof of Claim \ref{prop:densities}]
Let \begin{enumerate}
		\item $F_1 = \left\{ t\in \R : \Phi^t(p) \in K\cap \scrA\cap S_0\right\}$;
		\item $F_2 = \left\{ t\in \R : \Phi^t(q_j) \in K \cap S_0\right\}$;
		\item $F_3 = \left\{ t\in \R : \Phi^t(p) \in K\cap S_{\hat M}\right\}$;
		\item $F_4 = \left\{ t\in \R : \Phi^t(q_j) \in K\cap S_{\hat M}\right\}.$
\end{enumerate}

Write $L = L_{p,\delta, \epsilon}$.  
Since $p,q_j\in \rec(T_0)$,  $M_\delta\le -T_0$, and  $\frac{M_\delta-m_\delta}{M_\delta}>\alpha_0$, we have 
\begin{equation}\leb\left( [M_\delta, m_\delta ] \cap F_1\cap F_2\right)\ge (\alpha_0 - 5\alpha)|M_\delta|.\end{equation}

Furthermore, as $[-T_0, T_0]\subset L \left([M_\delta, 0]\right)$, we have
$$\leb \left(L\left([M_\delta, 0]\right)\cap F_3\cap F_4\right)
		\ge (1-4\alpha) \leb\left(L\left([M_\delta, 0]\right)\right)$$
hence, by Lemma \ref{lem:bilipest} and Claim \ref{claim:bilipRND},  $$\leb \left(L\left([M_\delta, 0]\right)\sm  (F_3\cap  F_4)\right)
		\le  (4\alpha) \leb\left(L\left([M_\delta, 0]\right)\right)\le 4\alpha \kappa_2 |M_\delta|.$$
Then, $$\Leb \left([M_\delta, 0] \sm  L\inv\left( F_3\cap F_4\right)\right)\le 4\alpha \kappa_1\inv  \kappa_2 |M_\delta|.$$ 
Thus, 
$$\leb \left ([M_\delta, m_\delta ] \cap F_1\cap F_2  \cap L\inv _{p, \delta, \epsilon}(F_3) \cap L\inv _{p, \delta, \epsilon}(F_4)\right)\ge (\alpha_0 - 5\alpha - 4\alpha \kappa_1\inv  \kappa_2)  |M_\delta|.$$

Our choice of $\alpha$ ensures $\alpha_0 - 5\alpha - 4\alpha \kappa_1\inv  \kappa_2>0$.  
\end{proof}

For each $j$, select a  $m_j<0$ satisfying Claim \ref{prop:densities}.  We define  
$\td p_j$ and $\td q_j$ 
satisfying the conclusions in Lemma \ref{lem:key}  by 
\begin{itemize}
\item 
$\td p_j=(\td \pt_j, \td \xi_j, \td  x_j )= \Phi^{L_{p,\delta_j}(m_j)}(p); $ 
\item $
 \td q_j =(\td \pt_j, \td \xi_j, \td  y_j )=  \Phi^{L_{p,\delta_j}(m_j)}(q_j).$ 
\end{itemize}

We also define
\begin{itemize} 
\item 
$\hat p_j=(\hat \pt_j, \hat \xi_j,\hat  x_j )= \Phi^{m_j}(p);
$ $ \hat q_j =(\hat \pt_j, \hat \xi_j,\hat  y_j )=  \Phi^{m_j}(q_j);$
\item 
$s'_j = \ess _{\hat p_j}(\tau_{p,\delta_j,\epsilon}(m_j));$
$ s''_j = \ess _{\hat q_j}(\tau_{p,\delta_j,\epsilon}(m_j)).$
\end{itemize}
Then  $\td p_j = \Psi^{s_j'}\left(\hat p_j\right) $ and $ \td q_j = \Psi^{s_j''}\left(\hat q_j\right).$
%

With the above choices, for each $j$ we choose a $\hat \eta_j$ satisfying the following.
\begin{claim}\label{claim:eta}
Given $p, q_j,$ and $m_j$ as above, for each $j$ there exists $\eta\in \Omega$ with 
	\begin{enumlemma}
		\item $(\hat \pt_j, \eta, \hat x_j ) \in  K \cap 
			\Psi^{-s'_j}(K\sm U)$ \label{choiceetat1};
		\item $(\hat \pt_j, \eta, \hat y_j ) \in  K\cap \Psi^{-s''_j}(K) $; 
		\item $\eta \in (A_{\gamma_2}(\hat p_j))$.
	\end{enumlemma}
	Furthermore, we may choose $\eta$ so that 
	\begin{enumlemma}[resume]
	\item $\cocycle = \cocycle [\eta][n]$ for all $n\le 0$;\label{itemclaim3:10}
\item $\unstM {\hat x_j } {{\eta}}  = \unstM {\hat x_j } {{\hat \xi_j}} $ and  $\bar \mususp_{\hat p_j} = \bar \mususp_{(\hat \pt_j, \eta, \hat x_j )}$;\label{itemclaim3:4}
		\item $\unstM {\hat y_j } {{\eta}} = \unstM {\hat y_j } {{\hat \xi_j} }$ and  $\bar \mususp_{\hat q_j} = \bar \mususp_{(\hat \pt_j, \eta, \hat y_j )}$.\label{itemclaim3:5}

	\end{enumlemma}

\end{claim}
\begin{proof}[Proof of Claim \ref{claim:eta}]
We have $\hat p_j=(\hat \pt_j, \hat \xi_j,\hat  x_j )$ and   $\hat q_j=(\hat \pt_j, \hat \xi_j,\hat  y_j )$ in $K\subset Y_0$.  
Then \ref{itemclaim3:10}--\ref{itemclaim3:5} hold  for $\nucondF_{\hat \xi_j}$-\ae  $\eta$. 

Recall for $p=(\pt,\xi,x)$ we have $\mucondS _p= \delta_\pt \times \nucondF_\xi \times \delta_x$ 
as discussed in Section \ref{sec:CondMeas}.
Since $\hat p_j \in S_0 \cap \scrA,  \hat q_j\in S_0$ we have
	\begin{enumerate}
	\item $\mucondS_{\hat p_j} (K) \ge .9$
	\item $\mucondS_{\hat q_j} (K) \ge .9$	
	\item $\nucondF_{\hat \xi_j}(A_{\gamma_2}(\hat p_j))\ge .9$.
\end{enumerate}
Furthermore, since $\td p_j, \td q_j \in S_{\hat M}$,  and since 
 $$ s_j'\ge (\lambda^u-\eps)\tau_{p, \delta, \epsilon}(m_j)\ge (\lambda^u-\eps) \tau_{p, \delta, \epsilon}(M_{\delta_j})\ge \hat M$$ and similarly $s''_j>\hat M$, we have by \ref{17}

 \begin{enumerate}[resume]
	\item $\mucondS_{\hat p_j} \left(\Psi^{-s_j'}(K\sm U)\right) \ge .9$
	\item $\mucondS_{\hat q_j}  \left( \Psi^{-s_j''}(K)\right) \ge .9$.
	\end{enumerate}
From the natural identification of $\mucondS_{\hat p_j}$ and $\mucondS_{\hat q_j}$ with $\nucondF_{\hat \xi_j}$, it follows that the set of $\eta$ satisfying the conclusions of the claim has $\nucondF_{\hat \xi_j}$-measure at least $.5$.  
\end{proof}	

\subsubsection{Proof of Lemma \ref{lem:key1}}
Having selected $p, {y_j}, {m_j}$ and ${\hat \eta_j}$ above, we define 
\begin{itemize}
\item $(\td \pt_j, \td \xi_j, \td z_j )=  \Phi^{L_{p, \delta, \epsilon}(m_j)}(\pt,\xi,z_j)$, and 
$(\td \pt_j, \td \xi_j, \td w_j )=  \Phi^{L_{p, \delta, \epsilon}(m_j)}(\pt,\xi,w_j)$;
\item $(\hat \pt_j, \hat \xi_j, \hat z_j) = \Phi^{m_j} (\pt,\xi, z_j)$, and $(\hat \pt_j, \hat \xi_j, \hat w_j) = \Phi^{m_j} (\pt,\xi, w_j)$;
\item $\bar p_j = (\hat \pt_j, \hat \eta_j, \hat x_j)$, and $\bar q_j = (\hat \pt_j, \hat \eta_j, \hat y_j)$;
\item  $ t_j' = \ess_{\bar p_j }\inv(s_j')$, and $t''_j =\ess_{\bar q_j}\inv(s_j'')$.

\end{itemize}
We show Lemma \ref{lem:key1} holds with $\td p_j, \td q_j$ defined above and 
  \begin{itemize}
\item $p'_j =  (\pt'_j, \xi'_j, x'_j) := \Psi^{s'_j}(\bar p_j)=  \Phi^{t'_j}(\bar p_j)$; 
\item
 $q'_j = (\pt'_j, \xi'_j, y'_j) := \Phi^{t'_j}(\bar q_j)$; 
\item $q''_j = (\pt''_j, \xi''_j, y''_j): = \Psi^{s''_j}(\bar q_j)=  \Phi^{t''_j}(\bar q_j).$  
\end{itemize}

\begin{proof}[Proof of Lemma \ref{lem:key}]
Part \ref{lemitem:a} of Lemma   \ref{lem:key} follows from the selection procedure in the above claims.   Part \ref{lemitem:2} follows from Claim \ref{claim:eta}\ref{choiceetat1}.  
Part \ref{lemitem:1} follows immediately from Claim \ref{claim:GP}\ref{itemclaim:4} since as $j\to \infty$, $\delta_j\to 0$ and $\tau_{p,\delta, \epsilon}(m_j)\ge  \tau_{p,\delta, \epsilon}(M_{\delta_j})\to \infty$.   

By Claim \ref{claim:GP}\ref{itemclaim:3} we have 
 $d(\hat y_j, \hat x_j)< r_0$.  By Lemma \ref{lem:standardcrap}, and the fact that $(\hat \pt_j,\hat \eta_j, \hat x_j)$ and $(\hat \pt_j,\hat \eta_j, \hat y_j)$ are in $K\subset [0,1)\times \Lambda'$, we  define $\hat v_j$ to be  the point of intersection  $$\hat v_j  = \locunstM [ r_1] {\hat x_j} {\hat \xi_j}  \cap \locstabM [ r_1]{\hat y_j} {\hat \eta_j}.$$   
 From Lemma \ref{prop:controlledintersections}, Claim \ref{claim:GP}\ref{itemclaim:1}, and  the fact that $\delta_j^{-\beta}\ge D_1$, for each $j$ we have 
$$ \dfrac{1}{C_3} \|H^s_{\hat p_j}(\hat z_j) \| \le \|H^u_{\hat p_j}(\hat v_j)\|\le C_3 \|H^s_{\hat p_j}(\hat z_j) \|.$$

Recall that $\locunstM [ r_1] {\hat x_j} {\hat \xi_j} = \locunstM [ r_1] {\hat x_j} {\hat \eta_j} $ and $$ \|H^u_{\hat p_j}(\hat v_j)\|= \|H^u_{\bar p_j}(\hat v_j) \|.$$
We define $v'_j$ in  Lemma \ref{lem:key}\ref{lemitem:4} by $$
(\pt'_j, \xi'_j, v'_j) = \Phi^{t_j'}(\hat \pt_j, \hat \eta_j, \hat v_j).$$
We claim
\begin{claim}\label{claim:annulus}
$\frac 1{C_3M_0^6}\epsilon 
\le \| H^u_{p_j'} (v_j') \|\le 
C_3M_0^6 \epsilon.$
\end{claim}

\begin{proof}[Proof of Claim \ref{claim:annulus}]
We have the upper bound 

\begin{align*}
\| H^u_{p_j'} (v_j') \|
		&\le  M_0 \lyap{ H^u_{p_j'} (v_j') }_{\eps,-}\\
		&= M_0 \lyap{ H^u_{\bar p_j} (\hat v_j) }_{\eps,-} 
			\lyap{ \restrict {D\Phi^{t_j}}{E^u(\bar p_j)}}_{\eps,-}= M_0 \lyap{ H^u_{\bar p_j} (\hat v_j) }_{\eps,-} \rexp{s_j'}\\
		&\le M_0^2 \| H^u_{\bar p_j} (\hat v_j) \| \rexp{s_j'} = M_0^2 \| H^u_{\hat p_j} (\hat v_j) \| \rexp{s_j'} \le M_0^2 C_3 \|H^s_{\hat p_j}(\hat z_j) \| \rexp{s_j'}\\	
		&\le M_0^3 C_3 \lyap{ H^s_{\hat p_j}(\hat z_j)  }_{\eps,\pm} \rexp{s_j'}\\
		&= M_0^3C_3 \lyap{ H^s_{\hat p_j}(\hat z_j)  }_{\eps,\pm} 
			\lyap{ \restrict {D\Phi^{\tau_{p,\delta_j,\epsilon}(m_j)}}{E^u(\hat p_j)}}_{\eps,-}\\
		&\le M_0^5 C_3\lyap{ H^s_{\hat p_j}(\hat z_j)  }_{\eps,\pm} 
		\lyap{ \restrict {D\Phi^{\tau_{p,\delta_j,\epsilon}(m_j)}}{E^u(\hat p_j)}}_{\eps,\pm}\\
		&\le M_0^6 C_3\lyap{ \restrict {D\Phi^{m_j}}{E^s( p)}}_{\eps,\pm}\delta_j
		\lyap{ \restrict {D\Phi^{\tau_{p,\delta_j,\epsilon}(m_j)}}{E^u(\hat p_j)}}_{\eps,\pm}\\
		&= M_0^6 C_3\epsilon.
\end{align*}
The lower bound is identical.  
\end{proof}

As $\bar q_j\in K$ 
 we have 
	\begin{align*}
			 d(y'_j, v'_j) 
			\le L_1 e^{t'_j(\lambda^s + \eps)} d(\hat y_j, \hat v_j)\\		
			\le C_2L_1 e^{t'_j(\lambda^s + \eps)} r_1.
	\end{align*}
	
Since 
$\bar q_j\in K$ for each $j$ and since $s'_j, 
\to \infty$, by the upper bound in  \eqref{eq:controlled} and the fact that $a$ in \eqref{eq:controlled} is bounded on $K$  we have $t_j'\to \infty $ as $j\to \infty$.  
Thus, $d(y'_j, v'_j) 
 \to 0$ as $j\to \infty$ completing the proof of  Lemma   \ref{lem:key}\ref{lemitem:4}.

To derive the bound in Lemma \ref{lem:key}\ref{lemitem:0},  first consider the case $t'_j\ge t''_j$.  As $\bar p_j\in K$,  by the lower bound in \eqref{eq:subcomp} we have 
$$\lyap {\restrict{D\Phi^{t''_j}}{E^u(\bar p _j)}}_{\eps, -}\ge \hat L\inv \|\restrict{D\Phi^{t''_j}}{E^u(\bar p _j)}\|. $$
Moreover, as $\bar q _j$ and $q''_j = \Phi^{t''_j}(\bar q_j)$ are in $K$, 
$$ \|\restrict{D\Phi^{t''_j}}{E^u(\bar q _j)}\| \ge \dfrac 1 {M_0^2}\lyap {\restrict{D\Phi^{t''_j}}{E^u(\bar q _j)}}_{\eps, -}. $$

Write $p''_j = (\pt_j'', \xi''_j, x''_j):= \Phi^{t_j''}(\bar p_j)$ and $(\pt_j'', \xi''_j, v''_j) := \Phi^{t_j''}(\hat \pt_j, \hat \eta_j, \hat v_j).$
For  $n'= \lfloor \hat \pt_j+t'_j\rfloor\ge n''= \lfloor \hat \pt_j+t''_j\rfloor\ge0$  we have 

\begin{align}
&\dfrac {\| \restrict { D\Phi^{t''_j}}{E^u(\bar q_j)}\|} 
	{\| \restrict { D\Phi^{t''_j}}{E^u(\bar p_j)}\|}
 		= 
	\dfrac {\| \restrict { D\cocycle[\hat \eta_j][n'']}{T_{\hat y_j} \unstM   {\hat y_j} {\hat \eta_j}}\|}
		{\| \restrict { D\cocycle[\hat \eta_j][n'']}{T_{\hat x_j} \unstM   {\hat x_j} {\hat \eta_j}}\|}
\notag\\
		&
\quad		= 
	\dfrac {\| \restrict { D\cocycle[\hat \eta_j][n'']}{T_{\hat y_j} \unstM   {\hat y_j} {\hat \eta_j}}\|}
		{\| \restrict { D\cocycle[\hat \eta_j][n'']}{T_{\hat v_j} \unstM   {\hat x_j} {\hat \eta_j}}\|}
			\cdot
	\dfrac {\| \restrict { D\cocycle[\xi''_j][-n'']}{T_{ x''_j} \unstM  { x''_j} { \xi''_j}}\|}
		{\| \restrict { D\cocycle[\xi''_j][-n'']}{T_{ v_j''} \unstM { x''_j} { \xi''_j}}\|}
\notag \\
		&
\quad	= 
	\dfrac {\| \restrict { D\cocycle[\hat \eta_j][n'']}{T_{\hat y_j} \unstM   {\hat y_j} {\hat \eta_j}}\|}
		{\| \restrict { D\cocycle[\hat \eta_j][n'']}{T_{\hat v_j} \unstM   {\hat x_j} {\hat \eta_j}}\|}
			\cdot
	\dfrac {\| \restrict { D\cocycle[ \xi'_j][-n']}{T_{ x_j'} \unstM   { x'_j} { \xi'_j}}\|}
		{\| \restrict { D\cocycle[ \xi'_j][-n']}{T_{ v_j'} \unstM   { x'_j} { \xi'_j}}\|}
		\cdot 
	\dfrac{\| \restrict { D\cocycle[ \xi'_j][-(n'- n'')]}{T_{ v_j'} \unstM   { x'_j} { \xi'_j}}\|}
	 {\| \restrict { D\cocycle[ \xi'_j][-(n'-n'')]}{T_{ x_j'} \unstM   { x'_j} { \xi'_j}}\|}				
%
  \label{eq:thisguy}
\end{align}

As $p'_j, \bar q_j\in K$ we have $(\hat \eta_j, \hat y_j)\in \Lambda'$ and $(\xi'_j, x'_j)\in \Lambda'$.  Moreover, as $\|H^u_{p'_j}(v'_j)\|\le r_1$
from Lemma \ref{prop:controlledintersections}\ref{item8:1} and \ref{item8:2} we have that \eqref{eq:thisguy} is bounded above by $C_1^3$.  Thus 
$$\lyap {\restrict{D\Phi^{t''_j}}{E^u(\bar p _j)}}_{\eps, -}\ge \dfrac 1 {C_1^3 M_0^2 \hat L}\lyap {\restrict{D\Phi^{t''_j}}{E^u(\bar q _j)}}_{\eps, -}. $$
As 
\begin{align*}
\lyap {\restrict{D\Phi^{t''_j}}{E^u(\bar q _j)}}_{\eps, -} & = 
\lyap {\restrict{D\Phi^{t'_j}}{E^u(\bar p _j)}}_{\eps, -} \ge e^{(\lambda^u - \eps)(t'_j -t''_j)}
\lyap {\restrict{D\Phi^{t''_j}}{E^u(\bar p _j)}}_{\eps, -} \\&\ge e^{(\lambda^u - \eps)(t'_j -t''_j)} \dfrac 1 {C_1^3 M_0^2 \hat L}\lyap {\restrict{D\Phi^{t''_j}}{E^u(\bar q _j)}}_{\eps, -} \end{align*}
it follows that 
$$t'_j - t''_j \le \dfrac {\log(C_1^3 M_0^2 \hat L)} {\lambda^u - \eps}= \hat T.$$

If $t''_j\ge t'_j$  we similarly have 
$$\lyap {\restrict{D\Phi^{t'_j}}{E^u(\bar q _j)}}_{\eps, -}\ge \hat L\inv \|\restrict{D\Phi^{t'_j}}{E^u(\bar q _j)}\|, 
\quad  \|\restrict{D\Phi^{t'_j}}{E^u(\bar p _j)}\| \ge \dfrac 1 {M_0^2}\lyap {\restrict{D\Phi^{t'_j}}{E^u(\bar p_j)}}_{\eps, -}. $$

Then with 
$n'= \lfloor \hat \pt_j+t'_j\rfloor\ge  0$  we have 
\begin{align}
\dfrac {\| \restrict { D\Phi^{t'_j}}{E^u(\bar q_j)}\|} 
	{\| \restrict { D\Phi^{t'_j}}{E^u(\bar p_j)}\|}
		&
		= 
	\dfrac {\| \restrict { D\cocycle[\hat \eta_j][n']}{T_{\hat y_j} \unstM  {\hat y_j} {\hat \eta_j}}\|}
		{\| \restrict { D\cocycle[\hat \eta_j][n']}{T_{\hat x_j} \unstM  {\hat x_j} {\hat \eta_j}}\|}
\notag\\
		&
		= 
	\dfrac {\| \restrict { D\cocycle[\hat \eta_j][n']}{T_{\hat y_j} \unstM  {\hat y_j} {\hat \eta_j}}\|}
		{\| \restrict { D\cocycle[\hat \eta_j][n']}{T_{\hat v_j} \unstM  {\hat x_j} {\hat \eta_j}}\|}
			\cdot
	\dfrac {\| \restrict { D\cocycle[\xi'_j][-n']}{T_{ x'_j} \unstM  { x'_j} { \xi'_j}}\|}
		{\| \restrict { D\cocycle[\xi'_j][-n']}{T_{ v_j'} \unstM  { x'_j} { \xi'_j}}\|}
  \label{eq:thisguy2}
\end{align}
and, as above,  by Lemma \ref{prop:controlledintersections}\ref{item8:1} and \ref{item8:2} equation  \eqref{eq:thisguy2} is bounded below by $\dfrac 1{C_1^2}$.  Thus 
$$\lyap {\restrict{D\Phi^{t'_j}}{E^u(\bar q _j)}}_{\eps, -}\ge \dfrac 1 {C_1^2 M_0^2 \hat L}\lyap {\restrict{D\Phi^{t'_j}}{E^u(\bar p _j)}}_{\eps, -}$$
and the same analysis as above gives 
$$t''_j - t'_j \le \dfrac {\log(C_1^2 M_0^2 \hat L)} {\lambda^u - \eps}\le \hat T.$$

Finally,  for Lemma \ref{lem:key}\ref{lemitem:6} we have 
$$\bar \mususp_{\td p_j} \simeq \left( \lambda_{\pm  \|\restrict{D \Phi^{\tau_{p, \delta_j,\epsilon}(m_j)}}{E^u(\hat p_j)}\|}\right)_* \bar \mususp_{\hat p_j}= 
\left(\lambda_{\pm  \|\restrict{D \Psi^{s'_j}}{E^u(\hat p_j)} \|}\right)_* \bar \mususp_{\hat p_j}$$
where the sign depends on whether or not $\restrict{D \Phi^{\tau_{p, \delta_j, \epsilon}(m_j)}}{E^u(\hat p_j)}\colon {E^u(\hat p_j)}\to {E^u(\td p_j)}$ preserves   orientation.  
We similarly have
$$\bar \mususp_{p_j'} \simeq \left(\lambda_{\pm  \|\restrict{D \Phi^{t_j'}}{E^u(\bar p_j)}\|}\right)_* \bar \mususp_{\bar p_j}= 
 \left(\lambda_{\pm  \|\restrict{D \Psi^{s'_j}}{E^u(\bar p_j)} \| }\right)_*\bar \mususp_{\bar p_j}.$$
Since $ \bar \mususp_{\bar p_j} =  \bar \mususp_{\hat p_j}$ we have 
$$\bar \mususp_{p_j'}  \simeq\left( \lambda_{a_j}\right)_* \  \bar \mususp_{\td p_j} $$
 where 
 
 \begin{align*}|a_j| = \dfrac{ 
   \|\restrict{D \Psi^{s'_j}}{E^u(\bar p_j)} \|}{
 \|\restrict{D \Psi^{s'_j}}{E^u(\hat p_j)} \| 
}\le M^4_0\dfrac{ 
   \lyapno{\restrict{D \Psi^{s'_j}}{E^u(\bar p_j)} }_{\eps, -}
  }
  {
 \lyapno{\restrict{D \Psi^{s'_j}}{E^u(\hat p_j)} }_{\eps,-}}
= M^4_0\end{align*}
proving the upper bound in Lemma \ref{lem:key}\ref{lemitem:6}.  The lower bound on $|a_j|$ and  the existence of $b_j$ and its bounds are similar.

\end{proof}

\subsection{
Proof of Proposition \ref{lem:main}}\label{sec:finallytheproof} 
We show Proposition \ref{lem:main}  follows with 
\begin{align}\label{eq:M}
M : =  C_1M_0^{14}C_3D_0 
\end{align}
where  $M_0$ is as in \ref{M0}, $C_1,$ and $C_3$ are as in  \ref{13}, and $ D_0$ is as in \ref{11}. 
For $ \epsilon<\epsilon_0$ define the set $G_\epsilon$ as in Remark \ref{rem:posmeasure}.  Set $\td G_\epsilon = [0,1)\times G_\epsilon$.   
Consider $\td G_\epsilon\cap K$.  Were $\mualt (\td G_\epsilon\cap K)<\hat \alpha$ there would exist an open $U\supset (\td G_\epsilon \cap K)$ with $\mualt (U)<\hat \alpha$.  With such a $U$ we obtain a sequence ${p'_j\in K}$ satisfying the conclusions of  Lemma \ref{lem:key}.  We have the following.
\begin{lemma}\label{lem:here2}
Let $p$ be an accumulation point of $\{p'_j\}$.  Then $p\in \td G_\epsilon$.  
\end{lemma}
\noindent On the other hand, as $p'_j\notin U$ for every $j$, we have $p\notin U$.  This yields a contradiction showing $\mualt (\td G_\epsilon\cap K)\ge \hat \alpha$ for all  $\epsilon< \epsilon_0 $. 
 Then for $G$  defined as in Proposition \ref{lem:main} and for 
$$G \times [0,1) =\{ p\mid p\in \td G_{1/N} \text{ for infinitely many $N$} \} =:\td G$$
we have that $\mualt (\td G)\ge \hat \alpha$ and hence, as $\mususp$ and $\mualt$ are equivalent measures,  $\mususp (\td G)>0$.  Then  $\muskew (G)>0$ and Proposition \ref{lem:main} follows.

We prove the lemma,  concluding the proof of Proposition \ref{lem:main}.  

%
%
\begin{proof}[Proof of Lemma \ref{lem:here2}]

\def\ponej{\td p_j}
\def\qonej{\td q_j}
 \def\ptwoj{p'_j} 
\def\qtwoj{q''_j}
\def\qmidj{q'_j}
 \def\zj{z_j}
 \def\wj{w_j} 
 \def\oj{o_j}

\def\oddlim{\lim_{j\in B \to \infty}}
With $U$ as above, we recall all notation from Lemma \ref{lem:key}.  
We have each $\ptwoj$ and $\qtwoj$ is contained in the compact set $K$.  Let $p_0\in G_\epsilon$ be an accumulation point of $\{\ptwoj\}$.
We may restrict to an infinite subset $B\subset   \N_0$ such that 
$\displaystyle{\oddlim}\ptwoj = p_0= (\pt_0, \xi_0, x_0)$.  Further restricting $B$ we may assume 
 that the sequence $(\qtwoj)_{j\in B}$ converges.  Let $q_1=  \displaystyle{\oddlim} \qtwoj .$

Recall that $\Phi^{t_j}(\qtwoj) = \qmidj $ for some $|t_j|\le \hat T$.  We may assume $(t_j)_{j\in B}$ converges.  
Note that $\qmidj$ is not assumed to be contained in $K$.  However, as $p'_j = (\pt'_j, \xi'_j, x'_j)_{j\in B}$ converges we have $\pt'_j\in [0,1-a]$ for some $a>0$ and all $j\in B$.  As $\qmidj = (\pt'_j, \xi'_j, y'_j)$, 
from 
 \ref{12} we have that $\qmidj = \Phi^{t_j}(\qtwoj)$ converges to  $q_0 = (\pt_0, \xi_0, y_0) =\Phi^{\hat t}(q_1)$ for some $|\hat t|\le \hat T$.  See Figure \ref{fig:2}.

Note that $q_1 \in K$, and by Lemma \ref{lem:key}\ref{lemitem:4},   $q_0\in \locunstp[r_1]{p_0}$  
and 			$$\frac 1{C_3M_0^6}\epsilon  \le \| H^u_{p_0} (y_0) \|\le  C_3M_0^6 \epsilon.$$
We need not have  $q_0\in K$.  However---as $q_1\in K\subset Y_0$,  $q_0 = \Phi^{\hat t} (q_1)$, and   $Y_0$ is  $\Phi^t$-invariant---we have $q_0\in Y_0$.  Thus, the unstable line field $E^u(q_0)$,  unstable manifold $\unstp {q_0} = \unstp {p_0}$,  trivialization $\I^u_{q_0}$, affine parameters $H^u_{q_0}$, and measure $\bar \omega_{q_0}$ are defined at $q_0$.  

Fix $\gamma := d(\left(\I^u_{q_0}\right)\inv  \circ \I^u_{p_0} (t))/dt(0)$ and let $v:= (\I^u_{p_0})\inv(y_0)$
where $\I^u_{p}$ is defined in \eqref{eq:affparammfolds}.  
As $p_0\in K$, by Proposition \ref{prop:Stabman} and \eqref{eq:rho} (applied to unstable manifolds), and Lemma \ref{prop:controlledintersections}\ref{item8:1} we have $C_1\inv \le |\gamma| \le C_1$.
We also have $(C_3M_0^6)\inv \epsilon \le |v| \le C_3M_0^6\epsilon.$
Define the map $\phi\colon \R\to \R$  by $$\phi\colon t \mapsto \gamma (t- v).$$  By  construction,  we have
\begin{equation}\label{eq:half}\phi_*\bar \mususp_{p_0} \simeq \bar \mususp_{q_0}.\end{equation}


Recall that given $\alpha\in \R$ we write $\lambda_{\alpha}\colon \R\to \R$ for the 
linear map  $\lambda_\alpha\colon x \mapsto \alpha x$.
 Let $\beta:= \|\restrict{D\Phi^{{\hat t}}}{E^u(q_1)}\|$. 
As $q_1\in K$ we have $\frac{1}{ D_0}\le \beta\le D_0$.  Also, $\bar \mususp_{q_0}\simeq (\lambda_{\pm \beta})_*\bar \mususp_{q_1}$
where the sign  depends on whether or not $\restrict{D\Phi^{{\hat t}}}{E^u(q_1)}\colon {E^u(q_1)}\to {E^u(q_0)}$ preserves orientation.  
It remains to relate the measures $\bar \mususp_{p_0}$  and $\bar \mususp_{q_1}$.  

\def\oddlim{\lim_{j\in B \to \infty}}
%
Let $a_j$ and $b_j$ be as in  Lemma \ref{lem:key}\ref{lemitem:6}.  
 We  further restrict the set $B\subset \N$ 
 so that the limits 
$$\displaystyle{\oddlim} a_j = a, \quad \quad \displaystyle{\oddlim}b_j = b$$ are defined.

We claim that $(\lambda_a)_*\bar\mususp_{p_0}\simeq (\lambda_b)_*\bar\mususp_{q_1}.$
 Indeed, for all $j$ we have 
$$(\lambda_{a_j})_* \bar\mususp_{\ptwoj} \simeq \bar\mususp_{\ponej}, \quad \quad (\lambda_{b_j})_* \bar\mususp_{\qtwoj} \simeq \bar\mususp_{\qonej}.$$
We introduce normalization factors
$$c_j := \bar\mususp_{\ptwoj}([-a_j\inv, a_j\inv]) \inv, \quad \quad d_j := \bar\mususp_{\qtwoj}([-b_j\inv, b_j\inv])\inv$$ and 
$$c := \bar\mususp_{p_0}([-a\inv, a\inv])\inv, \quad \quad d := \bar\mususp_{q_0}([-b\inv, b\inv])\inv.$$
We remark that for $q\in Y_0$, the measure $\bar \mususp_{q}$ has at most one atom which by assumption is at $0$. It follows that non-trivial  intervals centered at $0$ are continuity sets for each $\bar \mususp_{q}$ and thus $c_j\to c$ and $d_j\to d$.  
Let $f$ be a continuous, compactly supported function $f\colon \R \to \R$.  We note that $q\mapsto \bar\mususp_{q}(f)$ is uniformly continuous 
 on $K$ and that
$$\left | (\lambda_{a})_*\bar\mususp_{q}(f) 
			-(\lambda_{a_j})_*\bar\mususp_{q}(f)  \right| 
=  \left | \int f(a t) - f( {a_j} t)  \ d\bar\mususp_{q}(t) \right| $$
approaches zero uniformly in $q$ as $j\in B\to \infty$.
Thus for any $\kappa>0$ and for all sufficiently large $j\in B$ we have 
\begin{itemize}
\item $ | {c}(\lambda_a)_*\bar\mususp_{p_0}(f)-  {c}(\lambda_a)_*\bar\mususp_{\ptwoj}(f)| \le \kappa$,
\item $ |  {c}(\lambda_a)_*\bar\mususp_{\ptwoj}(f)
			- {c_j}(\lambda_{a_j})_*\bar\mususp_{\ptwoj}(f) | \le \kappa$,
			
\item $ | {d}(\lambda_b)_*\bar\mususp_{q_1}(f)-  {d}(\lambda_b)_*\bar\mususp_{\qtwoj}(f)| \le \kappa$,
\item $ |  {d}(\lambda_b)_*\bar\mususp_{\qtwoj}(f)
			- {d_j}(\lambda_{b_j})_*\bar\mususp_{\qtwoj}(f) | \le \kappa$,
\item  $|  \bar\mususp_{\ponej} (f)- \bar\mususp_{\qonej}(f)|\le \kappa$
 \end{itemize}
 where the final estimate follows since $\ponej$ and  $\qonej$ become arbitrarily close in $K\subset Y$ as $j\in B\to \infty$ by Lemma \ref{lem:key}\ref{lemitem:1}.  
 
 Since
$$ {c_j}(\lambda_{a_j})_*\bar\mususp_{\ptwoj}(f)  =  \bar\mususp_{\ponej}(f),\quad \quad  {d_j}(\lambda_{b_j})_*\bar\mususp_{\qtwoj}(f) =  \bar\mususp_{\qonej}(f)$$
we conclude 
$c(\lambda_a)_*\bar\mususp_{p_0} = d(\lambda_b)_*\bar\mususp_{q_1},$
or $$\bar \mususp_{p_0}\simeq  (\lambda_{b/a})_*\bar\mususp_{q_1}.$$

Combining the above  with \eqref{eq:half}, it follows that map (with the appropriate sign discussed above)
$$\psi=(\lambda_{b/a})\circ \lambda_{\pm \beta\inv}\circ \phi\colon t \mapsto \pm \dfrac{b  \gamma}{\beta a} (t-v)$$
satisfies 
$$\psi_* \bar \mususp_{p_0} \simeq \bar \mususp_{p_0}.$$
It follows that $p_0 \in \td G_\epsilon$.  
\end{proof}

\section{Geometry of the stable support of stationary measures}
In this and the following sections, we return to the special case where $\Omega = \Sigma$ to prove Theorem \ref{thm:mainAtominv}.  Recall the measure  $\muskew$ constructed by Proposition \ref{prop:mudef}.  We show that if the fiber-wise measures $\mu_\xi$ are finitely supported and if $(\xi, x) \mapsto E^s_\xi(x)$ is not $\F$-measurable then the stationary measure $\munaught$ is finitely supported and hence $\nunaught$-\as invariant.  
This result is analogous to \cite[Lemmas 3.10, 3.11]{MR2831114} but our methods of proof are completely different.  

In Section \ref{sec:atom}, assuming $(\xi, x) \mapsto E^s_\xi(x)$ is not $\F$-measurable, we show that the fiber-wise measures are non-atomic under the additional assumption that the conditional measures along total stable sets (in both the $\Sigmalocs$ and fiber-wise stable directions) satisfy  a certain geometric criterion.
In this section we consider the case in which the geometric criterion mentioned above fails.  
This degenerate case forces some rigidity of the measure $\mu$ which implies that the 
 stationary measure $\munaught$ is  $\nunaught$-\as invariant.

We remark however that in this section we do not use the fact that $M$ is a surface though we still require that the 
stationary measure  $\munaught$ be hyperbolic to obtain Lemma \ref{lem:existsnicepart} below.
Thus, for this section alone,  take $M$ to be any closed manifold, $\nunaught$ a measure on $\diff^2(M)$ satisfying  \eqref{eq:IC2a}, and take $\munaught$ to be an ergodic, hyperbolic, $\nunaught$-stationary measure. $\muskew$ is as in Proposition \ref{prop:mudef}.
We note that if $\munaught$ has only positive exponents then, by the invariance principle in \cite{MR2651382},  $\munaught$ is $\nunaught$-\as invariant and $\mu= \nunaught^\Z\times \munaught$.  
 We thus also assume $\munaught$ has one negative exponent.  
If all exponents of $\munaught$ are negative, the analysis and conclusions in this section are still valid.  

Consider $\P$ a $\muskew$-measurable partition of $\Sigma\times M$ with the property that for $\mu$-\ae $(\xi,x)\in X$, there is an $r(\xi,x)$ with 
\begin{equation}\label{eq:partprop}
	\Sigmalocs ( \xi) \times \locstab [r(\xi,x)] x \xi  \subset \P(\xi,x)\subset \Sigmalocs(\xi)\times \stab x \xi.
\end{equation}
   Let $\{\muskew^\P_{(\xi,x)}\}$ denote an associated family of conditional measures.  
We consider here the degenerate situation where $\muskew^\P_{(\xi,x)}$ is supported on $\Sigmalocs (\xi)\times \{x\}$ for $\muskew$-\ae $(\xi,x)$.  Note  that hyperbolicity and recurrence imply that  for any other partition $\P'$  satisfying \eqref{eq:partprop} we have that 
$\mu^{\P'}_{(\xi,x)}$ is supported on $\Sigmalocs(\xi) \times\{x\}$ for $\mu$-\ae $(\xi,x)$ and that 
$$\mu^{\P}_{(\xi,x)} = \mu^{\P'}_{(\xi,x)}.$$
In particular, the hypothesis that $\muskew^\P_{(\xi,x)}$ is supported on $\Sigmalocs (\xi)\times \{x\}$ for $\muskew$-\ae $(\xi,x)$ implies that the partition $\P' $ given by $\P'= \{\Sigmalocs(\xi)\times \stab x \xi\}$ is measurable.

The purpose of this section is to prove the following proposition.
\begin{prop}\label{prop:entropyrigidity}
Assume for some partition $\P$ as above that the measures  $\muskew^\P_{(\xi,x)}$ are supported on $\Sigmalocs (\xi)\times \{x\}$ for $\muskew$-\ae $(\xi,x)$.  Then $\muskew = \nunaught^\Z\times \munaught$ and $\munaught$ is $\nunaught$-\as invariant.  
\end{prop}

The idea behind the proof of  Proposition \ref{prop:entropyrigidity} is that if, for $\P$ as in \eqref{eq:partprop}, the conditional measures $\muskew^\P_{(\xi,x)}$ are supported on $\Sigmalocs (\xi)\times \{x\}$ then the entropy of the skew product $F\colon (X, \mu)\to (X, \mu)$ has no fiber-wise entropy and thus the $\mu$-entropy of $F$ equals the entropy of the shift $\sigma \colon (\Sigma , \nunaught ^\Z) \to (\Sigma , \nunaught ^\Z)$.  As $F$ is  hyperbolic, the entropy of $F\colon (X, \mu)\to (X, \mu)$ should be captured by the mean conditional entropy $H_\mu(F\P\mid \P)$ for any (decreasing) partition $\P$ subordinated to the stable sets of $F$ in $X$ (a partition $\P$ as in \eqref{eq:partprop} will be such a partition under the assumptions on the support of $\muskew^\P_{(\xi,x)}$.)  
Let $\beta$ denote the partition on $X$ given by $\beta(\xi,x) = \Sigmalocs(\xi)\times \{x\}$.  Then $\beta $ is equivalent to $\P \bmod \mu$  and we have $H_\mu(F\beta \mid \beta) = h_{\nunaught^\Z}(\sigma)$.  Using Jensen's inequality in a manner analogous to the  proof of \cite[Theorem 3.4]{MR743818} (see also \cite[(6.1)]{MR819556} for the argument in English) one could  show that the conditional measures $\mu^\beta_{(\xi,x)}$ are canonically identified with $\nu^\N$ almost everywhere.  This would complete the proof.

However, the main technical obstruction in implementing the above procedure is that $h_{\nunaught^\Z}(\sigma)$ is not assumed to be finite.  Thus extra care is needed to approximate  differences of the form $\infty-\infty$ arising from the outline above.  


\subsection{Proof of Proposition \ref{prop:entropyrigidity}.}
Before presenting the proof of Proposition \ref{prop:entropyrigidity} we recall some facts about mean conditional entropy.  A primary  reference is \cite{0036-0279-22-5-R01}.  Let  $(X,\mu)$ be a Lebesgue probability space.  Given measurable partitions $\alpha, \beta$ of $(X,\mu)$ (which may be uncountable) we  define the \emph{mean conditional entropy} of $\alpha$ relative to $\beta$ to be $$H_\mu(\alpha \mid \beta) = - \int \log (\mu^\beta_x(\alpha(x))) \ d \mu(x)$$ where $\{\mu^\beta_x\}$ is a family of conditional measures relative to the partition $\beta. $ The \emph{entropy of $\alpha$} is  $H_\mu(\alpha) = H_\mu(\alpha \mid \{\emptyset, X\})$.  Note that if $H_\mu(\alpha)<\infty$ then $\alpha$ is necessarily countable.  
Given measurable partitions $\alpha, \beta,\gamma$ of $(X,\mu)$ we have 
\begin{enumerate}
\item $H_\mu(\alpha \vee \gamma\mid \beta) = H_\mu(\alpha \mid \beta) + H_\mu(\gamma\mid \alpha \vee \beta)$;
\item If $\alpha\ge \beta$ then $H_\mu(\alpha\mid \gamma) \ge H_\mu(\beta\mid \gamma)$ and $H_\mu(\gamma\mid \alpha)\le H_\mu(\gamma\mid \beta)$;
\item If $\gamma _n \nearrow \gamma$ and if $H_\mu(\alpha\mid \gamma_1)<\infty $ then $H_\mu(\alpha\mid \gamma_n) \searrow H_\mu(\alpha\mid \gamma)$.
\end{enumerate}

We  proceed with the proof of Proposition \ref{prop:entropyrigidity}.


\begin{proof}[Proof of Proposition \ref{prop:entropyrigidity}]
Let $\beta$ denote the partition on $X$ given by $\beta(\xi,x) = \Sigmalocs(\xi)\times \{x\}$.
As remarked above, the hypothesis that 
  $\muskew^\P_{(\xi,x)}$ is supported on $\Sigmalocs (\xi)\times \{x\}$ for $\muskew$-\ae $(\xi,x)$ for some partition $\P$ satisfying \eqref{eq:partprop} implies that all such partitions are equivalent modulo $\muskew$ and, furthermore, that any such partition $\P$ is equivalent to $\beta$ modulo $\muskew$.  
 
Given a measure $\lambda$ on $\Sigma\times M$ and a $\lambda$-measurable partition $\Q$ of $\Sigma\times M$ we write $\restrict{\lambda}{\Q}$ for the restriction of $\lambda$ to the sub-$\sigma$-algebra of $\Q$-saturated subsets and $\lambda^\Q_{(\xi,x)}$ for the conditional measure of $\lambda$ along the atom  $\Q(\xi,x)$.
As we explain below, the proposition follows if we can show the conditional measures $\mu^\beta_{(\xi,x)}$ take the form 
$$d\mu^\beta_{(\xi,x)}(\eta,y) = \delta_x(y) \ d  \nunaught^\N(\dots, \eta_{-3}, \eta_{-2}, \eta_{-1})  \ \delta _{\xi_{0}} (\eta_{0}) \  \delta _{\xi_{1}} (\eta_{1}) \delta _{\xi_{2}} (\eta_{2})\dots$$
To this end, define a measure $\lambda$ on $\Sigma\times M$ with $\restrict \lambda \beta = \restrict \muskew \beta$ and define $\lambda^\beta_{(\xi,x)}$ by 
$$d\lambda^\beta_{(\xi,x)}(\eta,y) = \delta_x(y) \ d  \nunaught^\N(\dots, \eta_{-3}, \eta_{-2}, \eta_{-1})  \ \delta _{\xi_{0}} (\eta_{0}) \  \delta _{\xi_{1}} (\eta_{1}) \delta _{\xi_{2}} (\eta_{2})\dots.$$
In what follows we show---under the hypothesis that $\mu^{\P}_{(\xi,x)}$ is supported on $\Sigmalocs(\xi) \times\{x\}$---that $\mu = \lambda$.  

Define the partition $\Q$ of $\Sigma\times M$ by $$\Q(\xi,x) = \Sigmalocs (\xi) \times M.$$
Observe for any $k\ge 0$ that
\begin{equation}\label{eq:key}
	F^k\beta = F^k\Q \vee \beta.
\end{equation}
Given a partition $\alpha$ of $\Sigma\times M$ we write $$\alpha^-:= \bigvee _{i = 0} ^{\infty} f^{-i}\alpha.$$
We need the following lemma whose proof 
we postpone until the next subsection.  
\begin{lemma}\label{lem:existsnicepart}
There exists a finite entropy partition $\alpha$ of $\Sigma\times M$ with $\alpha\le \beta$ and \begin{equation}\label{eq:alphaisgood}\alpha^-\vee \Q \circeq \beta.\end{equation}
\end{lemma}

Our strategy below will be to show that $\mu= \lambda$ by showing that 
	\begin{equation}\label{eq:17}
		\restrict{\mu^{\alpha^-}_{(\xi,x)}}{F^k(\beta)} = \restrict{\lambda^{\alpha^-}_{(\xi,x)}}{F^k(\beta)}\end{equation}
		for \ae $(\xi,x)$ and all $k\ge 0 $.  Note that the equality  $\restrict \lambda {\alpha^-} = \restrict \muskew {\alpha^-} $  and  the $k=0$ case
				$$\restrict{\mu^{\alpha^-}_{(\xi,x)}}{\beta} = \restrict{\lambda^{\alpha^-}_{(\xi,x)}}{\beta}$$
		follow from the construction of $\lambda$ and that $ \alpha^-\le \beta$.  Thus, as $F^k\beta$ generates the point partition for $k\ge 0$, showing \eqref {eq:17} for all $k\ge 1$ is sufficient to prove that $\mu= \lambda$.

As noted above, the maps  $F\colon (\Sigma\times M, \muskew)\to  (\Sigma\times M, \muskew)$ and $\sigma \colon (\Sigma, \nunaught^\Z) \to (\Sigma, \nunaught^\Z)$ may have infinite entropy.  Thus it is necessary in the below  argument to approximate $(\Sigma, \nunaught^\Z)$ by a finite entropy sub-system.  
Fix an increasing family of partitions $\A_n$, $n\in \N$,  of $(\diff^2(M), \nunaught)$ with the following properties:
\begin{enumerate}
	\item $\A_n$ contains $n$ elements;
	\item $\A_{n+1} \ge \A_n$;
	\item $\A_n$ increases to the point partition on $(\diff^2(M), \nunaught).$
\end{enumerate}
Let $\bar \A_n$ be the partition of $(\Sigma, \nunaught^\Z)$ defined by $\bar \A_n(\xi) = \{ \eta \mid \eta_0\in \A_n (\xi_0)\}$.  
Define the partition $\Q_n$ on $\Sigma \times M$ by $$\Q_n(\xi,x) = \{ (\eta,y)\mid \eta_k \in \A_n(\xi_k) \text{ for all } k\ge 0\}.$$
Continue to write $\pi\colon \Sigma\times M\to \Sigma$.  
Then $\Q_n= (\pi\inv \bar \A_n)^-$.  
We have $$h_{\nunaught^\Z} (\sigma,\bar \A_n) =h_\mu(F, \pi\inv \bar  \A_n)= H_\muskew (F\Q_n\mid \Q_n) \le \log(n).$$

Given $i\le j\in \Z$ and $n\in \N$ define a (finite) partition $\calR^{[i,j]}_n$ of $\Sigma\times M$ by 
	$$\calR_n^{[i,j]}(\xi,x):= \{(\eta, y): \eta_\ell\in \A_n(\xi_\ell) \text{ for all } i\le \ell \le j\}.$$
We have  $\calR_n^{[-k,m]}\nearrow F^k( \Q_n)\nearrow F^k(\Q)$, respectively, as $m\to \infty$ and $n\to \infty$.  

For fixed $(\xi,x)$ and $k\ge 0$, consider the sequence 
\begin{equation}\label{eq:888}\dfrac {\lambda^{\alpha^-}_{(\xi,x)}(\calR^{[-k, m]}_m(\eta,y))}
	{\mu^{\alpha^-}_{(\xi,x)}(\calR^{[-k, m]}_m(\eta,y))}\end{equation}
as $(\eta,y)$ varies over $\alpha^-(\xi,x)$.  
For fixed $k$, this forms a non-negative supermartingale (on $\left(\alpha^-(\xi,x), \mu^{\alpha^-}_{(\xi,x)}\right)$, indexed by $m$) and hence converges pointwise. 

From 
 \eqref{eq:key}, \eqref{eq:alphaisgood} and the fact that $\Q\le F^k\Q$ we have 
\begin{equation}\label{eq:lampshade} \alpha^-\vee F^k\Q =  \alpha^-\vee F^k\Q \vee \Q= F^k\beta.\end{equation}
As the $\sigma$-algebras generated by $\calR^{[-k, m]}_m$ increase to the algebra generated by $F^k\Q$ as $m\to \infty$, 
by a theorem of Anderson and Jesson (\cite {MR0028376}, see also \cite{MR898154, MR501787} for statements) the pointwise limit of \eqref{eq:888} is the Radon--Nikodym derivative
$$\lim_{m\to \infty} \dfrac {\lambda^{\alpha^-}_{(\xi,x)}(\calR^{[-k, m]}_m(\eta,y))}
	{\mu^{\alpha^-}_{(\xi,x)}(\calR^{[-k, m]}_m(\eta,y))}= \dfrac {d \restrict {\lambda^{\alpha^-}_{(\xi,x)}}{F^k\beta}}{d \restrict {\mu^{\alpha^-}_{(\xi,x)}}{F^k\beta}}(\eta,y)
.$$

Note that $\calR_n^{[0,m]}\le \beta$ for all $m\ge 0 $  and $n\ge 1$ and hence   $$\lambda _{(\xi,x)}^{\alpha-}(\calR_n^{[0,m]}(\xi,x)) = \mu^{\alpha-} _{(\xi,x)}(\calR_n^{[0,m]}(\xi,x))$$ for any $m\ge 0$.  For $(\eta,y) \in \alpha^-(\xi,x)\cap \calR^{[0,m]}_n(\xi,x)$ we have
$$
 \dfrac {\lambda^{{\alpha^- \vee \calR^{[0, m]}_n}}_{(\xi,x)}(\calR^{[-k, m]}_n(\eta,y))}
	{\mu^{{\alpha^- \vee \calR^{[0, m]}_n}}_{(\xi,x)}(\calR^{[-k, m]}_n(\eta,y))}=
	 \dfrac {\lambda^{\alpha^-}_{(\xi,x)}(\calR^{[-k, m]}_n(\eta,y))}
	{\mu^{\alpha^-}_{(\xi,x)}(\calR^{[-k, m]}_n(\eta,y))}
\cdot
 \dfrac
	{\mu^{\alpha^-}_{(\xi,x)}(\calR^{[0, m]}_n(\xi,x))}
	 {\lambda^{\alpha^-}_{(\xi,x)}(\calR^{[0, m]}_n(\xi,x))}.
$$
Thus $$ \dfrac {\lambda^{{\alpha^- \vee \calR^{[0, m]}_n}}_{(\xi,x)}(\calR^{[-k, m]}_n(\eta,y))}
	{\mu^{{\alpha^- \vee \calR^{[0, m]}_n}}_{(\xi,x)}(\calR^{[-k, m]}_n(\eta,y))}=
	 \dfrac {\lambda^{\alpha^-}_{(\xi,x)}(\calR^{[-k, m]}_n(\eta,y))}
	{\mu^{\alpha^-}_{(\xi,x)}(\calR^{[-k, m]}_n(\eta,y))}.$$
	 
For every $k,n,m$ and $(\xi,x)$ we have
$$\displaystyle\int_{\left(\alpha^- \vee \calR^{[0, m]}_n\right)(\xi,x)} \dfrac {\lambda^{{\alpha^- \vee \calR^{[0, m]}_n}}_{(\xi,x)}(\calR^{[-k, m]}_n(\eta,y))}
	{\mu^{{\alpha^- \vee \calR^{[0, m]}_n}}_{(\xi,x)}(\calR^{[-k, m]}_n(\eta,y))}
	 \ d \mu_{(\xi,x)}^{\alpha^- \vee \calR^{[0, m]}_n}(\eta,y) \le 1.$$
Consider the expressions
$$I_1(n,m) = \int \int \log  \left(\lambda^{\alpha^-\vee \calR^{[0, m]}_n}_{(\xi,x)}(\calR^{[-k, m]}_n(\eta,y))\right)\ d \mu_{(\xi,x)}^{\alpha^- \vee \calR^{[0, m]}_n}(\eta,y)\ d \muskew(\xi,x)$$
and 
$$ I_2(n,m) =\int \int \log \left(\mu^{\alpha^-\vee \calR^{[0, m]}_n}_{(\xi,x)}(\calR^{[-k, m]}_n(\eta,y))\right) \ d \mu_{(\xi,x)}^{\alpha^- \vee \calR^{[0, m]}_n}(\eta,y) \ d \muskew(\xi,x).$$
From the above inequality and Jensen's inequality, for every $k,n$ and $m$ we have 
that $I_1(n,m) - I_2(n,m)\le 0$.  
From the explicit form of $\lambda^\beta_{(\xi,x)}$, for $(\eta,y)\in \alpha^-(\xi,x) \vee \calR^{[0, m]}_n (\xi,x)$ we have for $k\ge 1$
	\begin{align*}
	{\lambda^{\alpha^-\vee \calR^{[0, m]}_n}_{(\xi,x)}(\calR^{[-k, m]}_n(\eta,y))}
	& =   \prod_{i = 1}^k \nunaught \left(\A_n(\eta_{-i})\right) \\&=
\mu^{\Q_n}_{(\eta,y)} \left(F^k\Q_n(\eta,y)\right)
\end{align*}
whence \begin{align*}I_1(n,m)&= \int \left(\log \mu^{\Q_n}_{(\eta,y)} \left(F^k\Q_n(\eta,y)\right) \right)\ d \mu (\eta,y) 
= - H(F^k\Q_n\mid \Q_n) \\ & =- h_\mu( F^k, \pi\inv(\bar \A_n)).\end{align*}

On the other hand, we have 
\begin{align*}
I_2(n,m) &= \int \int \log \left ( \mu^{\alpha^-\vee \calR^{[0, m]}_n}_{(\xi,x)}(\calR^{[-k, m]}_n(\eta,y))\right) \ d \mu_{(\xi,x)}^{\alpha^- \vee \calR^{[0, m]}_n}(\eta,y) \ d \muskew(\xi,x)\\
&= \int \int \log\left(\mu^{\alpha^-\vee \calR^{[0, m]}_n}_{(\eta,y)}(\calR^{[-k, m]}_n(\eta,y))\right) \ d \mu_{(\xi,x)}^{\alpha^- \vee \calR^{[0, m]}_n}(\eta,y) \ d \muskew(\xi,x)\\
&= \int  \log\left(\mu^{\alpha^-\vee \calR^{[0, m]}_n}_{(\eta,y)}(\calR^{[-k, m]}_n(\eta,y))\right) \ d \muskew(\eta,y)\\
&= - H_\mu \left( \calR^{[-k, m]}_n\mid \alpha^-\vee \calR^{[0, m]}_n\right).
\end{align*}

Recall the facts about mean conditional entropy collected above.  We have the formula
\begin{align}
	H_\mu &\left( \calR^{[-k, m]}_n\vee F^k \alpha ^-\mid \alpha^-\vee \calR^{[0, m]}_n\right)\notag
		= H_\mu \left( \calR^{[-k, m]}_n\mid \alpha^-\vee \calR^{[0, m]}_n\right)  \\
	   & \quad \quad +  H_\mu \left( F^k \alpha ^-\mid \alpha^-\vee \calR^{[0, m]}_n\vee  \calR^{[-k, m]}_n\right).\label{eq:thisguyisfinite}
	\end{align}
As $H_\mu(\alpha)<\infty$ we have $H_\mu \left( F^k \alpha ^-\mid \alpha^-\right)<\infty$.  In particular, as $\calR^{[-k, m]}_n$ and $\calR^{[0, m]}_n$ are finite partitions,  both terms on the right hand side of \eqref{eq:thisguyisfinite} are finite.  

By \eqref{eq:lampshade} and the fact that that $\Q\le F^k\Q$, we have 
$H_\mu \left( F^k \alpha ^-\mid \alpha^-\vee F^k\Q\right) = 0$ .   
As 
\begin{align*}
H_\mu \left( F^k \alpha ^-\mid \alpha^-\vee  \calR^{[-k, m]}_n\right)&\underset{m\to \infty}{\searrow }
H_\mu \left( F^k \alpha ^-\mid \alpha^-\vee F^k\Q_n \right)\\&\underset{n\to \infty}{\searrow }
H_\mu \left( F^k \alpha ^-\mid \alpha^-\vee F^k\Q\right) \end{align*}given $\epsilon>0$ we may select $m_0$ so that 
$$H_\mu \left( F^k \alpha ^-\mid \alpha^-\vee  \calR^{[-k, m_0]}_{m_0}\right)<\epsilon.$$
Furthermore for any $n>0$
\begin{align*}
	H_\mu \left( \calR^{[-k, m]}_{n}\vee F^k \alpha ^-\mid \alpha^-\vee \calR^{[0, m]}_{n}\right) 
	&=H_\mu \left( \calR^{[-k,-1]}_{n}\vee F^k \alpha ^-\mid \alpha^-\vee \calR^{[0, m]}_{n}\right) \\
	&\underset{m\to \infty}{\searrow }
	H_\mu \left( \calR^{[-k,-1]}_{n}\vee F^k \alpha ^-\mid \alpha^-\vee \Q_{n}\right) \\
	&=	H_\mu \left( F^k\Q_{n}\vee F^k \alpha ^-\mid \alpha^-\vee \Q_{n}\right). \\
\end{align*}
But, for any $n$ 
$$H_\mu \left( F^k\Q_{n}\vee F^k \alpha ^-\mid \alpha^-\vee \Q_{n}\right) = h_\mu \left(F^k, \pi\inv (\bar \A_{n})\vee \alpha\right) \ge H_\mu \left( F^k\Q_{n}\mid \Q_{n}\right) $$
Thus for $m_0$ above we have
\begin{align*}
I_1(m_0,m_0) &- I_2(m_0,m_0) = - H_\mu \left( F^k\Q_{m_0}\mid \Q_{m_0}\right)  + 	H_\mu \left( \calR^{[-k, m_0]}_{m_0}\mid \alpha^-\vee \calR^{[0, m_0]}_{m_0}\right) 
\\
&=- H_\mu \left( F^k\Q_{m_0}\mid \Q_{m_0}\right)  + 	H_\mu \left( \calR^{[-k, m]}_n\vee F^k \alpha ^-\mid \alpha^-\vee \calR^{[0, m]}_n\right) \\
& \quad \quad - H_\mu \left( F^k \alpha ^-\mid \alpha^-\vee  \calR^{[-k, m_0]}_{m_0}\right)\\
&\ge- H_\mu \left( F^k\Q_{m_0}\mid \Q_{m_0}\right)  + H_\mu \left( F^k\Q_{m_0}\vee F^k \alpha ^-\mid \alpha^-\vee \Q_{{m_0}}\right)\\
& \quad \quad - H_\mu \left( F^k \alpha ^-\mid \alpha^-\vee  \calR^{[-k, m_0]}_{m_0}\right)\\
&\ge 
- H_\mu \left( F^k \alpha ^-\mid \alpha^-\vee  \calR^{[-k, m_0]}_{m_0}\right)\\
&\ge -\epsilon.
\end{align*}

It follows that  

\begin{align*}
	&\int \int \log \left(\dfrac {\lambda^{\alpha^-}_{(\xi,x)}(\calR^{[-k, m]}_m(\eta,y))}	
	{\mu^{\alpha^-}_{(\xi,x)}(\calR^{[-k, m]}_m(\eta,y))}
	\right)  \ d\mu^{\alpha^-}_{(\xi,x)}  (\eta, y)\ d  \mu(\xi,x)\\
	&= 
	\int \int \log \left(\dfrac {\lambda^{{\alpha^-\vee\calR_n^{[0,m]}}}_{(\xi,x)}(\calR^{[-k, m]}_m(\eta,y))}	
	{\mu^{{\alpha^-\vee\calR_n^{[0,m]}}}_{(\xi,x)}(\calR^{[-k, m]}_m(\eta,y))}
	\right)  \ d\mu^{\alpha^-\vee\calR_n^{[0,m]}}_{(\xi,x)}  (\eta, y)\ d  \mu(\xi,x)\\
&= I_1(m,m) - I_2(m,m)
\end{align*}
approaches $0$ as $m\to \infty.$

We have the following elementary claim.
\begin{claim}\label{claim:sillier}
Let $f_n$ be a sequence of positive, $\mu$-integrable functions. Assume $\int f_n \ d \mu \le 1$ for every $n$ and that $\int \log f_n \ d \mu \to 0$ as $n\to \infty$.  Then $f_n$ converges to $1$ in measure.
\end{claim}
\begin{proof}
Given $\delta>0$, there is a $c_\delta>0$ such that for all $x\in (0,\infty)$ with $|x- 1|>\delta$ 
we have $\log x \le x-1 - c_\delta$. Then for every $n$,
	\begin{align*}\int \log f_n  \ d \mu 
					& \le \int  f_n   \ d \mu - 1 - \mu \left(\{x: |f_n(x) - 1| > \delta \}\right)c_\delta\\
					& \le - \mu \left(\{x: |f_n(x) - 1| > \delta \}\right)c_\delta.
	\end{align*}
As $\int \log f_n \ d \mu  \to 0$ we have $\mu \left(\{x: |f_n(x) - 1| > \delta \}\right)\to 0$ as $n \to \infty$.  
\end{proof}
As $\dfrac {\lambda^{\alpha^-}_{(\xi,x)}(\calR^{[-k, m]}_m(\eta,y))}{\mu^{\alpha^-}_{(\xi,x)}(\calR^{[-k, m]}_m(\eta,y))}\to \dfrac {d \restrict {\lambda^{\alpha^-}_{(\xi,x)}}{F^k\beta}}{d \restrict {\mu^{\alpha^-}_{(\xi,x)}}{F^k\beta}}(\eta,y)$ it follows from Claim \ref{claim:sillier} that 
$$\dfrac {d \restrict {\lambda^{\alpha^-}_{(\xi,x)}}{F^k\beta}}{d \restrict {\mu^{\alpha^-}_{(\xi,x)}}{F^k\beta}}(\eta,y)= 1$$ for $\mu$-\ae $(\xi,x)$ and $\mu^{\alpha^-}_{(\xi,x)}$-\ae $(\eta,y)$. 
Taking $k\to \infty$ it follows that $\lambda = \mu$.  

 Now consider an atom of $\Q(\xi,x) $.  We have the canonical product representation 
 $\Q(\xi,x) = \Sigmalocs(\xi)\times M$.  Let $\bar \mu^\Q_{(\xi,x)}$ denote the projection of $\mu^\Q_{(\xi,x)}$ on $\Sigmalocs(\xi) \times M$ onto $M$.  Using that $\mu = \lambda$, in these coordinates we have for $\eta\in \Sigmalocs(\xi)$ and $y\in M$ that 
 	$$d\mu^\Q_{(\xi,x)} (\eta,y) = \ d \nunaught (\eta_{-1})\ d \nunaught (\eta_{-2})\dots  d \bar \mu^\Q_{(\xi,x)}(y)$$
Then we have the natural identification $\mu_\eta= \mu_\xi = \bar \mu^\Q_{(\xi,x)}$ for $\nunaught^\N$-\ae $\eta\in \Sigmalocs(\xi)$.  
In particular, the function $\xi\mapsto \mu_\xi$ is a.s.-constant on almost every local stable set.    As  $\xi\mapsto \mu_\xi$ is a.s.-constant on almost every local  unstable set in $\Sigma$, an argument similar to Proposition \ref{prop:hopf}   shows that $\xi\mapsto \mu_\xi$ is  \as constant  on $\Sigma$.  
\end{proof}

\def\diam{\mathrm{diam}}

\subsection{Proof of Lemma \ref{lem:existsnicepart}.}
We remark  that we continue to assume $M$ to be a compact, $d$-dimensional manifold.  For $\nunaught$ a measure on $\diff^2(M)$ satisfying \eqref{eq:IC2a}, we take $\munaught$ to be an ergodic, $\nunaught$-stationary measure.  We further assume that $\munaught$ is hyperbolic.  
Take $\kappa>0$ so that $\munaught$ has no exponents in the interval $[-\kappa, \kappa]$.

\subsubsection{One-sided Lyapunov charts and stable manifolds as Lipschitz graphs}
\label{sec:stabmanifold}
 Let $k $ be the almost-surely constant value of $\dim E^s_\omega(x)$.  
  Given $v\in \R^d= \R^k\times \R^{d-k}$ decompose $v = v_1+ v_2$ and write $|v|_i = |v_i|$ and $|v| = \max\{|v|_i\}$.
We will write $d_{\R^d}(\cdot,\cdot)$ for the induced metric on $\R^d$ and $d(\cdot,\cdot)$ for the metric on $M$.  
We use the notation $\R^d(r)$ to denote the ball of radius $r$ centered at $0$.  
To emphasize the one-sidedness of our constructions we work on $\Sigma_+\times M$.   Recall the associated skew product $\hat F\colon  \Sigma_+\times M \to \Sigma_+\times M$ and  the corresponding  $\hat F$-invariant measure $\nunaught^\N\times \munaught$ on $\Sigma_+\times M$.  
 
 \def\wtd{\widetilde}

As outlined in \cite[(4.1)]{MR968818}, 
for every sufficiently small $\epsilon>0$, there is a measurable function $l\colon \Sigma_+ \times M\to [1,\infty)$  and a full  measure set $\Lambda \subset \Sigma_+ \times M$ such that  
\begin{enumerate}
\item for $(\omega,x)\in \Lambda$ and every $n\in \N$, there exists a diffeomorphism $\phi_n$ defined on a small neighborhood of $\cocycle[\omega][n](x)$ whose range is $ \R^d(\ell(\omega,x)\inv e^{-n\epsilon})$ with 
\begin{enumerate}
	\item $\phi_0(\omega,x) (x) = 0$;
	\item $D\phi_0 (\omega,x)E^s_\omega (x) = \R^k \times \{0\}$;
	\item $D\phi_0 (\omega,x)\left (E^s_\omega (x)\right)^\perp = \{0\} \times \R^{d-k}$;
\end{enumerate}
\item for $n\ge 1$, writing $\td f_n(\omega, x)= \phi_{n+1} (\omega, x) \circ f_{\sigma^{n-1}(\omega)} \circ \phi_n (\omega, x) \inv$ where defined, for all $n\ge 0$ we have 
\begin{enumerate}
	\item $\td f_n (\omega, x) (0) = 0$;
	\item $D_0  \td f_n (\omega, x)  =  \left(\begin{array}{cc}A_n & 0 \\0 & B_n \end{array}\right)$
	where $A_n\in \Gl(k, \R) $, $B_n\in \Gl(d-k, \R) $  and
		 $ |A_nv|\le e ^{-\kappa+ \epsilon} |v|, v\in \R^k$, 
	$e ^{\kappa- \epsilon} |v|\le |B_nv|,  v\in \R^{d-k}$;
	\item $\lip(\td f_n(\omega, x) - D_0  \td f_n (\omega, x))<\epsilon$\end{enumerate} 
where $\lip(\cdot)$ denotes the Lipschitz constant of a map on its domain;
%
%
%
%
\item $\ell(\omega, x)\inv  e^{-n\epsilon}  \le \lip(\phi_n(\omega,x))\le \ell(\omega, x) e^{n\epsilon} $.
\end{enumerate}

Note that the domain of $\phi_n(\omega,x)$ contains a ball of radius $\ell(\omega, x)^{-2} e^{-2n \epsilon}$ centered at $\cocycle [\omega][n](x)$ in $M$.  
We remark that while the Lipschitz constant of $\td f_n$, norm of $B_n$, and conorm of $A_n$ need not be bounded, the hyperbolicity of $D_0\td f_n$ and the Lipschitz closeness of $\td f_n $ to $D_0\td f_n$ is uniform in $n$.  

Relative to the  charts $\phi_n (\omega,x)$, one may apply the Perron--Irwin method of constructing stable manifolds through each point of the orbit $\{\cocycle [\omega] [n] (x), n\ge 0\}$.   See the proof of Theorem 3.1 in \cite{MR1369243} or the similar   proof of \cite[Theorem V.4.2]{MR2542186}.  Choosing $\epsilon>0$ above sufficiently small, the outcome is the following.
\begin{proposition}
For $(\omega,x)\in \Lambda$ and every $n\ge 0$ there is a Lipschitz function $$h_n (\omega, x) \colon \R^k \left(\ell(\omega,x)\inv e^{-n\epsilon}\right) \to \R^{d-k}$$ with 
\begin{enumerate}
	\item $h_n(\omega, x)(0) = 0 $;
	\item $\Lip (h_n(\omega, x)) \le 1$;
	\item $\td f_n (\graph(h_n(\omega,x))) \subset  \graph (h_{n+1}(\omega,x))$ and if $y,z\in \graph(h_n(\omega,x))$ then 
		$$|\td f_n (\omega,x) (y) - \td f_n (\omega,x)(z) | \le\left( e^{-\kappa + \epsilon } + \epsilon \right)|y-z|;$$
\end{enumerate}
\end{proposition}
Note that we have that $\graph(h_n(\omega,x))$ is contained in the domain of $\td f_n$.  
We have that $\phi_n\inv (\graph(h_n(\omega,x))$ is an open subset of $W^s_{\sigma^n(\omega)}(\cocycle [\omega][n]( x))$.  

\subsubsection{Divergence from the stable manifold in local charts}
We have the following claim.  
\begin{claim}\label{claim:nonsense}
Fix $(\omega, x)\in \Lambda$ and suppose $y \in \R^d(\ell(\omega,x)\inv e^{-n\epsilon})$ is in the domain of $\td f_n(\omega, x)$.  Write $y= (u,v)$ and $f(y) = (u', v')$.  Then 
$$|v' - h_{n+1}(\omega, x) (u') |_2  \ge \left(e^{\kappa-\epsilon} - 2\epsilon\right) 
|v - h_{n}(\omega, x) (u) |_2.$$
\end{claim}
\begin{proof}
Write  $z = (u, h_n(\omega,x)(u))$ and  $(\hat u, \hat v) = \td f_n(\omega, x)(z)$.  Then 
$$|v' - h_{n+1}(\omega, x) (u') |_2\ge |v' - \hat v |_2 - |\hat v - h_{n+1}(\omega, x) (u') |_2.$$
As $$\td f_n (\omega ,x )( y) -\td f_n (\omega ,x )( z)  = D_0 \td f_n (\omega ,x )( y - z)  + w$$
where $|w| \le \epsilon |y-z|$ we have  
\begin{enumerate}
\item $|u' - \hat u|_1 = |\td f_n (\omega ,x )( y) -\td f_n (\omega ,x )( z) |_1 \le \epsilon |y-z|$; 
\item $|v'- \hat v|_2 = |\td f_n (\omega ,x )( y) -\td f_n (\omega ,x )( z) |_2= |\td f_n (\omega ,x )( y )- \td f_n (\omega ,x )(z) | \ge e^{\kappa - \epsilon} |y-z| - \epsilon |y-z|$.  
\end{enumerate}
As $\Lip(h_n(\omega, x))\le 1$ and as $\td f_n(\omega,x)(z) \in \graph (h_{n+1}(\omega,x))$ we have
$$| \hat v- h_{n+1}(\omega, x)(u')|_2 = | h_{n+1}(\omega,x)(\hat u)- h_{n+1}(\omega, x)(u')|_2 \le |u'- \hat u|_1\le \epsilon |y - z|.$$ 
As  $|y-z| = |v - h_{n}(\omega, x) (u) |_2$, the claim follows. 
\end{proof}
Note that having taken  $\epsilon>0$ sufficiently small we can arrange that $e^{\kappa -  \epsilon}- 2 \epsilon \ge e^{\kappa - 3 \epsilon}$. 

We write $\wtd {W}^s_m(\omega,x):=\graph(h_m(\omega,x)) $ for the remainder.  Note that 
$\wtd {W}^s_m(\omega,x)$ is the path-connected component of $$\phi_m(\omega,x)(W^s_{\sigma^m(\omega)}(\cocycle [\omega ][m] (x)))$$ in $\R^d (\ell(\omega, x)\inv e^{-m\epsilon})$ containing $0$.

\subsubsection{Radius function and related estimates}
Fix $K_0\subset \Sigma_+\times M$ with positive measure on which the function $\ell(\omega, x)$ is bounded above by some $\ell>10$.  
Fix $m_0\in \N$ so that $$\chi:= \left(e^{-m_0(\kappa - 4\epsilon)}\right)2 \ell ^2<1.$$

For $(\omega,x)\in K_0$ define $n(\omega,x)$ to be the $m_0$th return of $(\omega, x)$ to $K_0$.   
We define $\rho\colon \Sigma\times M \to (0,\infty)$ as 
$$
\rho(\omega,x) = \begin{cases}
\dfrac 1 {4} \ell^{-4 } e^{-2\epsilon n(\omega,x)  }  \left( \displaystyle \prod _{k= 0 }^{n(\omega,x)-1}  \left (|f_{\sigma^k (\omega)}|_{C^1}\right)\inv \right)  
& (\omega,x) \in K_0,\\
\ell\inv & (\omega,x)\notin K_0.
\end{cases}
$$

Consider $(\omega, x)\in K_0$ and $y\in M$ with $d(x,y)< \rho(\omega, x)$.  Let $n = n(\omega, x)$ and for $ 0\le j\le n$ write $x_j = \cocycle [\omega][j](x)$ and $y_j = \cocycle [\omega][j](y)$.   
For all $0\le j\le n$ we have $d(x_j, y_j)  
\le \frac 1 4 \ell^{-4 } e^{-2\epsilon n  }  $ hence $y_j$ is in the domain of $\phi_j(\omega,x)$; it follows that   for $0\le j\le n-1$ we have that 
$\phi_j(\omega,x)(y_j) $ is in the domain of $\td f_j(\omega, x)$.  
We claim 
\begin{equation}\label{eq:mcfly} d_{\R^d} \left(\phi_0(\omega, x)(y), \wtd W^s_0(\omega, x)\right)\le 2 e^{-n( \kappa - 3\epsilon)} 
d_{\R^d} \left(\phi_n(\omega, x)(y_n), \wtd W^s_n(\omega, x)\right)\end{equation}
Indeed write $(u_j, v_j) = \phi_j(\omega, x)(y_j)$.  By Claim \ref{claim:nonsense} and the fact that 
$\wtd W^s_n(\omega, x)$ is a graph of the  1-Lipschitz function 
we have 
\begin{align*}
2d_{\R^d} \left(\phi_n(\omega, x)(y_n), \wtd  W^s_n(\omega, x)\right) 
&\ge  |v_n - h_n(\omega, x) (u_n)|_2\\
&\ge e^{n(\kappa-3\epsilon)}   |v_0 - h_0(\omega, x) (u_0)|_2\\
&\ge  e^{n(\kappa-3\epsilon)} d_{\R^d} \left(\phi_0(\omega, x)(y), \wtd W^s_0(\omega, x)\right).
\end{align*}

We now consider the transition between the charts  $\phi_n(\omega,x)$ and $\phi_0(\hat F^n(\omega, x))$.  Recall $n = n(\omega,x)$ and write $\hat x = x_n, \hat y = y_n$ and $\hat \omega = \sigma^n(\omega)$.  Recall that $(\hat \omega, \hat x) \in K_0$.  
As $d(\hat x, \hat y) \le \frac 1 4 \ell^{-4} e^{-2\epsilon n}$  we have that $\hat y $ is in the domain of $\phi_0(\hat \omega, \hat x)$.  
Furthermore, as $|\phi_0(\hat \omega
  , \hat x)(\hat y)| \le \ell^{-3} \le {.01} {\ell}\inv $, we can find  $z\in \stabM {{\hat x}} {{\hat \omega}}$ 
 such that $$d_{\R^d}\left(\phi_0(\hat \omega, \hat x)(\hat y), \wtd W^s_0(\hat \omega, \hat x)\right) = 
d_{\R^d}\left(\phi_0(\hat \omega, \hat x)(\hat y),  \phi_0(\hat \omega, \hat x)(z) \right).$$
Let $ \phi_0(\hat \omega, \hat x)(z)   = (u,v) = (u, h_0(\hat \omega, \hat x)(u))$.  
As $h_0(\hat \omega, \hat x) $ has  Lipschitz constant less than $1$, for $t\in [0,1]$ we have 
$$|(tu, h_0(\hat \omega, \hat x)(tu))| = |(tu, h_0(\hat \omega, \hat x)(tu))|_1 \le |u|_1  = |(u, v)|.$$ 
Then for any $0\le t\le 1$,  writing $$z(t) = \phi_0(\hat \omega, \hat x) \inv \big(tu, h_0(\hat \omega, \hat x)(u)\big)$$ we have
\begin{align*}
d(\hat x,z(t)) &\le \ell  |\phi_0(\hat \omega, \hat x)(z)|\\
&\le \ell\big(d_{\R^d}\left(0, \phi_0(\hat \omega, \hat x)(\hat y)\right)+ d_{\R^d}\left(\phi_0(\hat \omega, \hat x)(\hat y),  \phi_0(\hat \omega, \hat x)(z)\right)\big)\\
& \le \ell 2 |\phi_0(\hat \omega, \hat x)(\hat y)|\\
&\le 2 \ell^2  d(\hat x, \hat y)\\
&\le \frac 1 2 \ell^{-2} e^{-2\epsilon n}.
\end{align*}
Thus $z(t)$ is in the domain of $\phi_n(\omega, x)$ for all $0\le t\le 1$ whence 
$\phi_n(\omega, x)(z(t)) \in \wtd W^s_n(\omega, x)$ for all $0\le t \le 1$.  
It follows that 
\begin{align*}
d_{\R^d} \left(\phi_n (\omega, x)(\hat y), \wtd W^s_n(\omega, x)\right)
&\le  d_{\R^d} \big(\phi_n (\omega, x)(\hat y), \phi_n (\omega, x)(z)\big) \\
& \le \ell e^{n\epsilon}
d\left(\hat y, z \right)\\ 
& \le \ell^2 e^{n\epsilon}
d_{\R^d}\left(\phi_0(\hat \omega, \hat x)(\hat y), \wtd W^s_0(\hat \omega, \hat x)\right). 
\end{align*}

Combining the above with \eqref{eq:mcfly}  we have 
\begin{align}
d_{\R^d}& \left(\phi_0(\omega, x)(y), \wtd W^s_0(\omega, x)\right) \notag  \\
&\le 2 e^{-n( \kappa - 3\epsilon)} 
\ell^2 e^{n\epsilon}
d_{\R^d}\left(\phi_0(\hat \omega, \hat x)(\hat y), \wtd W^s_0(\hat \omega, \hat x)\right) \notag \\
&\le \chi d_{\R^d}\left(\phi_0(\hat \omega, \hat x)(\hat y), \wtd W^s_0(\hat \omega, \hat x)\right).\label{eq:90}\end{align}

Now let $n_j $ denote the $(jm_0)$th return of $(\omega, x)$ to $K_0$.  Suppose for some $k$ that 
$d\left(\cocycle[\omega][n_j] (x), \cocycle[\omega][n_j] (y)\right)\le \rho(\hat F^{n_j}(\omega, x))$ for all $0\le j \le k$.
By induction on \eqref{eq:90} we have that 
$$d_{\R^d} \left(\phi_0(\omega, x)(y), \wtd W^s_0(\omega, x)\right) \le \chi^k
d_{\R^d}\left(\phi_0(\hat F^{n_k}( \omega, x))( \cocycle [\omega][n_k](y)), \wtd W^s_0(\hat F^{n_k}( \omega, x))\right).$$ 
This establishes  the following claim.  
\begin{claim}\label{claim:thisguywashardashelltoprove}
Let $(\omega,x)\in K_0$ and let $y\in M$ be such that $d(\cocycle[\omega][n] (x), \cocycle[\omega][n] (y))\le \rho(\hat F^n(\omega, x))$ for all $n\ge 0$.  Then 
$y \in \stabM x \omega$.  
\end{claim}

\subsubsection{Construction of the partition $\alpha$.}

Recall the  integrability hypothesis  \eqref{eq:IC2a}.  As $\omega\mapsto \log^+ |f_\omega|_{C^2}$ is integrable, it follows that $$\int |\log\rho(\omega, x))| \ d (\nunaught ^\Z \times \munaught)(\omega, x) < \infty.$$ 
We adapt \cite[Lemma 2]{MR627789}  to our $\rho$ to produce a finite entropy partition $\hat \alpha$ of $\Sigma_+\times M$ such that $\diam(\hat \alpha(\omega,x)\cap M_\omega)\le \rho(\omega,x)$ for almost every $(\omega,x)$.   
The only modification needed in the proof of \cite[Lemma 2]{MR627789} is  to replace, for each $r$, the family $\mathscr P_r$ at the top of page 97 with the partition $\bar {\mathscr P}_r = \{ \Sigma \times P\mid P\in \mathscr P_r\}$.  

Take $\alpha$ to be the preimage  of $\hat \alpha$ under the natural projection $\bar \pi_+\colon \Sigma\times M  \to \Sigma _+ \times M$. Then clearly $\alpha \le \beta$.   Furthermore,  if $(\eta,y) \in \Q \vee \alpha ^-(\xi,x)$ then \begin{enumerate}
\item there is an  $\omega\in \Sigma _+$ with $\bar \pi_+ (\eta, y ) = (\omega, y)$ and $\bar \pi_+ (\xi,x) = (\omega, x)$ and 
\item $\hat F^n (\omega, y) \in \hat \alpha (\hat F^n(\omega, x))$ for all $n\ge 0 $.  
\end{enumerate}
If $(\omega ,x)\in K_0$ then, by Claim \ref{claim:thisguywashardashelltoprove}, $y\in \stabM x \omega$.   
If $(\omega ,x)\notin K_0$ then take $n$ so that $\hat F^n(\omega, x)\in K_0$.  Then 
$\cocycle [\omega ][n] (y) \in  \stabM {{\cocycle [\omega ][n] (x)}} {{\sigma ^n(\omega)}}$ whence $y\in \stabM x \omega$.
This completes the proof of Lemma \ref{lem:existsnicepart}.

\section{Proof of Theorem \ref{thm:atomicfibertoinvariant}}\label{sec:atom} 

We continue to work in the case $\Omega = \Sigma$.  
As remarked earlier, the $\F$-measurability of $(\xi,x)\mapsto E^s_\xi(s)$ holds trivially if all exponents of $\munaught$ are negative and the $\nunaught$-\as invariance of $\munaught$ follows from the invariance principle of \cite{MR2651382} if all exponents of $\munaught$ are positive.  We thus assume $\munaught$ has two exponents, one of each sign $\lambda^s<0<\lambda^u$.  Moreover, assume  that the map  $(\xi, x) \mapsto E^s_\xi(x)$ is not $\F$-measurable.  

As above, let $\P$ be a measurable partition of $\Sigma\times M $ satisfying \eqref{eq:partprop}.  
 We show that if $\muskew^\P_{(\xi,x)}$ is not supported on a set of the form $\Sigmalocs(\xi)\times \{x\}$ then the measures $\mu_\xi$ are non-atomic. 
From this contradiction and Proposition \ref{prop:entropyrigidity}, the finiteness and $\nunaught$-\as invariance of $\munaught$ follows.  
 The non-atomicity of the measures $\mu_\xi$ is established, under the above hypotheses, through a procedure similar to the proof of Proposition \ref{lem:main}.

We introduce one piece of new notation in the specific case $\Omega = \Sigma$.  
\begin{definition}
	Given $\xi = (\dots, \xi_{-1}, \xi_{0}, \xi_{1}, \xi_{2}, \dots)$ and $\eta = (\dots,  \eta_{-1}, \eta_{0}, \eta_{1}, \eta_{2}, \dots)$ in $\Sigma$ define 
		$$[ \xi, \eta]:= (\dots, \xi_{-2}, \xi_{-1}, \eta_{0}, \eta_{1}, \eta_{2}, \dots).$$
\end{definition}

\newcommand\Sigmalocuu{\Sigma^{+}_{\text {loc},-1}}
Recall that in Section \ref{sec:modifyF} we replaced the $\sigma$-algebra of $\Sigmalocu$-saturated sets with  its preimage under $\sigma$. 
 Let $$\Sigmalocuu(\xi) = \sigma\inv (\Sigmalocu(\sigma(\xi)))= \{ \eta\in \Sigma : \eta_j= \xi_j \text{ for all $j\le 0$} \}.$$ 
Then  $\hat \Fol$ as modified in Section \ref{sec:modifyF}  is the sub-$\sigma$-algebra of $\Sigmalocuu$-saturated sets.  

The proof of Theorem \ref{thm:mainAtominv} is a simplified version of the proof of the Theorem \ref{thm:skewproductABS}  except our initial points $p$ and $q$ remain fixed and, as $p$ and $q$ are in the same total stable space, we use only positive times. 
In particular,  the open set $U$, the choice of $M_\delta, m_\delta$, and the estimates in Section \ref{sec:times} are not used here.

\begin{proof}[Proof of Theorem \ref{thm:mainAtominv}]  

We assume in the setting of  Theorem \ref{thm:atomicfibertoinvariant} that the map  $(\xi, x) \mapsto E^s_\xi(x)$ is not $\F$-measurable.  

We recall all constructions and notations from Sections  \ref{sec:lemMainPrep} and \ref{sec:lemMainProof} in the   case that  $\Omega = \Sigma$. 
In particular we retain the notation $Y = (\R \times \Sigma \times M )/\sim$ equipped with the measurable parametrization, $\Phi^t$ the suspension flow, $\ess_p(t)$ and $\Psi^s$ the time change and corresponding flow, $\mususp$ and $\mualt$ the $\Phi^t$- and  $\Psi^s$-invariant measures,  and $\tau_{p, \delta, \epsilon},$ and $L_{p, \delta, \epsilon}$ the stopping times.  

  Recall the choice of   $\kappa_1, \kappa_2$ in Section  \ref{sec:choices} and take $\alpha =\left( \dfrac {\kappa_1}{5(\kappa_1 +\kappa_2)}\right)$. %
Recall the choices of various  parameters 
$$M_0, \hat M , \gamma_1, \gamma_2, r_0, r_1, \hat r, C_1, C_2, C_3, D_0,D_1, L_1, a_0, \hat L, \hat T, T_0$$  in Section  \ref{sec:choices}  as well as the sets $K_0, S_0, S_{\hat M}, \scrA, \rec (T_0)$ and the $\sigma$-algebras $\calS$, $\calS^m$.  
In this section, the constants $r_0, r_1,  C_1, C_3, D_1$ are chosen so that Lemma \ref{ninetynine} holds and we take $\rec (T_0) \subset K$ where $K$ is defined below.  

We assume for the sake of contradiction that the  measures $\muskew_\xi$ are finitely supported $\nunaught^\Z$-\as but that $\munaught$ is  not $\nunaught$- \as invariant.   
 By ergodicity, each 
$\muskew_{\xi}$ is supported on a finite set $F(\xi)\subset M$ with the same cardinality \as   We fix a compact $\Lambda'''\subset \Sigma \times M$ such that $\muskew_{\xi}$ has an atom at $(\xi,x)$ for every $(\xi,x)\in \Lambda'''$ and 
$$\min\{d(x, y)\mid (\xi,x)\in \Lambda''', y \in F(\xi)\sm \{x\}\}$$
is bounded below by some $\epsilon_1>0$.  

By choosing the above parameters so that the associated sets have sufficiently large measures, we can take the  compact set   
$$K =  K_0 \cap \Lambda'' \cap \big([0,1)\times \Lambda' \big)\cap \big([0,1)\times \Omega'\times M\big)\cap ([0,1)\times\Lambda'''\big),$$
where $K_0$,  $\Lambda'$, $\Lambda''$, $\Omega'$ are as in Section  \ref{sec:choices}, to be such that  $$\mususp(K)>1 - \frac \alpha {10} \text{ and } \mualt(K)>1 - \frac { \alpha}{20N_0} .$$
We have the same estimates as in Claim \ref{claim:silly} (with $U = \emptyset$.)

As we assume the measure $\munaught$ is  not $\nunaught$-\as invariant, by Proposition \ref{prop:entropyrigidity}, it follows that the measures $\{\muskew_{\xi}^{\P}\}$ are not supported on sets of the form $\Sigmalocs(\xi)\times \{x\}$ where $\P$ is a partition of $\Sigma\times M$ 
satisfying \eqref{eq:partprop}.  
Recall the set $\Rec(T_0)\subset K$ in \ref{19}.  We may find  $p = (\pt,\xi,x)$ and $q = (\pt,\zeta, y)$ in $Y$ with \begin{itemize}
\item $p\in  \Rec(T_0)$, $q\in  \Rec(T_0)$;
\item  $\zeta\in \Sigmalocs( \xi)$;
\item $y\in \locstabM [r_1] x \xi \sm \{x\}$.  
\end{itemize}
Fix $\delta = \|H^s _p(y)\|>0.$  We may  assume  $\delta<\epsilon_1/(2C_2C_3M_0^6).$ 

As in Claim \ref{prop:densities} we have the following.  Note that unlike in Claim \ref{prop:densities}, $\ell_j>0.$

\begin{claim}\label{lem:goodtimes}
The exists a sequence $\{\ell_j\}$ with $\ell_j\to \infty$  such that  
	\begin{enumlemma}
		\item $\Phi^{\ell_j}(p)\in K \cap S_0 \cap \scrA$;
		\item {$\Phi^{\ell_j}(q)\in K\cap S_0 $;}  
		\item $\Phi^{L_{p,\delta, \delta}( \ell_j)}(p)\in K\cap S_{\hat M}$;
		\item $\Phi^{L_{p,\delta, \delta} (\ell_j)}(q)\in K\cap S_{\hat M}$.
	\end{enumlemma}
\end{claim}
\begin{proof}
Let $F_k$ be as in Claim \ref{prop:densities} (with $q=q_j$).   
Then, as in the proof of  Claim \ref{prop:densities}, for our fixed  $T_0$ and any $T>T_0$ with $L_{p,\delta,\delta}(T)>T_0$ we have 
$$\leb\left( [0, T] \cap F_1\cap F_2\right)\ge (1 - 5\alpha)T$$
and, as $L_{p, \delta, \delta}(0)=\tau_{p, \delta, \delta}(0) = 0$, 
 $$\leb \left(L_{p, \delta, \delta}\left([0, T]\right)\sm  (F_3\cap  F_4)\right)
		\le  (4\alpha) \leb\left(L_{p, \delta, \delta}\left([0, T]\right)\right)\le 4\alpha \kappa_2 T$$
whence, $$\Leb \left([0, T] \sm  L_{p, \delta, \delta}\inv\left( F_3\cap F_4\right)\right)\le 4\alpha \kappa_1\inv  \kappa_2 T.$$ 
Then
$$\leb\left([0,T]\cap  F_1\cap F_2\cap  L_{p,\delta, \delta}\inv(F_3) \cap L_{p,\delta, \delta}\inv (F_4) \right)> (1 -5\alpha - 4\alpha \kappa_1\inv  \kappa_2) T.$$
The choice of $\alpha$ guarantees $(1 - 5\alpha - 4\alpha \kappa_1\inv  \kappa_2) T\to \infty $ as $T\to \infty$.   
\end{proof}

Let $\{\ell_j\}$ be a sequence of times satisfying Claim \ref{lem:goodtimes}. As in Section \ref{sec:lemMainProof}, for  each $j$ write 
$\hat p_j=(\hat \pt_j, \hat \xi_j,\hat  x_j )= \Phi^{\ell_j}(p)$, $\hat q_j =(\hat \pt_j, \hat \zeta_j,\hat  y_j )=  \Phi^{\ell_j}(q)$, $\td p_j=(\td \pt_j, \td \xi_j,\td  x_j )= \Phi^{L_{p,\delta,\delta}(\ell_j)}(p)$, $\td q_j =(\td \pt_j, \td \zeta_j,\td  y_j )=   \Phi^{L_{p,\delta,\delta}(\ell_j)}(p)$, $s'_j = \ess _{\hat p_j}(\tau_{p,\delta, \delta}(\ell_j))$, $s''_j = \ess _{\hat q_j}(\tau_{p,\delta, \delta}(\ell_j)).$

Note that $\tau_{p, \delta, \delta}(\ell_j) \to \infty$ as $\ell_j\to \infty$.  Then for $\ell_j$ large enough,
 $$ s_j'',s_j'\ge (\lambda^u-\eps)\tau_{p, \delta, \delta}(\ell_j)\ge \hat M$$ and, since $\td p_j, \td q_j \in S_{\hat M}$, 
it follows that $$\Exp_\mualt(1_{K} |\Sal^{ s_j'})\left(\td p_j \right) >.9, \quad \quad \Exp_\mualt(1_{K} |\Sal^{ s_j''})\left(\td q_j \right) >.9.$$
As in Section \ref{sec:lemMainProof} we have 
		$ \hat p_j, \hat q_j\in K$,
		$\mucondS_{\hat p_j}(K)>.9$,
	 	$\nunaught^\N 
		( A_{\gamma_2}(\hat p_j))>.9$,
	 	$\mucondS_{\hat p_j}( \Psi^{-s'_j}(K))>.9$,
		$\mucondS_{\hat q_j}(K)>.9$, and
	 	$\mucondS_{\hat q_j}( \Phi^{-s''_j}(K))>.9$.

The measures $\mucondS_{\hat p_j}$ and $\mucondS_{\hat q_j}$ are, respectively,  canonically identified with $\nunaught ^\N$ (the $\hat \Fol$-conditional measure) on $\Sigmalocuu(\hat \xi_j)$ and $\Sigmalocuu(\hat \zeta_j).$ 
Furthermore, the natural identification 
	$$\Sigmalocuu(\hat \xi_j) \to \Sigmalocuu(\hat \zeta_j), \quad \quad \eta \mapsto \eta'= [\hat  \zeta_j,\eta]$$ preserves the measure $\nunaught^\N$.  
Thus the set of $\hat \eta_j\in \Sigmalocuu(\hat \xi_j) $ such that 
\begin{enumerate}
	\item $\hat \eta_j \in A_{\gamma_2}(\hat p_j)$,
	\item $\bar p_j:= (\hat \pt_j, \hat \eta_j, \hat x_j) \in K \cap  \Psi^{-s'_j}(K)$, 
	\item  $\bar q_j:= (\hat \pt_j, \eta'_j, \hat y_j) \in K \cap  \Psi^{-s''_j}(K)$
\end{enumerate}
where $\eta'_j= [\hat  \zeta_j,\hat \eta_j]$, has $\nunaught^\N$-measure at least $1/2$.  For each $j$, fix such a pair $\hat \eta_j$ and $\eta'_j= [\hat  \zeta_j, \hat \eta_j]$.

As before, write  $ t_j' = \ess_{\bar p_j }\inv(s_j')$,  $t''_j =\ess_{\bar q_j}\inv(s_j'')$ and define
%
$p'_j =  (\pt_j', \xi_j', x_j') := \Psi^{s'_j}(\bar p_j) =\Phi^{t'_j}(\bar p_j) \in K$,
$q''_j= (\pt_j'', \zeta_j'', y_j'') := \Psi^{s''_j}(\bar q_j)= \Phi^{t''_j}(\bar q_j)\in K$, and 
  	$q'_j = (\pt_j', \zeta_j', y_j'):=\Phi^{t'_j}(\bar q_j) $. 
For $\ell_j$ sufficiently large we have $d(\hat x_j, \hat y_j)<r_0$.  For such $\ell_j$, as $\bar p_j,\bar q_j\in K$  let 
$${\hat v _j  
= \locstabM[r_1] {\hat y_j}{\eta'_j}\cap \locunstM[r_1] {\hat x_j}{\hat \xi_j}.}$$
Since $\hat y_j \in \locstabM[r_1]{\hat x_j}{\hat \xi_j}$, by   \ref{itme3'} of Lemma \ref{ninetynine}  
we have $$\frac 1{C_3 } \|H^s_{\hat p_j}(\hat y_j) \|\le \|H^u_{\hat p_j}(\hat v_j) \|\le C_3 \|H^s_{\hat p_j}(\hat y_j) \|.$$
Exactly as in Claim \ref{claim:annulus} we have 
$$\frac 1 {C_3M_0^6} \delta \le  \| H^u_{p_j'} (v_j') \|  \le C_3M_0^6\delta$$
where $
(\pt'_j, \xi'_j, v'_j) = \Phi^{t_j'}(\hat \pt_j,  \hat \eta_j, \hat v_j).$
(We take $\hat z_j = \hat y_j$ in the proof).
Hence
$$\frac 1 {C_2C_3M_0^6} \delta \le  d(x_j', v_j')  \le C_2C_3M_0^6\delta$$

As in Lemma \ref{lem:key}\ref{lemitem:0} we have 
$q'_j = \Phi^{\hat t_j} (q''_j)$ for some $|\hat t_j| \le \hat T$.
To adapt the proof to the current  setting, we replace the estimate  \eqref{eq:thisguy} with
\begin{align}
&\dfrac {\| \restrict { D\Phi^{t''_j}}{E^u(\bar q_j)}\|} 
	{\| \restrict { D\Phi^{t''_j}}{E^u(\bar p_j)}\|}
		= 
	\dfrac {\| \restrict { D\cocycle[ \eta'_j][n'']}{T_{\hat y_j} \unstM   {\hat y_j} { \eta'_j}}\|}
		{\| \restrict { D\cocycle[\hat \eta_j][n'']}{T_{\hat x_j} \unstM   {\hat x_j} {\hat \eta_j}}\|}
\notag\\
		&\quad \quad
			= 
	\dfrac {\| \restrict { D\cocycle[ \eta_j'][n'']}{T_{\hat y_j} \unstM   {\hat y_j} { \eta_j'}}\|}
		{\| \restrict { D\cocycle[\hat \eta_j][n'']}{T_{\hat v_j} \unstM   {\hat x_j} {\hat \eta_j}}\|}
			\cdot
	\dfrac {\| \restrict { D\cocycle[ \xi'_j][-n']}{T_{ x_j'} \unstM   { x'_j} { \xi'_j}}\|}
		{\| \restrict { D\cocycle[ \xi'_j][-n']}{T_{ v_j'} \unstM   { x'_j} { \xi'_j}}\|}
		\cdot 
	\dfrac{\| \restrict { D\cocycle[ \xi'_j][-(n'- n'')]}{T_{ v_j'} \unstM   { x'_j} { \xi'_j}}\|}
	 {\| \restrict { D\cocycle[ \xi'_j][-(n'-n'')]}{T_{ x_j'} \unstM   { x'_j} { \xi'_j}}\|}			
\label{eq:bananashark}
\end{align}
and similarly modify   \eqref{eq:thisguy2}.  
 Note that the bound on the first term of  \eqref{eq:bananashark} now follows from Lemma \ref{ninetynine}\ref{item8:2'} as $\pi_+(\eta'_j) = \pi_+(\hat \eta_j)$.

Consider  an accumulation point $p_0 = (\pt_0, \xi_0, x_0)$ of $\{p_j'\}$ and $B\subset \N$ such that $\lim_{j\in B\to \infty} p_j' = p_0$.  Then the measure $\mususp_{(\pt_0, \xi_0)}$ has an atom at $p_0$.   

 Note that, 
 as $\xi'_{j_k}\to \xi_0$ for some subsequence $\{j_k\}$, we have $\zeta'_{j_k}\to \xi_0$.  Indeed  for any $j$ and  any $n\in \N$ with $n\le \ell_j$ we have that $\hat \xi_j$ and $\hat \zeta_j$, and  hence $\hat \eta_j$ and $\eta'_j$, agree in the $k$th index for all $-n \le k \le \infty$.  As as $s'_j>0$, $\xi_j'$ and $\zeta_j'$  agree in the $k$th index for all $-n \le k \le \infty$. 

Thus, as in the proof of Lemma \ref{lem:here2},  passing to subsequences   of $B$ there are  accumulation points $q_0 = (\pt_0, \xi_0, y_0)$ of $\{q'_j\}$ and $q_1 = (\pt_1, \xi_1, y_1)\in K$ of $\{q''_j\}$ and a $\hat t \in [-\hat T, \hat T]$  
such that $\Phi^{\hat t}( q_1)  = q_0$.  
 As $v_j \in \locstabM [r_1] {\hat y_j} {\eta'_j}  $, $\bar q_j\in K$, and $t'_j\to \infty$ and as $\cocycle[\eta'_j][n] =\cocycle[\hat \eta_j][n] $ for $n\ge 0$, 
 we have $d(v'_j, y'_j) \to 0$ hence  $d(x_0,y_0)\ge \frac 1 {C_2C_3M_0^6} \delta$.  
 Since $q_1\in K \subset [0,1)\times \Lambda'''$, the measure $\mususp_{(\pt_1, \xi_1)}$ has an atom at  $q_1$.   By the invariance of $\mususp$, it follows that  $\mususp_{(\pt_0, \xi_0)}$ has an atom at $q_0$.  On the other hand,  $x_0 \neq y_0$ yet $ d(x_0, y_0)\le C_2C_3M_0^6\delta<\epsilon_1$.  As $p_0\in K\subset[0,1)\times \Lambda'''$, this   contradicts  the choice of $\epsilon_1$.  
\end{proof}
\begin{remark}  
In the above proof, we have that $q_0\in \locunstp [r_1] {p_0}$.  Thus one can modify the above proof to conclude that the skew-product $F\colon (X,\muskew) \to (X,\muskew)$ has positive fiber-wise entropy.  In this way, one can show that for any hyperbolic, $\nunaught$-stationary measure $\munaught$ such that 
\begin{enumerate}
\item $E^s_\omega(x)$ is not non-random, and 
\item $\munaught$ is not $\nunaught$-\as invariant
\end{enumerate}
that the $\munaught$ entropy $h_\munaught(\MP^+(M,\nunaught))$ is positive.  Under the positive entropy hypothesis, the authors showed in an earlier version of this paper that $\munaught$ must then be SRB.  However, one still need to perform the more detailed analysis in Section \ref{sec:lemMainProof} to rule out the existence of a $\nunaught$-\as invariant, hyperbolic measure $\munaught$ with zero entropy and such that $E^s_\omega(x)$ is not non-random to derive the full result in Theorem \ref{thm:1+}.  
\end{remark}

\section{Proofs of remaining theorems}\label{sec:left}
\subsection{Proof of Theorem \ref{thm:3}}
\begin{proof}
Let $\munaught$ be as in Theorem \ref{thm:3}, and assume $\munaught$ is not finitely supported and that  the stable distribution $E^s_\omega(x)$ is non-random.  It follows from Theorem \ref{thm:1+} that $\munaught$ is SRB.  
\def\B{\mathcal B}
Let $F\colon \Sigma \times M\to \Sigma \times M$ be the canonical skew product  constructed in Section \ref{sec:skewRDS} and let $\mu$ be the $F$-invariant measure defined by Proposition \ref{prop:mudef}. Then   
the conditional measures of $\mu$ along almost every unstable manifold $\unst \xi x$ for the skew product $F$ are absolutely continuous.  
Define the \emph{ergodic basin}  $B\subset \Sigma\times M$  of $\mu$ to be the set of $(\xi,x) \in \Sigma\times M$ such that 
 $$ \lim_{n\to \infty}\frac 1 N \sum _{n = 0}^{N-1} \phi(\cocycle  [\xi] [n](x))= \int \phi \ d \munaught$$
 for all $\phi\colon M \to \R$ continuous. 
By the pointwise ergodic theorem and the separability of $C^0(M)$, we have $\mu(B) =1.$
Furthermore, for points $(\xi, x)\in B$  whose fiber-wise stable manifold $\stab x \xi $ is defined 
we have $$\stab x \xi \subset B .$$

We have  the following ``transverse'' absolute continuity property.   
Given a typical $\xi\in \Sigma$ and a certain continuous families of 
fiber-wise local stable manifolds $\mathcal S:= \{\locstabM x \xi \}_{x\in Q}$, 
consider two manifolds $T_1$ and $T_2$ everywhere uniformly transverse to the collection $\mathcal S$.  Define the holonomy map from $T_1$ to $T_2$ by ``sliding along'' elements of $\mathcal S$.  Such holonomy maps 
were shown by Pesin to be absolutely continuous in the deterministic volume-preserving setting \cite{MR0458490}.  
For fiber-wise stable manifolds associated to skew products satisfying \eqref{eq:IC2}, such holonomy maps are also known to be absolutely continuous.  See \cite[(4.2)]{MR968818} or \cite[III.5]{MR1369243} for further details and references   to proofs.

The above absolute continuity property implies that  if $\munaught$ is SRB (whence $\mu$ is fiber-wise SRB) and if $A\subset\Sigma\times M$ is any set with $\mu(A)>0$ then for a positive measure subset of $\xi$, 
$$\bigcup_{(\xi,x) \in A\cap M_\xi}\stabM x \xi \subset M_\xi$$
has positive Lebesgue measure in $M_\xi$.   
It follows that for the ergodic basin $B$, $$(\nunaught^\Z\times m)( B)>0.$$
We note that if $\eta\in \Sigmalocs(\xi)$ then (under the natural identification of subsets of $M_\eta$ and $M_\xi$) 
$$B\cap M_\eta = B\cap M_\xi$$
since $\cocycle[\xi][n]= \cocycle[\eta][n] $ for $n\ge 0$.  
Define $\hat B$ to be the ergodic basin of $\nu^\N\times \munaught$ for the skew product $\hat F\colon \Sigma_+\times M$; that is $(\omega, x)\in \hat B$ if 
$$ \lim_{n\to \infty}\frac 1 N \sum _{n = 0}^{N-1} \phi(\cocycle  [\omega] [n](x))= \int \phi \ d \munaught$$
 for all continuous $\phi\colon M \to \R$ continuous. 
 We have that $\hat B$ is the image of $B$ under the natural projection $\Sigma\times M \to \Sigma _+ \times M$ whence 
 $ {(\nunaught^\N\times m)(\hat  B)} >0$.

Define a measure $$\hat m= \tfrac 1 {(\nunaught^\N\times m)(\hat  B)} \restrict {(\nunaught^\N\times m)}{ \hat B}$$ on $\Sigma_+\times M$.  Since both the set $\hat  B$ and the measure $\nunaught^\N\times m$ are $\hat F$-invariant (recall that $m$ is $\nunaught$-\as invariant) the measure $\hat m$ is $\hat F$-invariant.  
  Furthermore, for $\hat m$-\ae $(\omega, x)$ and any continuous $\phi\colon M \to \R,$ the Birkhoff sums satisfy $$\lim_{n\to \infty}\frac 1 N \sum _{n = 0}^{N-1} \phi(\cocycle  [\omega] [n] (x)) = \int \phi \ d \munaught$$ which  implies that $\hat m$ is ergodic for $F$ and, in particular, is an ergodic component of $\nunaught^\Z\times m$.  This implies (see e.g.\ \cite[Proposition I.2.1]{MR884892}) that $\hat m $ is of the form $\hat m= \nunaught^\Z\times m_0$ for $m_0$ an ergodic component of $m$ for $\MP^+(M, \nu)$.  

Then, for any continuous function $\phi\colon M\to \R$,  $\nunaught^\N$-\ae $\omega\in \Sigma_+$, and $m_0$-a.e. $x\in M$, we have 
$$\lim_{n\to \infty}\frac 1 N \sum _{n = 0}^{N-1} \phi(\cocycle [\omega]  [n] (x) )= \int \phi \ d \munaught.$$
Furthermore, since  $\nunaught \times m_0$ is invariant and ergodic  for $\hat F$, for $\nunaught^\N$-\ae $\omega\in \Sigma_+$  and $m_0$-a.e. $x\in M$   we also have that 
$$\lim_{n\to \infty}\frac 1 N \sum _{n = 0}^{N-1} \phi(\cocycle [\omega]  [n] (x)) = \int \phi \ d m_0.$$
In particular, $ \int \phi \ d \munaught =  \int \phi \ d m_0$ for all $\phi\colon M\to \R$, whence $\munaught= m_0$.  
\end{proof}

\subsection{Proof of Theorem \ref{thm:4}}\label{sec134}
Let $M$ be a compact surface and let $\mu$ be a non-atomic Borel probability on $M$.  Let $f \in \diff_\mu^2(M)$ as in Theorem \ref{thm:4}.
In particular,  $f$ is ergodic, hyperbolic, and, as $\mu$ has no atoms, $f$ has one positive and one negative  exponent which we denote by $\lambda^s_f <0<\lambda^u_f$.  
%
%

\subsubsection{Preliminary constructions and observation}
\def\M{\mathcal{M}}
 Let $K\subset \diff^2_\mu(M)$ be a fixed compact  subset with $f\in K$.  Moreover, assume that $K$ is symmetric in that if $g\in K$ then $g\inv \in K$.  For   this section  set $$\Sigma := \Sigma_K = K^\Z.$$  
Let $\sigma\colon \Sigma\to \Sigma$ be the left shift, $F\colon \Sigma\times M\to \Sigma\times M$   the canonical invertible skew products, and $DF\colon\Sigma\times TM\to \Sigma\times TM$  the fiber-wise derivative.  
With $X= \Sigma \times M$, we observe that $F$ and $DF$ are continuous transformations of $X$ and $TX$.  
In what follows, we will study the fiber-wise exponents of the cocycle $DF$ as the measures on $\Sigma $ changes.  We rely on  tools developed in  the study of continuity properties of Lyapunov exponents appearing in many  sources including \cite{1507.08978, MR1981401, Bocker-Neto:2010aa, viana2014lectures}.

Write $\M(K)$ for the space of all Borel probability measures on $K$. 
Given $\nu \in \M(K)$, equip $\Sigma$ with the shift-invariant measure $\nu^\Z$.  
For any $\nu\in \M(K) $, we have that 
$\mu$ is $\nu$-stationary.  Moreover, as $\mu$ is preserved by every element of $K$, the measure $\nu^\Z\times \mu$ is $F$-invariant and coincides with the measure given by Proposition \ref{prop:mudef}.  
We will say that $\mu$ is ergodic for $\nu$ if it is ergodic as a $\nu$-stationary measure.   



We make some  preliminary  observations.  

\begin{claim}
Let $\nu \in \M(K)$ with  $\nu (f) >0$.  Then $\mu$ is ergodic for  $\nu$.  
\end{claim}
\begin{proof}
Suppose $\mu = \mu_1 + \mu_2$ where $\mu_i$ are nontrivial, $\nu$-stationary, mutually singular measures.  
Then 
	$$\mu_1 = \int_{g\neq f} g_\ast\mu_1 \ d\nu(g)+ \nu(f) f_*\mu_1.$$
	By the $f$-ergodicity of $\mu$, $f_*\mu_1$ is not mutually singular with respect to $\mu_2$.  This contradicts that $\mu_1$ is mutually singular with respect to $\mu_2$.
\end{proof}


For $\nu \in \M(K)$, we recall the definition of Lyapunov exponents guaranteed by Proposition \ref{prop:OMT} for the stationary measure $\mu$.  We recall that the exponent is $\nu^\N$-\as independent of choice of word.  We write $$\lambda^1_{\nu}(x)\le \lambda^2_\nu(x)$$ for the Lyapunov exponents  of $\mu$ for words defined by $\nu$ at the point $x$ with the convention that if $\mu$ has only one exponent at $x$ we declare $\lambda^1_{\nu}(x)= \lambda^2_\nu(x)$.  
Given $\nu$ and $x$ with $\lambda^1_{\nu}(x)\neq  \lambda^2_\nu(x)$ and $\xi \in \Sigma = K^\Z$ write  \begin{equation}T_x M = E^1_\xi(x)\oplus E^2_\xi(x)\label{eq:OscSplitting}\end{equation} for the associated Lyapunov splitting. 
If $\lambda^1_{\nu}(x)=  \lambda^2_\nu(x)$ then we write   $E^1_\xi(x) =  E^2_\xi(x)= T_xM$.  

Consider an involution on $\M(K)$ defined as follows.  For $g\in K$ define 
define $\theta(g):= g\inv$.  
For $\nu \in \M(K)$, $\theta_*\nu$ is the measure 
\begin{equation}\theta_*\nu (A) : = \nu(\theta (A))\label{eq:invol}\end{equation}
for $A\subset K$.  We have that $\theta_*\colon  \M(K)\to  \M(K)$ is involutive.  

\begin{lemma}
For  $\nu \in \M(K)$ and $\mu$-\ae $x$  we have 
	$$\lambda^1_{\nu}(x)  = -\lambda^2_{\theta (\nu)} (x).$$
\end{lemma}
\begin{proof}
On $\Sigma:= K^\Z$, define  the involution
$\Psi\colon \Sigma \to \Sigma$ given by $$\Psi\colon (\dots, g_{-2},  g_{-1}.  g_{0},  g_{1}, \dots )\mapsto 
(\dots, g_{1}\inv,  g_{0}\inv.  g_{-1}\inv,  g_{2}\inv, \dots ). $$
We have  $\Psi_*(\nu^\Z) = (\theta_*\nu)^\Z$.  

Consider a $\mu$-generic $x$ and $\nu^\Z$-generic $\xi$.  Then 
$$\lim_{n\to \infty} \dfrac 1 n \log \|\restrict{Df^n_{\Psi(\xi)}}{E^1_\xi(x)} \|= 
\lim_{n\to \infty} \dfrac 1 n \log \|\restrict{Df^{-n}_{\xi}} {E^1_\xi(x)} \|= - \lambda^1_\xi(x).$$

Similarly 
$$\lim_{n\to \infty} \dfrac 1 n \log \|\restrict{Df^n_{\Psi(\xi)}}{E^2_\xi(x)}\| =  - \lambda^2_\xi(x).$$
Since $\Psi$ takes $\nu^\Z$-generic words to $(\theta_*\nu)^\Z$-generic words, this completes the proof.  
\end{proof}


Consider  $\nu \in \M(K)$.  We remark that if the fiber-wise exponents  were both positive or both negative, the measure $\mu$ would necessarily by atomic. Hence
for $\nu \in \M(K)$ we have $\lambda^1_{\nu}(x)\le 0\le \lambda_{\nu}^2(x)$.  
%
%


\subsubsection{Invariant measures for the projectivized cocycle}
With $X:= \Sigma \times M$ and $TX = \Sigma \times TM$, let $\Pp  TX$ denote the projectivized tangent bundle $$\Pp  TX: = \Sigma \times \Pp  TM.$$
Given   $(\xi,x,v)$ in $TX$ with  $v\neq 0$,  write $(\xi,x,[v])$  for the class in $\Pp TX$.  
We write $\Pp DF\colon \Pp TX\to \Pp TX$ to denote the action induced by $DF$ on $\Pp TX$.  

For a fixed $\nu\in \M(K)$ let $\eta$ be a $\Pp DF$-invariant, Borel probability measure on $\Pp TX$ which projects to $\nu^\Z\times \mu$ under the natural projection $\Pp TX\to X$.  
Given such an $\eta$ write $\{\eta_{\xi}\}$ for the family of conditional measures induced by the projection $\Pp TX \to \Sigma$. 

Given $\zeta,\xi\in \Sigma$ we have a natural identification of $\{\xi \}\times \Pp T M$ and $\{\zeta \}\times \Pp T M$ and   we view   $\xi\mapsto \eta_{\xi}$ as a measurable map from $X$ to the space of measures on $ \Pp TM$. 

Recall we have two natural partitions of $\Sigma$: the partition into local stable and unstable sets $\{\Sigmalocs\}$ and $\{\Sigmalocu\}$.  The conditional measures on $\Sigma$ induced by either partition is naturally identified with $\nu^\N$.  
Recall (c.f. Section \ref{sec:skewreint} and ignoring the modification in Section \ref{sec:modifyF}) that  we write $\hat \Fol$ for the $\sigma$-algebra of local unstable sets on $\Sigma$.  
We will  say $\eta$ is a \emph{$u$-measure} if 
$\xi  \mapsto \eta_{\xi }$ is $\hat \Fol$-measurable.  
Alternatively $\eta$ is a {$u$-measure} if 
for   $\nu^\Z$-\ae $\xi\in \Sigma$  
 and $\nu^\N$-\ae $\zeta \in \Sigmalocu(\xi)$, we have $$ \eta_{\xi} = \eta_{\zeta}.$$ We similarly define \emph{$s$-measures} and define $\eta$ to be an \emph{$su$-measure} if it is simultaneously an $s$- and $u$-measure.
We remark that $u$-measures correspond to $\nu$-stationary measures on $\Pp TM $ projecting to $\mu$.  (See \cite[Chapter 5]{viana2014lectures} for more details.)

If $\eta$ is an $su$-measure then  the level sets of the map $\zeta\mapsto \eta_{\zeta}$ are essentially saturated by local stable and  local unstable sets.  Since the measure $\nu^\Z$ has product structure, if $\eta$ is an $su$-measure the assignment  $\zeta \mapsto \eta_{\zeta}$ is $\nu^\Z$-\as constant. 
 In particular, if $\nu$ is an $su$-measure, there is a measure $\eta_0$ on $\Pp TM$ projecting to $\mu$ on $M$ with $\eta = \nu^\Z\times \eta_0$.  If $\eta$ is assumed $\Pp DF$-invariant it follows that $\eta_0$ is $\nu$-\as invariant. 

\begin{claim}\label{claim:smeasuresclosed}
Let $\nu_j\in \M(K)$ converge to $\nu$ in the weak-$*$ topology.
Let $\eta_j$ be a sequence of $\Pp  DF$-invariant $u$-measures, projecting to  $(\nu_j)^\Z\times \mu$.  
Then the set of weak-$*$ accumulation points of  $\{\eta_j \}$ is non-empty and consists of $\Pp DF$-invariant $u$-measures projecting to $\nu^\Z\times \mu$.  
\end{claim}
\noindent The above claim follows for instance from   \cite[Proposition 5.18]{viana2014lectures}.

We note that there always exist $\Pp DF$-invariant $s$- and $u$-measures.  
However, the existence of $\Pp DF$-invariant $su$-measures is unexpected, absent the existence of  a  $\nu$-\as invariant subbundle $V\subset TM$.  
However, there is a dynamical situation where \emph{every} $\Pp DF$-invariant measure is an $su$-measure.  

\begin{prop}[{\cite{MR850070, MR2651382}}]\label{prop:ledrappier}
Suppose $\nu\in \M(K)$ is such that $\lambda_\nu^1(x)  = 0 = \lambda_\nu^2(x)$ for $\mu$-\ae $x$.  Then any $\Pp DF$-invariant measure $\eta$ for the projectivized cocycle $\Pp D F\colon \Pp TX \to  \Pp TX$ projecting to $\nu^\Z \times \mu$ is an $su$-measure.
\end{prop}

In what follows, we will primarily focus on measures $\nu$ such that $\mu$ is ergodic and has two distinct Lyapunov exponents $\lambda^1_\nu< \lambda^2_\nu$ for $DF$.  
In this case we have two canonical measures $\eta^1_{\nu}$ and $\eta^2_{\nu}$ given by 
\begin{equation} \label{eq:msrs}d \eta^j_{\nu}(\xi,x, [v]): = \ d \delta_{E^j_{\xi}(x) } ([v]) \ d \mu(x) \ d \nu^\Z(\xi)\end{equation}
where  ${E^j_{\xi}(x) } $ is the associated subspace of the Lyapunov splitting  \eqref{eq:OscSplitting}.  
By the $DF$-invariance of the distributions ${E^j_{\xi}(x) } $, we have that the measures $\eta^j_{\nu}$ are $\Pp  DF$-invariant.  Furthermore, it follows from Proposition \ref{prop:measurability} that $\eta^2_{\nu}$ is a $u$-measure and $\eta^1_{\nu}$ is an $s$-measure.

In the above setting, the measures defined by \eqref{eq:msrs} are the only ergodic $\Pp DF$-invariant measures  on $\Pp TX$ projecting to $\nu^\Z\times \mu$. Indeed,
\begin{claim}\label{lem:decomp}
Let $\nu\in \M(K)$ be   such that  $\mu$ is ergodic for $\nu$ and has  two distinct Lyapunov exponents.  
  Then any $\Pp DF$-invariant probability measure $\eta$ projecting to $\nu^\Z \times \mu$ is of the form 
$$\eta = a \eta^1_{\nu} +(1-a) \eta^2_{\nu}$$
for some $a\in [0,1]$.
\end{claim}
\begin{proof}
For $ (\xi,x, v)\in TX\sm E^1(\xi,x)$  write 
 $v = v^1+ v^2$ with $ v^j\in E^j_{\xi} (x)$ and define 
$\psi( \xi,x, v)\in [0,\infty)$ by 
$$ \psi( \xi,x, v) = \dfrac{\|v^1\|}{\|v^2\|}.$$ 
As $\psi( \xi,x, tv) = \psi( \xi,x, v)$, $\psi$ descends to a   function $$\psi\colon \Pp TX \sm E^1(\xi,x)  \to [0,\infty).$$

For $(\xi,x, v)\in TX\sm E^1(\xi,x)$, we have that 
$$DF^n(\xi,x, v) = \left(\sigma^n(\xi), f_\xi^n(x), D_xf_\xi^n v^1+ D_xf_\xi^n  v^2\right)$$
and hence for any sufficiently small $\epsilon>0$ and $\nu^\Z\times \mu$-\ae $(\xi,x)$ there is a $c$ with 
	$$\psi\left(\Pp DF^n(\xi,x, [v])\right) \le  c\exp \left( n(\lambda^1_\nu-\lambda^2_\nu + 2\epsilon)\right)\psi( \xi,x, [v]).$$
In particular, for almost every $(\xi,x, [v])\in \Pp TX\sm E^1(\xi,x)$ we have  $$\psi\left(\Pp DF^n(\xi,x, [v])\right) \to 0$$ 
as $n\to \infty$.  
By Poincar\'e recurrence, we conclude that $\eta(\psi\inv (0,\infty)) = 0 $.  
\end{proof}

In the context of Theorem \ref{thm:4} we have the following characterization of $su$-measures.  
\begin{lemma}\label{lem:nosu}
Let $\nu\in \M(K) $ be such that $\nu(f)>0$.  Then there exists a $\Pp DF$-invariant, $su$-measure projecting to  $\nu^\Z\times \mu$ if and only if one of the subbundles  $\{E^u_f, E^s_f\}$ or their union $E^u_f\cup E^s_f$ is $\nu$-\as invariant. 
\end{lemma}
\begin{proof}
The only if case is clear.  Indeed if $E^u_f$ is $\nu$-\as invariant then $\eta$ defined by $d \eta = \ d \delta _{E^u_f }\ d \mu \ d \nu^\Z$ is an $su$-measure.  If  the union $E^u_f\cup E^s_f$ is $\nu$-\as invariant we may take $$d \eta = \tfrac 1 2 \ d \delta _{E^u_f }\ d \mu \ d \nu^\Z + 
\tfrac 1 2 \ d \delta _{E^s_f }\ d \mu \ d \nu^\Z.$$  

To prove the converse, suppose $\eta$ is a $\Pp DF$-invariant, $su$-measure on $\Pp TX$.  
As remarked above, there is a measure  $\eta_0$ on $\Pp TM$, projecting to $\mu$, such that $\eta = \nu^\Z \times \eta_0$.  Furthermore, such $\eta_0$ is $Dg$-invariant for $\nu$-\ae $g$.  Since $\nu(f)>0$ we have  $Df_*(\eta_0) = \eta_0$.  However, by the hyperbolicity of $f$ and arguments analogous to the proof of Claim \ref{lem:decomp},
the only such measures are supported on $E^u_f \cup E^s_f$.  
\end{proof}

\subsubsection{Characterization of discontinuity of exponents}
\def \erg{\mathrm{erg}}
Let $\M_\erg(K)\subset \M(K)$  be the set of $\nu$ such that $\mu$ is ergodic for $\nu$.  
Then for $\nu\in\M_\erg(K)$ the Lyapunov exponents 
$\lambda^1_{\nu} \le \lambda^2_{\nu}$ are independent of $x$.  
We study the  continuity properties of  the maps $$\lambda^j_{(\cdot)} \colon \M_\erg \to \R$$ 
as $\nu$ varies in $\M_\erg(K)$ with the weak-$*$ topology.  
 The arguments here are well known. (See for example  \cite{Bocker-Neto:2010aa,viana2014lectures} and references therein.)

\begin{prop}\label{prop:discont}
Let $\nu\in \M_\erg(K)$ be a point of discontinuity for one of $\lambda^1_{(\cdot)}  $, $ \lambda^2_{(\cdot)} $. 
Then \begin{enumerate}
	\item $\lambda^1_{\nu}< \lambda^2_{\nu}$, and 
	\item there exists a $\Pp DF$-invariant $su$-measure $\eta$ projecting to $\nu^\Z\times \mu$.  
\end{enumerate}
\end{prop}

\begin{proof}
We first consider the case where $\lambda^1_{\nu}= \lambda^2_{\nu}$.  Recall then that   $\lambda^1_{\nu} = \lambda^2_{\nu}= 0.$
 Suppose $\lambda^1_{(\cdot)}$ is discontinuous at $\nu\in M_\erg(K)$.  
 Then there is some $\epsilon>0$ and  a sequence $\displaystyle {\nu_j\to \nu}$ in $\M_\erg(K)$ 
with  $\lambda^1_{\nu_j}<-\epsilon< 0$  for every $j$.  
For such $j$, we have two distinct exponents $\lambda^1_{\nu_j}<0\le \lambda^2_{\nu_j}. $   By the pointwise ergodic theorem we have
\[\lambda^1_{\nu_j}= \int \log \|\restrict{D_x f_\xi}{E^1_\xi(x)} \| \ d \mu(x) \ d\nu^\Z(\xi)=
\int \log {\restrict{\|D_x f_\xi}{[ v]}\|}  \ d \eta_{\nu_j}^1(\xi,x,[v])\]
where $\eta_{\nu_j}^1$ are as defined in \eqref{eq:msrs}.
Let $\eta_0$ be an  accumulation point of $\{ \eta_{\nu_j}^1\}.$   Passing to subsequences assume $\eta_{\nu_j}^1 \to \eta_0$.  Since each $\eta_{\nu_j}$ is $\Pp DF$-invariant, it follows that $\eta_0$ is $\Pp DF$-invariant.  

Note that $(\xi,x, E) \mapsto  \| \restrict{D_x f_\xi }{E } \|$ is a continuous 
function on $\Pp TX$.  
By 
 weak-$*$ convergence we have 
\begin{align}-\epsilon
\ge \lim_{j\to \infty}  
\int \log {\restrict{\|D_x f_\xi}{[ v]}\|}  \ d \eta_{\nu_j}^1(\xi,x,[v])
=
\int \log {\restrict{\|D_x f_\xi}{[ v]}\|}  \ d \eta_{0}(\xi,x,[v]).
\label{eq:oddity}
\end{align}
From the pointwise ergodic theorem, \eqref{eq:oddity} implies that for $(\nu^\Z\times \mu)$-\ae $(\xi,x)\in \Sigma \times M$ there is a $v \in T_x M$ with 
$$\lim_{n\to \infty}\frac 1 n \|D\cocycle [\xi] [n] (v) \| <-\epsilon$$ contradicting that  $  \lambda^1_\nu= 0$.  This shows that if $ \lambda^1_\nu= 0$, 
then $\lambda^1_{(\cdot)}$ is continuous at $\nu$.  Similarly,  $\lambda^2_{(\cdot)}$ is continuous at $\nu$ if $ \lambda^2_\nu= 0$.

We  now assume that $\lambda^1_{\nu}< \lambda^2_{\nu}$.   Suppose again that $\lambda^1_{(\cdot)}$ is discontinuous at $\nu$.  
 Then there is a convergent sequence   $\nu_j\to\nu$ in $\M_\erg(K)$ with 
$$\lim _{j\to \infty}\lambda^1_{\nu_j} \neq \lambda^1_{\nu}.$$
We may then select a sequence of $\Pp DF$-invariant  $s$-measures $\eta_j$ projecting to $\nu_j^\Z\times \mu$ with 
\[\lambda^1_{\nu_j}: = \int \log {\restrict{\|D_x f_\xi}{[ v]}\|}    \ d \eta_j(\xi,x,[v]).\]
Indeed if $\lambda^1_{\nu_j}< \lambda^2_{\nu_j}$ we may take the canonical $s$-measures $\eta_j= \eta^1_j$.  Otherwise 
we have $\lambda^1_{\nu_j}= \lambda^2_{\nu_j}=0$  and hence, by Proposition \ref{prop:ledrappier}, we may take $\eta_j$ to be any $\Pp DF$-invariant measure with projection $\nu_j^\Z\times \mu$.  

Let $\eta_0$ be any accumulation point of $\{\eta_j\}$.  Again, $\eta_0$ is $\Pp DF$-invariant and by Lemma \ref{lem:decomp} we have 
\[\eta_0 = \alpha \eta^1_{\nu} +  \beta \eta^2_{\nu},\quad \alpha +\beta=1.\] Moreover, by weak-$*$ convergence we have 
\begin{align*} \alpha \lambda^1_\nu + \beta \lambda^2_\nu &= 
\int \log {\restrict{\|D_x f_\xi}{[ v]}\|}  \ d (\alpha \eta^1_{\nu} +  \beta \eta^2_{\nu} ) (\xi,x,[v]) \\&=
\int \log {\restrict{\|D_x f_\xi}{[ v]}\|}   \ d \eta_0(\xi,x,[v]) &\\&= \lim _{j\to \infty} 
\int \log {\restrict{\|D_x f_\xi}{[ v]}\|}   \ d \eta_j(\xi,x,[v]) &\\&= \lim _{j\to \infty}  \lambda^1_{\nu_j} \neq \lambda^1_{\nu}.
\end{align*}
\noindent It follows that $\alpha\neq 1$, whence $\beta\neq 0$.  
By Claim \ref{claim:smeasuresclosed},   $\eta_0$ is an $s$-measure.  
On the other hand, we have that $\eta^1_{\nu}$ is an $s$-measure and $\eta^2_{\nu}$ is an $u$-measure whence 
$$\eta^2_{\nu} = \frac{1}\beta (\eta_0  - \alpha\eta^1_{\nu} )$$ is an $su$-measure. 
\end{proof}


\newcommand{\inter}[1]{\mathrm{int}(#1)}

\subsubsection{Proof of Theorem \ref{thm:4}: irreducible case} 
We   prove the  conclusion of Theorem \ref{thm:4}\ref{thm:4:1}.  
Let $f$ be as in Theorem \ref{thm:4}.
Under the hypotheses of Theorem \ref{thm:4}\ref{thm:4:1} we 
 may find $g_1, g_2\in \Gamma$ with 
$Dg_1E_f^u \not\subset E_f^s\cup E^u_f$
and $Dg_2E^s_f \not \subset  E_f^s\cup E^u_f$.   
Indeed, without loss of generality we may assume there is $g_2\in \Gamma$ with 
 $Dg_2E_f^s (x) \not\subset \{E^u_f(g_2(x)),E^s_f(g_2(x))\} $ for all $x\in A$ with $\mu(A)>0$.  
  Let $g\in \Gamma$ be such that $DgE_f^u (x) \neq E^u_f(x)$ for all $x\in B$ with $\mu(B)>0$.  If 
$DgE_f^u (x) = E^s_f(g(x)) $ for almost every $x\in B$ then there is some $k$ with $\mu(f^k(g(B)) \cap A)>0
$.  Then take $g_1= g_2\circ f^k\circ g$.



Let $K = \{ f,f\inv, g_1, g_1\inv, g_2, g_2\inv\}$.  Then 
$ \M(K)$  is the  simplex $\Delta$ given by the  convex hull of 
\[\left\{\delta_f, \delta_{f\inv}, \delta_{g_1}, \delta_{g_1\inv}, \delta_{g_2}, \delta_{g_2\inv}\right\}.\]
We write $\inter \Delta$ for the interior of the simplex $\Delta$.  

\begin{proof}[Proof of Theorem \ref{thm:4}\ref{thm:4:1}]
Note that for $\nu\in \inter \Delta$ we have $\nu(f)>0$, whence $\mu$ is ergodic for $\nu$.
It follows from the choice of $g_i$, Proposition \ref{prop:discont}, and Lemma \ref{lem:nosu} that every $\nu \in \inter \Delta$ is a continuity point of the functions $\nu\mapsto \lambda^1_{\nu},\nu\mapsto  \lambda^2_{\nu}$.  Indeed, were $\nu$ a discontinuity point, there would exist a $\Pp DF$-invariant $su$-measure projecting to $\nu^\Z\times \mu$ which by Lemma \ref{lem:nosu} would imply union of the two  distributions $E^s_f\cup E^u_f$ is $Dg_1$ and $Dg_2$ invariant.  
Moreover,  for $\nu \in \inter \Delta$, at least one  $\lambda^1_\nu$, $\lambda^2_{\nu}$ is non-zero.  Indeed by Proposition \ref{prop:ledrappier}, if $\lambda^1_\nu= \lambda^2_{\nu}=0$ then there exists a $\Pp DF$-invariant $su$-measure over $\nu^\Z\times \mu$ which again, by Lemma \ref{lem:nosu}, contradicts the choice of $g_i$.   

Let $P,N\subset \Delta$ be the sets
$$P= \{ \nu\in \Delta \mid \lambda_\nu^2 >0\},\quad N= \{ \nu\in \Delta \mid \lambda_\nu^1 <0\} .$$  By the continuity of $\lambda^j$ the sets $P$ and $N$ are open in $\inter \Delta$.  Furthermore, the simplex $\Delta$ is invariant under the involution \eqref{eq:invol} whence $P$ is non-empty if and only if $N$ is non-empty.  Since there are no $\nu \in \inter \Delta$ with all exponents of $\mu$ of the same sign or all zero, it follows that $\{P,N\}$ is an open cover of $\inter \Delta$.  In particular there exist a $\nu_0 \in \inter \Delta$ such that $\lambda^1_\nu< 0<\lambda^2_\nu$.  

The conclusion then follows from Theorem  \ref{thm:skewproductABS} for $\nu_0$.  
  Indeed 
  we have that $\mu$ is an ergodic, hyperbolic, $\nu_0$-stationary measure that is not finitely supported. 
  Recall that sub-$\sigma$-algebras $\Fol$ and $\Gol$ on $\Sigma \times M$. 
  If $(\xi,x)\mapsto E^1_\xi(x)$  were $\Fol$-measurable, then since is  it $\Gol$-measurable, we have  $E^1_\xi(x) = V(x)$ for some $\nu_0$-\as invariant $\mu$-measurable line field  $V\subset TM$.  
As $\nu_0(f)>0$, by the hyperbolicity of $f$, we can conclude that $V(x)$ coincides with either 
 $E^u_f(x)$ or $E^s_f(x)$ for almost every $x$.  By the ergodicity of $ f$ and $f$-invariance of $V$, $E^u_f$ and $E^s_f$, it follows that $V(x) = E^s_f(x)$ \as or $V(x) = E^u_f(x)$ \as
 The hypotheses on $g_i$ ensure no such $V(x)$ exists and thus the measure  
 $\nu_0^\Z\times \mu$ is fiber-wise-SRB for the skew product $F$. 

%
 
 Repeating the above argument, and using the fact that $\mu$ is $\nu_0$-\as invariant 
 we conclude that 
$\nu^\Z\times \mu$ is fiber-wise-SRB for the skew product $F\inv$.  It follows from the transverse absolute continuity property of stable and unstable manifolds discussed in the proof of Theorem \ref{thm:3}  that $\mu$ is absolutely continuous.  \end{proof}

%



\subsubsection{Proof of Theorem \ref{thm:4}: reducible case}
We  prove Theorem \ref{thm:4}\ref{thm:4:2}. Theorem  \ref{thm:4}\ref{thm:4:3} is proved similarly.   Note that 
in this case, the continuity of exponents  follows immediately from the hypotheses.  

Let $f$ be as in  Theorem \ref{thm:4}\ref{thm:4:2} and take $g\in \Gamma$ with $Dg E^s_f(x)\neq E^s_f(g(x))$ for a positive measure set of $x$.  Let $K = \{ f, f\inv, g, g\inv\}.$
For $t\in [0,1]$ write 
$$\nu_t:= t\delta_f + (1-t) \delta_g.$$
Note that for $t> 0$ we have $\nu_t\in \M_\erg(K)$.  

Write $V(x)= E^u_f(x)$.   By hypotheses, the line field $V$ is preserved by $f$ and $g$.  
Define \begin{equation}\label{eq:jew}\chi(t):= \int t \log \|\restrict{D_xf} {V_x}\| + (1-t) \log \|\restrict{D_xg} {V_x}\| \ d \mu(x).\end{equation}
It follows that $\chi(t)$ is a Lyapunov exponent for the $\nu_t$-stationary measure $\mu$.   Fixing a Riemannian structure on $M$, 
define the average Jacobian $J(\nu_t) = \int t \log |\det D_xf| + (1-t)\log |\det D_xg| \ d \mu(x)$. Then from \eqref{eq:jacobians}
\begin{equation}\label{eq:gentile}\td \chi(t) : = J(\nu_t) - \chi(t)\end{equation} is also a Lyapunov exponent.  
This establishes   the following.  
\begin{claim}\label{cor:contSRB}
For $t\in (0,1]$ the Lyapunov exponents $\lambda_{\nu_t}^j $ are continuous.  
\end{claim}

We continue the proof of the theorem.  
 
\begin{proof}[Proof of Theorem \ref{thm:4}\ref{thm:4:2}]  
Let $t\to 1$.  By 
the hyperbolicity of $f$, from \eqref{eq:jew} and \eqref{eq:gentile},  for $t$ sufficiently close to $1$ $\mu$ has one positive and one negative exponent. 
Moreover, for $t$ sufficient close to $1$, it follows that   the stable bundle for the 
random dynamics  $E^s_\omega(x)$ does not coincide with $E^u_f(x)$ on a set of positive measure.  Thus, were $E^s_\omega(x)$ non-random, as $\nu_t(f)>0$ by the ergodicity of $f$ the line bundle $E^s_\omega(x)$ would have to coincide with $E^s_f$.  
As $g$ does not preserve $E^s_f$, we conclude that $E^s_\omega(x)$ is not non-random.  

As $\mu$ is not finitely supported, by Theorem \ref{thm:main}  it follows  that $\mu$ is an  SRB  $\nu_t$-stationary measure for all sufficiently large $t<1$.  We show $\mu$ is SRB for $f$.   Let $\delta^u$ denote the unstable dimension of $\mu$ with respect to the single diffeomorphism $f\colon M\to M$.  We show below that $\delta^u = 1$ which implies $\mu$ is SRB for $f$.  This  follows from the following entropy trick.

Let $D = \dim (\mu)$.   Recall the fiber-wise entropy and dimension formulae for skew products given by Proposition \ref{prop:ento}.  Similar formulas hold for the individual diffeomorphism $f.$  Suppose for the sake of contradiction that $\delta^u <1$.  Then 
given any $\epsilon>0$, for all sufficiently large $0<t<1$, 
\[\delta^u\lambda^2_{\nu_1} =( D-\delta^u) (-\lambda^1_{\nu_1})> (D- 1) (-\lambda^1_{\nu_t})  = \lambda^2_{\nu_t}  \ge \lambda^2_{\nu_1} - \epsilon.\]
As  $\epsilon \to 0$ as $t\to 1$ this yields a contradiction.  
 We thus have $\delta^u = 1$.  
\end{proof}

\subsection{Proof of  Proposition \ref{thm:perturb} and  Theorem \ref{thm:perturbSRB}}\label{sec:left3}
Recall the joint cone condition and relevant  notation from   Section \ref{sec:alge}.  

\begin{proof}[Proof of Proposition \ref{thm:perturb}]
If $A$ and $B$ don't commute, it follows that $E^s_A\neq E^s_B$ and $E^u_A\neq E^u_B$. Then for $n>0$ large enough, we have  $A^{-n}C^s$ and $B^{-n} C^s$ are disjoint.  We take $f$ and $g$ sufficiently close to $L_A $ and $L_B$ so that  for some $\kappa>1$ and any $x\in M$
\begin{enumerate}
\item $D_{f(x)} f^{-1} C^s\subset C^s$ and  $D_{g(x)} g^{-1} C^s\subset C^s$;
\item if $v\in C^s$ then $\|D_{g(x)} g^{-1} v\| > \kappa \|v\| $ and   $\|D_{f(x)} f^{-1} v\| > \kappa \|v\| $;
\item $D_{f^n(x)} f^{-n} C^s$ and $D_{g^n(x)} g^{-n} C^s$ are disjoint in $T_x \T^2$.
\end{enumerate}
We further assume analogous  properties to the above hold relative to the unstable cones.

Let $\Sigma_{+}= \{f,g\}^\N$.  
Given  $\omega = (f_0, f_1, f_2, \dots)\in \Sigma_+$ define 
$$E^s_\omega( x):= \bigcap_{i= 0 }^M  
D_{(f_M  \circ \dots \circ f_0)(x)}\left(f_M  \circ \dots \circ f_0\right) \inv (C^s).$$ 

The set  $E^s_\omega( x)$ is invariant under scaling; moreover,  the cone conditions ensure $E^s_\omega( x)$ is non-empty for every $\omega$ and every $x$.    
Note that if $v\in E^s_\omega( x)$ then for any $j\ge 0 $, $\| D(f_j\circ\cdots \circ f_0) v \| \in C^s$  hence we have $$\| D(f_j\circ\cdots \circ f_0) v \|  \le \kappa ^{-j} \|v\|.$$ Similarly, if $u\in C^u$ then 
$\| D(f_j\circ\cdots \circ f_0) u \| \in C^u$ for any $j$ and hence we have $\| D(f_j\circ\cdots \circ f_0) v \|  \ge \kappa ^{j} \|v\|$.  It follows that every $\nu$-stationary measure is hyperbolic with one exponent of  each sign.  
We claim that $E^s_\omega( x)$ is a 1-dimensional subspace.  Indeed, if otherwise  there are non-zero $v,u\in E^s_\omega( x)$ with $v =u+w$ for $w\in C^u$.  
But then  for $M$ sufficiently large we obtain a contradiction as 
 \begin{align*}
 \kappa^M \|w\| \le  \| D(f_M\circ\cdots \circ f_0) w \|&\le \| D(f_M\circ\cdots \circ f_0) u \| + \|D(f_M\circ\cdots \circ f_0) v\| \\&\le \kappa^{-M} (\|v\|+ \|u\|).\end{align*}
 In particular, for any $\nu$-stationary measure, $E^s_\omega( x)$ coincides with the stable Lyapunov subspace  for the word $\omega$ at $x$.

Recall that the cones $D_{f^n(x)}f^{-n}C^s$ and $D_{g^n(x)}g^{-n}C^s$ are disjoint.  As the  set of words $\omega = (f_0, f_1, f_2, \dots) \in \Sigma$ with $f_i= f$ for $0\le i\le n$ and  set of words $\omega = (f_0, f_1, f_2, \dots) \in \Sigma$ with $f_i= g$ for all $0\le i\le n$ have positive $\nu^\N$-measure, it follows that the distribution $E^s_\omega( x)$ is not non-random for every $\nu$-stationary measure.  

%

It then follows from Theorem \ref{thm:1+} that any ergodic, $\nu$-stationary measure $\mu$ on $\T^2$ is either SRB or finitely supported.
Moreover, fixing $f$, by choosing a generic perturbation $g$, for any periodic point $p$ for $f$  we may further assume that  $p$ is not a periodic point  for $g$.  Then, as $f$ and $g$ have no common finite invariant subsets,  there are no finitely supported $\nu$-stationary measures.  
\end{proof}

 \begin{proof}[Proof of Theorem \ref{thm:perturbSRB}]
Recall in the statement of 
Theorem \ref{thm:perturbSRB} we set $\nu_0= \sum p_k \delta_{L_{A_k}}$.  We take $\td \nu_0 =  \sum p_k \delta_{{A_k}}$ on $\Sl(2,\Z)$. 
Consider $\mu$ any $\nu_0$-stationary measure. The Lyapunov exponents of $\mu$ coincide with the Lyapunov exponents of the random product of matrices given by $\td \nu_0$.  
In particular, the Lyapunov exponents of $\mu$ are constant \as and independent of the choice of $\mu$.  
As $\Gamma\subset \Sl(2,\Z)$ is infinite and does not have $\Z$ as a finite-index subgroup, it follows that $\Gamma$ is not contained in a compact subgroup and that  any line $L\in \R\Pp^1$ has infinite $\Gamma$-orbit.  By a theorem of Furstenberg (\cite[Theorem 8.6]{MR0163345}, see also \cite[Theorem 6.11]{viana2014lectures}) it follows that the random product of matrices given by $\td \nu_0$  has one positive and one negative Lyapunov exponent.  The same is  then true for  any $\nu_0$-stationary measure on $\T^2$.
Moreover, as $\Gamma$ is not virtually-$\Z$, one can find hyperbolic elements $B_1, B_2\in\Gamma$ that satisfy a joint cone property (defined in  Section \ref{sec:alge}) and such that $B_1$ and $B_2$ do not commute. 
Write  \begin{equation}\label{eq:generate}B_1 = A_{i_1} A_{i_2} \dots A_{i_\ell},\quad  B_2 = A_{j_1} A_{j_2} \dots A_{j_p}\end{equation}in terms of the generators.  

For each $1\le k\le n$ take a neighborhood $ L_{A_k}\in U_k\subset \diff^2(\T^2)$ sufficiently small so that 
\begin{enumerate}
\item $|g_k^{\pm1}|_{C^2} \le C$ for all $g_k\in U_k$ and some $C>0$, and
\item  
writing 
$$f_1 = g_{i_1} \circ \dots \circ g_{i_\ell}, \quad f_2= g_{j_1} \circ \dots \circ g_{j_p} $$
as in \eqref{eq:generate} for any choice of  $g_{i_\ell}\in U_{i_\ell}$ and  $g_{j_m}\in U_{j_m}$, 
$f_1$ and $f_2$ are sufficiently close to $B_1$ and $B_2$ so that Proposition \ref{thm:perturb} holds.  
\end{enumerate} 
In particular, such $f_1$ and $f_2$ satisfy a joint cone condition, are Anosov diffeomorphisms of $\T^2$, and $E^s_{f_1} (x) \neq E^s_{f_2}(x)$ and $E^u_{f_1} (x) \neq E^u_{f_2}(x)$ for any $x\in \T^2$.  
%
%
%
%

Take $U\subset \diff^2(\T^2)$ in the  theorem to be the set $U= \{ g\in \diff^2(\T^2) : |g|_{C^2} < C\}$.  
Let $\nu$ be a probability measure on  $U$. 
We moreover assume $\nu$ is sufficiently close to $\nu_0$ so that $\nu(U_k)>0$ for each $1\le k\le n$.

We introduce some notation.  Given $f\in \Diff^2(\T^2)$, consider $\Pp D f $ acting on the projectivized tangent bundle $\Pp T\T^2$.   We naturally identify  $\nu$ with a measure on $\{\Pp D f: f\in \diff^2(\T^2)\}$.  
Consider  a $\nu$-stationary probability measure $\eta$ on $\Pp T\T^2$.  Note that the projection of $\eta$ onto $\T^2$ is also a $\nu$-stationary measure.  

Given $f\in \diff^2(\T^2)$, write 
$$ \Phi(f,x,E)  = \log(\|\restrict{D_xf}{E}\|).$$
Note that $\Phi \colon \diff^2(\T^2) \times \Pp T\T^2\to \R$ is continuous and uniformly bounded on $U$.  

\begin{lemma}\label{lem:skylight}
For all $\nu$  sufficiently close to $\nu_0$,  every ergodic, $\nu$-stationary measure  on $\T^2$ has a positive Lyapunov exponent.   
\end{lemma}
\begin{proof}
Suppose $\nu_k\to \nu_0$ on $U$  in the weak-$*$ topology and that for each $k$, there is an ergodic, $\nu_k$-stationary measure $\mu_k$ with only non-positive exponents.  For each $k$ we may select a  $\nu_k$-stationary probability measure $\eta_k$ on $\Pp T\T^2$ projecting to $\mu_k$ such that $$\int \int  \Phi(f,x,E)  \ d \eta_k(x, E) \ d \nu_k(f)\le 0.$$
Indeed, the existence and construction of $\eta_k$  is essentially the same  as in the proof of Proposition \ref{prop:discont} and \eqref{eq:msrs}.  

As $\Pp T \T^2$ is compact, let $\eta_0$ be an accumulation point of $\{\eta_k\}$.  Then (see for example \cite[Proposition 5.9]{viana2014lectures}) $\eta_0$ is $\nu_0$-stationary.  Moreover,  $\eta_0$ projects to a $\nu_0$-stationary measure $\mu_0$ on $\T^2$ and by weak-$*$ convergence (and boundedness of $ \Phi(f,x,E)  $ on $U$ )
$$\int \int  \Phi(f,x,E)  \ d \eta_0(x,E)\ d \nu_0(f)\le 0.$$  

Recall we define  $\td \nu_0 = \sum p_k \delta _{A_k}$ to be a measure on $\Gamma\subset \Sl(2,\Z)$.  
Note that $T \T^2$ is parallelizable so $\Pp T\T^2 = \T^2\times  \R\Pp^1$. 
Then 
define a factor measure $\td \eta_0$ on $\R\Pp ^1$ by $$\td \eta_0(D):=\eta_0(\T^2 \times D).$$
We have that $\td \eta_0$ is a $\td \nu_0$-stationary measure for the natural action of $\Sl (2, \R)$ on $\R\Pp^1$.  
Moreover, with $\td \Phi(A,E) = \log \| \restrict A E\|$ we have 
$$\int \int \td \Phi (A,E ) \ d \td \eta(E) \ d \td \nu_0 (A) \le 0$$
On the other hand, by a  theorem of Furstenberg (\cite[Theorem 8.5]{MR0163345},\cite[Theorem 6.8]{viana2014lectures}) this is impossible under our hypotheses on $\Gamma$.
\end{proof}

Take $\nu$ sufficiently close to $\nu_0$ so that every ergodic, $\nu$-stationary measure  on $\T^2$ has a positive Lyapunov exponent.   
Consider  $\mu$  an ergodic, $\nu$-stationary measure on $\T^2$.  Suppose that all exponents of $\mu$ were non-negative.  By the invariance principle  in \cite{MR2651382}, it would follow that  $\mu$ is invariant for $\nu$-\ae $f\in \diff^2(\T_2)$.  In particular, the sets of $f_1$ and $f_2$ constructed above for which 
$\mu$ is simultaneously $f_1$- and $f_2$-invariant have positive measure.   
As  $f_1$ does not preserve $E^s_{f_2}$, $E^u_{f_2}$, or their union,  Theorem \ref{thm:4} implies that either $\mu$ is atomic or is absolutely continuous.  If $\mu$ were absolutely continuous then, as $\nu$-a.e.\ $f$ preserves an absolutely continuous measure it follows from \eqref{eq:jacobians} that $\mu$ is necessarily hyperbolic.  Hence, for all $\nu$ satisfying Lemma \ref{lem:skylight}, every ergodic, $\nu$-stationary measure is either atomic or is hyperbolic with one exponent of each sign.  

%

In the case that   $\mu$  is hyperbolic with one exponent of each sign,  
we claim that the  stable line fields for $\mu$ are not non-random.  
\begin{definition}
$\nu$ is \emph{strongly expanding} if, for any $\nu$-stationary measure $\eta$ on $\Pp T \T^2$,
$$\int \int \Phi(f,x,E) \ d \eta (x, E) d \nu(f)>0.$$
\end{definition}

Let $\mu$ be an ergodic, hyperbolic, $\nu$-stationary measure with one exponent of each sign.  Suppose the stable line bundle is non-random.  That is,  $E^s_\omega (x) = V(x)$ for some measurable line-field $V(x)$ on $T\T^2$.  
Let $\eta$ be the measure on $\Pp T \T^2$ defined as follows: for measurable $\psi\colon\Pp T \T^2 \to \R$ set
 $$\int \psi(x,E) \ d \eta (x,E)  = \int \psi(x, V(x)) \ d \mu(x).$$
 It follows from the invariance of $E^s_\omega(x)$ that $\eta $ is a $\nu$-stationary measure.  Moreover, from the pointwise ergodic theorem we have 
 $$\int \int \Phi \ d \eta \ d \nu<0.$$
Thus, \begin{claim}
If $\nu$ is {strongly expanding} then the stable line bundle for any hyperbolic, $\nu$-stationary measure $\mu$ is not non-random.  
\end{claim}

As in the previous lemma we have
\begin{lemma}
Every $\nu$ sufficiently close to $\nu_0$ is strongly expanding.
\end{lemma}

From the above, it follows  that for all $\nu$ sufficiently close to $\nu_0$,  any ergodic $\nu$-stationary measure $\mu$ which is not atomic   is hyperbolic with one exponent of each sign and, moreover, the stable  line-field  for $\mu $ is not non-random.  From Theorem \ref{thm:1+}, it follows that if $\mu$ is non-atomic then  $\mu$ is  SRB for $\nu$.  
\end{proof}

%
\end{document}

%% file: controlledintersections.tex
\begin{figure}[h]
\def\dsz{5pt}
\def\lwd{1.5pt}
\def\lwda{.8pt}
\def\framecolor{black}
\def\labsizea{\normalsize}
\def\labsizec{\large}
\def\labsizeb{\small}

\psscalebox{1.0 1.0} 
{
\psset{unit=.6cm}

\begin{pspicture}(0,-4.5)(14.13852,4)

\psdots[linecolor=black, dotsize=\dsz](2.9,-2.3)
\psdots[linecolor=black, dotsize=\dsz](5.0,1.3)
\psdots[linecolor=black, dotsize=\dsz](10.9,-2.6)
\psdots[linecolor=black, dotsize=\dsz](11.3,-1.9)
\psbezier[linecolor=black, linewidth=\lwd](0.41316566,-2.50676)(1.3083032,-2.4573998)(1.8393892,-2.2313633)(2.9,-2.3)(3.8374684,-2.3189983)(6.3595386,-2.5280867)(7.358429,-2.5751808)(8.357319,-2.6222749)(9.898433,-2.652231)(10.9,-2.6)
(11.898424,-2.6581306)(12.90618,-2.58966)(14.138429,-2.5951807)
\psbezier[linecolor=black, linewidth=\lwd](1.9417126,-3.7276871)(2.3444693,-3.2858083)(2.537873,-2.8332872)(2.9,-2.3)(3.3020096,-1.7429138)(4.6989446,0.27557805)(5.0,1.3)(5.4092207,2.4233994)(6.7563324,3.6614165)(7.1477766,3.9194307)
\psbezier[linecolor=black, linewidth=\lwd](9.958428,-3.8951807)(10.084205,-3.7256408)(10.757951,-2.891771)(10.9,-2.6)
(11.078907,-2.3722749)(11.234731,-2.1035357)(11.3,-1.9)(11.362127,-1.7268261)(12.48838,1.5908831)(13.3384285,3.2048192)

\psbezier[linecolor=black, linewidth=\lwd,linestyle=dashed,  dash=3pt 2pt](0.93948144,3.554293)(2.5927973,2.9731712)(3.7508197,2.2609324)(5.0,1.3)(6.3480587,0.41415808)(7.2929063,-0.2473017)(8.220923,-0.6034756)(9.148941,-0.95964944)(10.766837,-1.4823141)(11.3,-1.9)(11.9598675,-2.3073137)(12.927789,-2.749726)(13.794218,-3.4678123)
\psbezier[linecolor=black, linewidth=\lwd, linestyle=dashed,  dash=3pt 2pt](0.011060381,-0.022549238)(0.98336405,-0.66720515)(2.1136239,-1.6998501)(2.9,-2.3)(3.7190232,-2.8926167)(4.486175,-3.645016)(5.2531657,-3.9172862)

\uput{8pt}[270](2.9,-2.3){\labsizea$x$}
\uput{8pt}[0](5.0,1.3){\labsizea$v$}
\uput{5pt}[-55](10.9,-2.6){\labsizea$z$}
\uput{8pt}[25](11.3,-1.9){\labsizea$y$}

\uput[-45](5.2531657,-3.9172862){\labsizeb $W^s_\eta(x)$}
\uput[270](13.794218,-3.4678123){\labsizeb$W^s_\eta(y)$}
\uput[270](1.9417126,-3.7276871){\labsizeb$W^u_\xi(x) $} 
\uput[270](9.958428,-3.8951807){\labsizeb$W^u_\xi(y)$}
\uput[0](14.138429,-2.5951807){\labsizeb$W^s_\xi(x)$} 

\end{pspicture}

}

\caption{Lemma \ref{prop:controlledintersections}.}\label{fig:1}

\end{figure}

%% file: fig3.tex
\begin{figure}[h]\def\dsz{4pt}
\def\lwd{1.5pt}
\def\lwda{.8pt}
\def\framecolor{black}
\def\labsizea{\tiny}
\def\labsizec{\large}
\def\labsizeb{\tiny}
\psset{unit=.7cm}

\begin{minipage}{.3\textwidth}
\centering
\begin{pspicture}(0,-4.5)(5,4.5) 
\psbezier[linecolor=black, linewidth=\lwda](0.73686975,-4.449636)(0.7127771,-4.460336)(1.8730544,-4.009415)(2.8716369,-3.8637156)(3.8702192,-3.7180164)(5.073722,-3.8698223)(5.1290383,-3.9136796)(5.1843553,-3.957537)(5.0430784,-1.3179214)(4.6728888,0.97566915)(4.3026996,3.2692597)(4.2267647,4.4525065)(4.2357507,4.453625)(4.244736,4.454744)(3.3139105,3.9660668)(2.610363,3.747806)(1.9068155,3.5295453)(0.05490279,3.8556824)(0.022326251,3.8718326)(-0.010250286,3.8879828)(0.32000947,1.2388599)(0.6153129,-0.39083123)(0.9106163,-2.0205224)(0.7609624,-4.4389358)(0.73686975,-4.449636)
\psbezier[linecolor=black, linewidth=\lwd](1.5924122,2.748983)(1.7714443,2.4981642)(2.136812,1.1666895)(2.0752194,-0.008463906)(2.0136266,-1.1836174)(2.0082865,-1.0431124)(2.0057456,-1.5642086)(2.0032048,-2.0853047)(2.0259876,-2.379708)(2.1324122,-2.791017)
\psbezier[linecolor=black, linewidth=\lwd](4.4037166,0.23663507)(4.1770067,0.09837163)(4.0394573,0.053873222)(3.6394575,-0.0015866251)(3.2394571,-0.057046473)(2.2494652,-0.04379658)(1.5242893,0.029158738)

\psdots[dotsize=\dsz]
(2.0729,-0.011017098)
(1.9513657,2.3384326)
(2.2506042,-0.011043609)
(1.7374282,2.3866103)

\uput[225](2.0729,-0.011017098){\labsizea $\tilde x_m$}
\uput[45](1.9513657,2.3384326){\labsizea $\tilde y _m$}
\uput[-45](2.2506042,-0.011043609){\labsizea $\tilde z_m $}
\uput[135](1.7374282,2.3866103){\labsizea $\tilde w_m $}

\psbezier[linecolor=black, linewidth=\lwd](1.7932395,2.7234473)(2.054135,2.1398587)(2.3303888,1.0958833)(2.2272973,0.06283411)(2.1242058,-0.970215)(2.1548598,-1.5416058)(2.1856308,-2.0001223)(2.2164018,-2.4586391)(2.2617662,-2.458341)(2.2954133,-2.6132917)
\psbezier[linecolor=black, linewidth=\lwd](4.075615,2.541407)(3.8545492,2.443519)(3.5254943,2.3946238)(3.1354525,2.3553593)(2.7454107,2.3160949)(1.9749955,2.356681)(1.2678735,2.408332)
\end{pspicture}
\end{minipage}
\begin{minipage}{.35\textwidth}
\centering
\begin{pspicture}(0,-4.5)(5,4.5)
\psbezier[linecolor=black, linewidth=\lwda](0.73686975,-4.449636)(0.7127771,-4.460336)(1.8730544,-4.009415)(2.8716369,-3.8637156)(3.8702192,-3.7180164)(5.073722,-3.8698223)(5.1290383,-3.9136796)(5.1843553,-3.957537)(5.0430784,-1.3179214)(4.6728888,0.97566915)(4.3026996,3.2692597)(4.2267647,4.4525065)(4.2357507,4.453625)(4.244736,4.454744)(3.3139105,3.9660668)(2.610363,3.747806)(1.9068155,3.5295453)(0.05490279,3.8556824)(0.022326251,3.8718326)(-0.010250286,3.8879828)(0.32000947,1.2388599)(0.6153129,-0.39083123)(0.9106163,-2.0205224)(0.7609624,-4.4389358)(0.73686975,-4.449636)
\psbezier[linecolor=black, linewidth=\lwd](1.5924122,2.748983)(1.7714443,2.4981642)(2.136812,1.1666895)(2.0752194,-0.008463906)(2.0136266,-1.1836174)(2.0082865,-1.0431124)(2.0057456,-1.5642086)(2.0032048,-2.0853047)(2.0259876,-2.379708)(2.1324122,-2.791017)
\psbezier[linecolor=black, linewidth=\lwd](2.0176296,2.6063743)(2.2785254,2.0227857)(2.554779,0.9788101)(2.4516876,-0.054239064)(2.348596,-1.0872881)(2.3792498,-1.658679)(2.410021,-2.1171956)(2.440792,-2.5757122)(2.4861565,-2.575414)(2.5198035,-2.730365)
\psbezier[linecolor=black, linewidth=\lwd](4.4037166,0.23663507)(4.1770067,0.09837163)(4.0394573,0.053873222)(3.6394575,-0.0015866251)(3.2394571,-0.057046473)(2.2494652,-0.04379658)(1.5242893,0.029158738)
\psbezier[linecolor=black, linewidth=\lwd]
(4.1341515,1.8975047)(3.9130857,1.7996167)(3.5840309,1.7507213)(3.193989,1.7114569)(2.8039472,1.6721923)(2.0335321,1.7127786)(1.3264102,1.7644296)

\psdots[dotsize=\dsz](2.0729,-0.011017098)
(2.3220975,1.6945301)
(2.4554822,-0.020799706)
(1.922794,1.7231956)

\uput[215](2.0729,-0.011017098){\labsizea $x$}
\uput[45](2.3220975,1.6945301){\labsizea $y $}
\uput[250](2.1049123,-2.7387345){\labsizeb $W^u_\xi(x)$}
\uput[310](2.519586,-2.6936257){\labsizeb $W^u_\xi(y )$}
\uput{0pt}[120](1.5242893,0.029158738) 
{\labsizeb $W^s_\xi(x)$}
\uput[-45](2.4554822,-0.020799706){\labsizea $z $}
\uput[135](1.922794,1.7231956){\labsizea $w $}
\uput{2pt}[250](1.3445861,1.7722279){\labsizeb $W^s_\xi(y )$}
\end{pspicture}
\end{minipage}
\begin{minipage}{.3\textwidth}
\centering
\begin{pspicture}(0,-4.5)(5,4.5)
\psbezier[linecolor=black, linewidth=\lwda](0.73686975,-4.449636)(0.7127771,-4.460336)(1.8730544,-4.009415)(2.8716369,-3.8637156)(3.8702192,-3.7180164)(5.073722,-3.8698223)(5.1290383,-3.9136796)(5.1843553,-3.957537)(5.0430784,-1.3179214)(4.6728888,0.97566915)(4.3026996,3.2692597)(4.2267647,4.4525065)(4.2357507,4.453625)(4.244736,4.454744)(3.3139105,3.9660668)(2.610363,3.747806)(1.9068155,3.5295453)(0.05490279,3.8556824)(0.022326251,3.8718326)(-0.010250286,3.8879828)(0.32000947,1.2388599)(0.6153129,-0.39083123)(0.9106163,-2.0205224)(0.7609624,-4.4389358)(0.73686975,-4.449636)
\psbezier[linecolor=black, linewidth=\lwd](1.5924122,2.748983)(1.7714443,2.4981642)(2.136812,1.1666895)(2.0752194,-0.008463906)(2.0136266,-1.1836174)(2.0082865,-1.0431124)(2.0057456,-1.5642086)(2.0032048,-2.0853047)(2.0259876,-2.379708)(2.1324122,-2.791017)

\psbezier[linecolor=black, linewidth=\lwd](3.4908004,2.6453986)(3.751696,2.06181)(4.02795,1.0178344)(3.9248583,-0.015214675)(3.8217669,-1.0482638)(3.8524206,-1.6196545)(3.8831916,-2.0781713)(3.9139628,-2.5366879)(3.9593272,-2.5363896)(3.9929743,-2.6913404)
\psbezier[linecolor=black, linewidth=\lwd](4.4037166,0.23663507)(4.1770067,0.09837163)(4.0394573,0.053873222)(3.6394575,-0.0015866251)(3.2394571,-0.057046473)(2.2494652,-0.04379658)(1.5242893,0.029158738)
\psbezier[linecolor=black, linewidth=\lwd](4.2609806,0.5414071)(4.039915,0.44351918)(3.7108603,0.39462382)(3.3208184,0.3553593)(2.9307766,0.31609476)(2.1603613,0.35668102)(1.4532394,0.408332)


\psdots[dotsize=\dsz](2.0729,-0.011017098)
(3.9318535,0.43599355)
(3.9481652,0.047492977)
(2.0886476,0.38661015)

\uput[215](2.0729,-0.011017098){\labsizea $\hat x_m$}
\uput[55](3.9318535,0.43599355){\labsizea $\hat y_m $}
\uput[-45](3.9481652,0.047492977){\labsizea $\hat z_m $}
\uput[135](2.0886476,0.38661015){\labsizea $\hat w_m $}

\end{pspicture}
\end{minipage}
%
%
%
%
%
\caption{Claim \ref{claim:GP}}\label{fig:22}

\end{figure}

%% file: choicesinendgame.tex
\begin{figure}[h]
\def\dsz{5pt}
\def\lwd{1.5pt}
\def\lwda{.8pt}
\def\framecolor{black}
\def\labsizea{\small}
\def\labsizeb{\tiny}
\def\labsizec{\normalsize}
\def\unithere{.85cm}
\def\opa{.85}
\centering
\begin{minipage}{.5\textwidth}
  \centering
\psscalebox{1.0 1.0} 
{
\def\rightside{6}
\def\leftside{0}
\psset{unit=\unithere}

\begin{pspicture}(\leftside,-5)(\rightside,5.0)
\psbezier[linecolor=black, linewidth=\lwda](0.5511891,-4.0814466)(0.52709645,-4.0921464)(1.2050209,-3.780572)(1.7524921,-3.6593173)(2.2999632,-3.5380623)(3.1119628,-3.4225478)(3.1672795,-3.4664052)(3.2225962,-3.5102625)(2.9017115,-1.9499934)(3.031522,0.2835971)(3.1613324,2.5171876)(2.885398,4.226317)(2.8943837,4.2274356)(2.9033692,4.2285542)(2.6948965,3.7763476)(1.9113489,3.5580869)(1.1278015,3.3398263)(0.054058746,3.3545907)(0.021482209,3.370741)(-0.011094329,3.3868914)(0.54323477,1.5530454)(0.5599874,-0.43635595)(0.5767401,-2.4257572)(0.57528174,-4.0707464)(0.5511891,-4.0814466)
\psbezier[linecolor=black, linewidth=\lwd, linestyle=dashed,arrowinset=0.2]{-<}(1.3550848,1.5453709)(1.8993475,1.6551495)(2.6271996,1.7623473)(3.308346,1.9027885)(3.9894924,2.0432296)(5.14913,2.1716938)(5.9400935,2.1601727)
\psbezier[linecolor=black, linewidth=\lwd, linestyle=dashed,arrowinset=0.2]{-<}(3.3590202,1.4081529)(3.7784555,1.5220177)(4.067908,1.5505306)(4.4763713,1.6101002)(4.8848352,1.6696697)(5.2818594,1.759988)(5.963242,1.7589179)

\psbezier[linecolor=black, fillstyle=solid, opacity=\opa, linewidth=\lwda](2.1645224,-4.217002)(2.1404297,-4.2282677)(3.176132,-3.7411213)(3.7369366,-3.573006)(4.297741,-3.4048908)(5.054185,-3.3741274)(5.109502,-3.4203033)(5.1648183,-3.4664793)(5.24432,-1.2220107)(5.0463047,0.49615484)(4.848289,2.2143204)(5.0436587,4.452345)(5.0526443,4.4535227)(5.0616302,4.4547)(4.0548964,3.8787544)(3.431349,3.522611)(2.8078015,3.1664674)(1.8629476,3.0837445)(1.8303711,3.1007488)(1.7977946,3.1177528)(2.0240011,1.4169611)(1.9189147,-0.84493685)(1.8138282,-3.106835)(2.188615,-4.205736)(2.1645224,-4.217002)

\psbezier[linecolor=black, linewidth=\lwd](1.4055722,2.919008)(1.3227805,2.4751565)(1.3160833,2.313473)(1.3455722,1.5895457)(1.375061,0.8656185)(1.5227804,-1.077675)(1.5055722,-1.5125332)(1.488364,-1.9473913)(1.475061,-2.3094351)(1.3922389,-2.512103)
\psbezier[linecolor=black, linewidth=\lwd](2.7477944,2.839008)(2.8380897,2.0482585)(2.875171,1.8896512)(2.8855722,1.4547127)(2.8959734,1.0197744)(2.7277193,-0.9828825)(2.7226093,-1.4981731)(2.7174993,-2.0134637)(2.6989362,-2.4412546)(2.6389055,-2.6743252)
\psbezier[linecolor=black, linewidth=\lwd](3.3455722,2.899008)(3.3722932,2.6127858)(3.452181,2.1383276)(3.3715181,1.4299171)(3.2908552,0.72150666)(3.2852662,-1.1180443)(3.3542209,-1.5961698)(3.4231756,-2.0742953)(3.4903145,-2.3140664)(3.5055723,-2.560992)
\psbezier[linecolor=black, linewidth=\lwd](2.2655723,1.3445895)(2.6918073,1.3641443)(3.765086,1.5824842)(4.404346,1.1372603)

\psdots[dotsize=\dsz](3.3480723,1.414008)
\psdots[dotsize=\dsz](2.873451,1.4128717)
\psdots[dotsize=\dsz](2.7230723,-1.5609919)
\psdots[dotsize=\dsz](1.5105722,-1.6359919)
\psdots[dotsize=\dsz](1.3480722,1.539008)

%

\uput[35](3.3480723,1.414008){\labsizea$q'_j$}
\uput[135](2.873451,1.4128717){\labsizea$v'_j$}
\uput[180](2.7230723,-1.5609919){\labsizea$p'_j$}
\uput[180](1.5105722,-1.6359919){\labsizea$p_0$}
\uput[180](1.3480722,1.539008){\labsizea$q_0$}

\uput[270](4.422253,1.1137438){\labsizeb$W^s_{r_1}(q'_j)$}
\uput[310](3.4980721,-2.573492){\labsizeb$W^u_{r_1}(q'_j)$}
\uput[270](2.6314056,-2.6707141){\labsizeb$W^u_{r_1}(p'_j)$}
\uput[260](1.4022388,-2.510992){\labsizeb$W^u_{r_1}(p_0)$}
\uput[90](5.9400935,2.1601727){\labsizea$\Phi^{\hat t}$}
\uput[270](5.963242,1.7589179){\labsizea$\Phi^{t_j}$}

\uput[270](0.7077944,-4.272103){\labsizec$M_{(\varsigma_0, \xi_0)}$}
\uput[270](2.9868283,-4.2568374){\labsizec$M_{(\varsigma'_j, \xi'_j)}$}

\end{pspicture}

}

\end{minipage}%
\begin{minipage}{.5\textwidth}
\centering

\psset{unit=\unithere}

\psscalebox{1.0 1.0} 
{\def\rightside{5}

\def\leftside{-1}

\begin{pspicture}(\leftside,-5)(\rightside,5.0)

\psbezier[linecolor=black, linewidth=\lwda](0.68985224,-4.212293)(0.66575956,-4.2229934)(1.6359917,-3.9295385)(2.634574,-3.8146083)(3.6331563,-3.6996784)(4.2972927,-3.6089501)(4.352609,-3.6528077)(4.4079256,-3.696665)(4.313708,-1.6452847)(3.9435184,0.64830583)(3.573329,2.9418964)(3.4973943,4.4310255)(3.50638,4.432144)(3.5153656,4.433263)(2.9284313,4.105672)(2.2248838,3.887411)(1.5213364,3.6691504)(0.054944072,3.645966)(0.022367533,3.6621163)(-0.010209003,3.6782668)(0.31249616,0.76339513)(0.6425822,-0.77933955)(0.9726682,-2.3220742)(0.71394485,-4.2015934)(0.68985224,-4.212293)
\psbezier[linecolor=black, linewidth=\lwd](1.7514148,3.188332)(1.9304469,2.9375134)(2.2958145,1.6060388)(2.2342217,0.43088537)(2.172629,-0.74426806)(2.167289,-0.60376316)(2.1647482,-1.1248593)(2.1622074,-1.6459554)(2.1849902,-1.9403589)(2.2914147,-2.351668)
\psbezier[linecolor=black, linewidth=\lwd](2.2983713,3.237028)(2.559267,2.6534393)(2.8355207,1.6094637)(2.7324293,0.5764146)(2.6293378,-0.45663458)(2.6599915,-1.0280254)(2.6907625,-1.486542)(2.7215338,-1.9450586)(2.816898,-2.2697604)(2.8505452,-2.4247112)
\psbezier[linecolor=black, linewidth=\lwd](3.5714147,0.78033215)(3.32424,0.65024805)(3.1752126,0.504173)(2.8363152,0.42037135)(2.4974177,0.3365697)(2.1602266,0.31650907)(1.4627192,0.42674696)
\psdots[dotsize=\dsz](2.22,0.35)
\psdots[ dotsize=\dsz](2.67,1.8556185)
\def\td{\tilde}
\uput[135](2.22,0.35){\labsizea$\td p_j$}
\uput[0](2.67,1.8556185){\labsizea$\td q_j$}
\uput[270](2.1554952,-4.312002){\labsizec$M_{(\td \varsigma_j, \td \xi_j)}$}

\uput[190](2.2639148,-2.3341677){\labsizeb$W^u_{r_1}(\td p_j)$}
\uput[-10](2.8514147,-2.446668){\labsizeb$W^u_{r_1}(\td q_j)$}
\uput[260](1.4639148,0.42833218){\labsizeb$W^s_{r_1}(\td p_j)$}

\end{pspicture}
}

\end{minipage}
\psset{unit=\unithere}

\psscalebox{1.0 1.0} 
{
\def\rightside{6.5}

\begin{pspicture}(0,-5)(\rightside,5.0)

\psbezier[linecolor=black, linewidth=\lwda,fillstyle=solid, opacity =\opa](1.4825524,-4.0814466)(1.4584597,-4.0921464)(2.1363842,-3.780572)(2.6838553,-3.6593173)(3.2313266,-3.5380623)(4.043326,-3.4225478)(4.098643,-3.4664052)(4.1539593,-3.5102625)(3.8330746,-1.9499934)(3.9628851,0.2835971)(4.0926957,2.5171876)(3.8167613,4.226317)(3.8257468,4.2274356)(3.8347325,4.2285542)(3.6262598,3.7763476)(2.8427122,3.5580869)(2.0591648,3.3398263)(0.985422,3.3545907)(0.9528455,3.370741)(0.92026895,3.3868914)(1.474598,1.5530454)(1.4913508,-0.43635595)(1.5081034,-2.4257572)(1.5066451,-4.0707464)(1.4825524,-4.0814466)

\psbezier[linecolor=black, linewidth=\lwda,fillstyle=solid, opacity =\opa](3.0958858,-4.217002)(3.071793,-4.2282677)(4.1074953,-3.7411213)(4.6682997,-3.573006)(5.2291045,-3.4048908)(5.9855485,-3.3741274)(6.040865,-3.4203033)(6.0961814,-3.4664793)(6.1756835,-1.2220107)(5.977668,0.49615484)(5.779652,2.2143204)(5.975022,4.452345)(5.984008,4.4535227)(5.9929934,4.4547)(4.9862595,3.8787544)(4.3627124,3.522611)(3.7391648,3.1664674)(2.794311,3.0837445)(2.7617345,3.1007488)(2.729158,3.1177528)(2.9553642,1.4169611)(2.850278,-0.84493685)(2.7451916,-3.106835)(3.1199784,-4.205736)(3.0958858,-4.217002)
\psbezier[linecolor=black, linewidth=\lwd, linestyle=dashed,arrowinset=0.2]{->}(2.286448,1.5453709)(1.5628992,1.536089)(1.1790792,1.464238)(0.8534774,1.4413978)(0.5278757,1.4185575)(0.43931457,1.363265)(0.0026307455,1.305112)
\psbezier[linecolor=black, linewidth=\lwd, linestyle=dashed,arrowinset=0.2]{->}(4.290384,1.4081529)(3.4283798,1.2583376)(2.7330022,1.1858749)(2.1669688,1.0924269)(1.6009356,0.99897885)(0.6006486,0.8018271)(0.1561439,0.7435334)

\psbezier[linecolor=black, linewidth=\lwd](2.3369355,2.919008)(2.2541437,2.4751565)(2.2474468,2.313473)(2.2769356,1.5895457)(2.3064244,0.8656185)(2.4541438,-1.077675)(2.4369354,-1.5125332)(2.4197273,-1.9473913)(2.4064243,-2.3094351)(2.3236022,-2.512103)
\psbezier[linecolor=black, linewidth=\lwd](4.2769356,2.899008)(4.3036566,2.6127858)(4.3835444,2.1383276)(4.3028812,1.4299171)(4.2222185,0.72150666)(4.2166295,-1.1180443)(4.285584,-1.5961698)(4.354539,-2.0742953)(4.421678,-2.3140664)(4.4369354,-2.560992)

\psdots[linecolor=black, dotsize=\dsz](4.2794356,1.414008)
\psdots[linecolor=black, dotsize=\dsz](2.2794354,1.539008)

\uput[0](4.2794356,1.414008){\labsizea$q_j''$}
\uput[45](2.2794354,1.539008){\labsizea$q_1$}
\uput[270](4.4294357,-2.573492){\labsizeb$W^u(q''_j)$}
\uput[260](2.3336022,-2.510992){\labsizeb$W^u(q_1)$}
\uput[270](1.6391578,-4.272103){\labsizec$M_{(\varsigma_1, \xi_1)}$}
\uput[270](3.9181914,-4.2568374){\labsizec$M_{(\varsigma''_j, \xi''_j)}$}
\uput[90](0.0026307455,1.305112){\labsizea$\Phi^{\hat t}$}
\uput[270](0.1561439,0.7435334){\labsizea$\Phi^{t_j}$}

\end{pspicture}
}

\caption{Choices of points in Lemma \ref{lem:key} and proof of Lemma  \ref{lem:here2}}\label{fig:2}
\end{figure}
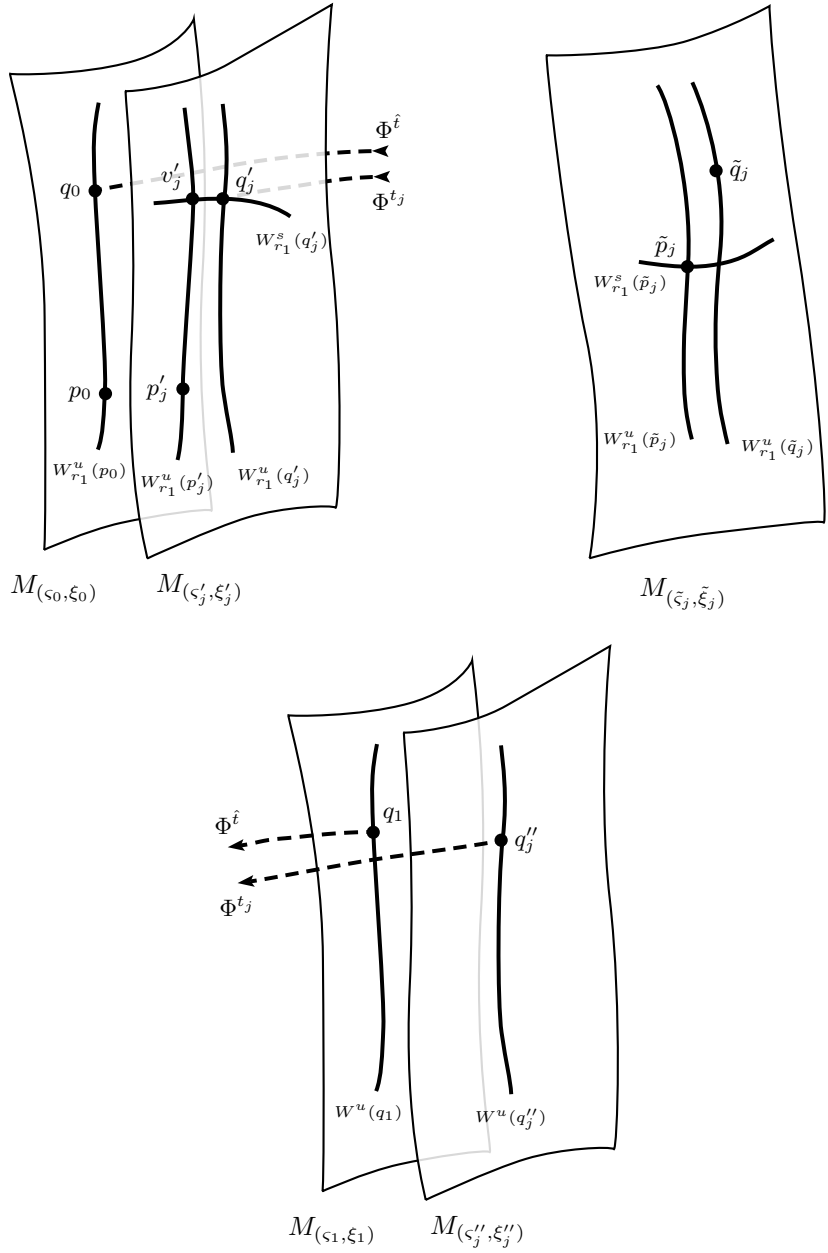